\documentclass[reqno, a4paper,11pt]{amsart}
\usepackage{amssymb}
\usepackage{braket}
\theoremstyle{definition}
\newtheorem{notation}{Notation}[section]
\newtheorem{definition}[notation]{Definition}
\theoremstyle{plain}
\newtheorem{theorem}[notation]{Theorem}
\newtheorem{theorem_i}{Theorem}
\newtheorem{lemma}[notation]{Lemma}
\newtheorem{corollary}[notation]{Corollary}
\newtheorem{proposition}[notation]{Proposition}
\newtheorem{fact}[notation]{Fact}
\newtheorem{claim}{Claim}
\theoremstyle{remark}
\newtheorem{example}[notation]{Example}
\newtheorem{remark}[notation]{Remark}
\newcommand{\dis}{\displaystyle}
\newcommand{\nor}{\normalfont}

\newcommand{\Ker}{\mathop{\mathrm{Ker}}\nolimits}
\newcommand{\Img}{\mathop{\mathrm{Im}}\nolimits}
\newcommand{\Hom}{\mathop{\mathrm{Hom}}\nolimits}
\newcommand{\End}{\mathop{\mathrm{End}}\nolimits}
\newcommand{\spn}{\mathop{\mathrm{span}}\nolimits}
\newcommand{\gra}{\mathop{\mathrm{gr}}\nolimits}
\newcommand{\weight}{\mathop{\mathrm{wt}}\nolimits}
\newcommand{\height}{\mathop{\mathrm{ht}}\nolimits}
\newcommand{\adr}{\mathop{\mathrm{Ad }}\nolimits}
\makeatletter
\def\l@section{\@tocline{1}{0pt}{1pc}{}{}}
\def\l@subsection{\@tocline{2}{0pt}{1pc}{4.6em}{}}
\def\l@subsubsection{\@tocline{3}{0pt}{1pc}{7.6em}{}}
\renewcommand{\tocsection}[3]{%
  \indentlabel{\@ifnotempty{#2}{\makebox[2.3em][l]{%
    \ignorespaces#1 #2.\hfill}}}#3}
\renewcommand{\tocsubsection}[3]{%
  \indentlabel{\@ifnotempty{#2}{\hspace*{2.3em}\makebox[2.3em][l]{%
    \ignorespaces#1 #2.\hfill}}}#3}
\renewcommand{\tocsubsubsection}[3]{%
  \indentlabel{\@ifnotempty{#2}{\hspace*{4.6em}\makebox[3em][l]{%
    \ignorespaces#1 #2.\hfill}}}#3}
\makeatother

\makeatletter
\ifcsname phantomsection\endcsname
    \newcommand*{\qrr@gobblenexttocentry}[5]{}
\else
    \newcommand*{\qrr@gobblenexttocentry}[4]{}
\fi
\newcommand*{\addsubsection}{%
    \addtocontents{toc}{\protect\qrr@gobblenexttocentry}%
    \subsection}
\makeatother
 \makeatletter
    
    \@addtoreset{equation}{section}
  \makeatother
\begin{document}
\title[Representations of quantized function algebras]{Representations of quantized function algebras and the transition matrices from Canonical bases to PBW bases}
\author{Hironori Oya}
\address
{Graduate School of Mathematical Sciences, The University of Tokyo,
 Komaba, Tokyo, 153-8914, Japan}
\email{oya@ms.u-tokyo.ac.jp}
\date{}
\begin{abstract}
Let $G$ be a connected simply-connected simple complex algebraic group and $\mathfrak{g}$ the corresponding simple Lie algebra. In the first half of the present paper, we study the relation between the positive part $U_q(\mathfrak{n^+})$ of the quantized enveloping algebra $U_q(\mathfrak{g})$ and the specific irreducible representations of the quantized function algebra $\mathbb{Q}_q[G]$, taking into account the right $U_q(\mathfrak{g})$-algebra structure of $\mathbb{Q}_q[G]$. This work is motivated by Kuniba, Okado and Yamada's result (\cite{KOY}) together with Tanisaki and Saito's results (\cite{Tanisaki_QFA}, \cite{Saito_QFA}). In the latter half, we calculate the transition matrices from the canonical basis to the PBW bases of $U_q(\mathfrak{n^+})$ using the above relation. Consequently, we show that the constants arising from our calculation are described by the structure constants for the comultiplication of $U_q(\mathfrak{g})$. In particular, when $\mathfrak{g}$ is of type $ADE$, this result implies the positivity of the transition matrices, which was originally proved by Lusztig (\cite{Lus_canari1}) in the case when the PBW bases are associated with the adapted reduced words of the longest element of the Weyl group, and by Kato (\cite{KatoKLR}) in arbitrary cases. In fact, the constants in our calculation coincide with ones arising from the calculation using the bilinear form on $U_q(\mathfrak{n}^{\pm})$. We explain this coincidence in Appendix.
\end{abstract}
\maketitle
\tableofcontents 
\section{Introduction}\label{intro}
\begin{notation}
Denote the set of non-negative integers$(=\mathbb{Z}_{\geqq 0})$ by $\mathbb{N}$.
\end{notation}
Let $G$ be a connected simply-connected simple complex algebraic group and $\mathfrak{g}$ the corresponding simple Lie algebra. Then, we can define two Hopf algebras over the complex number field $\mathbb{C}$, the quantized enveloping algebra $U_q(\mathfrak{g})_{\mathbb{C}}$ and the quantized function algebra $\mathbb{C}_q[G]$. (In this case, $q$ is a complex number which is neither 0 nor a root of unity.) The Hopf algebras $U_q(\mathfrak{g})_{\mathbb{C}}$ and $\mathbb{C}_q[G]$ are $q$-analogues of the universal enveloping algebra $U(\mathfrak{g})$ of $\mathfrak{g}$ and the coordinate algebra $\mathbb{C}[G]$ of $G$ respectively. We can also define these Hopf algebras as the Hopf algebras over the rational function field $\mathbb{Q}(q)$ in the same way, and denote these by $U_q(\mathfrak{g})$ and $\mathbb{Q}_q[G]$. We will deal with $U_q(\mathfrak{g})$ and $\mathbb{Q}_q[G]$.
The algebra $\mathbb{Q}_q[G]$ has the natural left and right $U_q(\mathfrak{g})$-algebra structure. 

In the first half of this paper, we study the relation between the positive part $U_q(\mathfrak{n}^+)$ of $U_q(\mathfrak{g})$ and the specific irreducible representations of $\mathbb{Q}_q[G]$, which has been pointed out by Kuniba, Okado and Yamada \cite{KOY}. To explain this relation, we recall the ``tensor product construction'' of the irreducible representations of $\mathbb{Q}_q[G]$ due to Soibelman \cite{Soirep}. 

Let $I$ be the index set of simple roots of $\mathfrak{g}$ and $W$ the Weyl group of $\mathfrak{g}$. It is known that, for each $i\in I$, there exists an infinite dimensional irreducible representation $\pi_i$ of $\mathbb{Q}_q[G]$, whose representation space $V_{s_i}$ has a natural basis ${\big \{ }|m\rangle_i{\big \}}_{m\in \mathbb{Z}_{\geqq 0}}$ indexed by the non-negative integers (\cite{VakSoi}, \cite{Soirep}). Take an element $w\in W$ and fix its reduced expression $w=s_{i_1}\cdots s_{i_l}$. 
\begin{fact}[{\cite[5.4]{Soirep}}]
The representation $\pi_w:=\pi_{i_1}\otimes\cdots\otimes\pi_{i_l}$ of $\mathbb{Q}_q[G]$ is irreducible and its isomorphism class does not depend on the choice of the reduced expressions of $w$. 
\end{fact}
See for instance \cite[Chapter 3]{KroSoi}, \cite[Chapter 10]{Josbook} for more details.

We explain Kuniba, Okado and Yamada's result \cite{KOY}.　Let $w_0$ be the longest element of $W$ and $N$ its length.　It follows from the construction that the representation space $V_{w_0}$ of $\pi_{w_0}$ has a basis ${\big \{ }|m_1\rangle_{i_1}\otimes\cdots\otimes|m_l\rangle_{i_N}(=:|({\bf m})\rangle_{{\bf i}}){\big \}}_{{\bf m}\in(\mathbb{Z}_{\geqq 0})^N}$ for each reduced word ${\bf i}=(i_1,\dots,i_N)$ of $w_0$ (which means that $w_0=s_{i_1}\cdots s_{i_N}$ is a reduced expression of $w_0$). An intertwiner $\Theta_{{\bf j}, {\bf i}}$ from $\pi_{i_1}\otimes\cdots\otimes\pi_{i_N}$ to $\pi_{j_1}\otimes\cdots\otimes\pi_{j_N}$ (${\bf i}, {\bf j}$ are reduced words of $w_0$) is given by
$$\dis |( 0 )\rangle_{{\bf i}}\mapsto |( 0 )\rangle_{{\bf j}}.$$
We regard the map $\Theta_{{\bf j}, {\bf i}}$ as the identity map on $V_{w_0}$.

On the other hand, we have a basis $\{E_{{\bf i}}^{{\bf c}}\}_{{\bf c}\in (\mathbb{Z}_{\geqq 0})^N}$, called the PBW basis, of $U_q(\mathfrak{n}^+)$ for each reduced word ${\bf i}$ of $w_0$. See Definition \ref{E_i^c} for the precise definition. Set the $\mathbb{Q}(q)$-linear isomorphism $\Phi_{{\bf i}}:U_q(\mathfrak{n}^+)\to V_{w_0}$ by 
$$E_{{\bf i}}^{{\bf c}}\mapsto |({\bf c})\rangle_{{\bf i}}\ ({\bf c}\in (\mathbb{Z}_{\geqq 0})^N).$$
Then, Kuniba, Okado and Yamada's result is the following:
\begin{fact}[{\cite[Theorem 5]{KOY}}]For any reduced words ${\bf i}, {\bf j}$ of $w_0$, we have
$$\Phi_{{\bf j}}=\Theta_{{\bf j}, {\bf i}}\circ\Phi_{{\bf i}}.$$
\end{fact}
This fact says that the definition of the map $\Phi_{{\bf i}}$ does not depend on the choice of ${\bf i}$, so we denote this map by $\Phi_{\mathrm{KOY}}$. 

We also use the conjecture \cite[Conjecture 1]{KOY}, recently proved by Saito \cite{Saito_QFA} and Tanisaki \cite{Tanisaki_QFA}. Let $c_{f_{\lambda}, v_{w_0\lambda}}^{\lambda}$ denote the element of $U_q(\mathfrak{g})^{\ast}$ determined by 
$$ x\mapsto\langle f_{\lambda}, x.v_{w_0\lambda}\rangle,$$
where $v_{w_0\lambda}$ is a lowest weight vector of the integrable highest weight $U_q(\mathfrak{g})$-module $V(\lambda)$ with highest weight $\lambda$, and $f_{\lambda}$ is a highest weight vector of the right $U_q(\mathfrak{g})$-module $V(\lambda)^{\ast}$. See Notation \ref{normalization} for their normalization. We set $\mathcal{S}:={\big\{ }c_{f_{\lambda}, v_{w_0\lambda}}^{\lambda}{\big  \}}_{\lambda\in P_+}$, where $P_+$ is the set of the dominant weights. Then, $\mathcal{S}$ is a (left and right) Ore multiplicative set in $\mathbb{Q}_q[G]$. We consider the quotient ring $\mathbb{Q}_q[G]_{\mathcal{S}}$ and set 
$$\xi_i:=(1-q_i^2)^{-1}(c_{f_{\varpi_i}, v_{w_0\varpi_i}}^{\varpi_i}.E_i)({c_{f_{\varpi_i}, v_{w_0\varpi_i}}^{\varpi_i}})^{-1}\in \mathbb{Q}_q[G]_{\mathcal{S}},$$
where $\varpi_i$ is the fundamental weight and $E_i$ is the positive Chevalley generator associated with $i\in I$.
Then, the action of $\mathbb{Q}_q[G]$ on $V_{w_0}$ can be extended to the action of $\mathbb{Q}_q[G]_{\mathcal{S}}$, and 
\begin{fact}[{\cite[Proposition 7.6]{Tanisaki_QFA}, \cite[Corollary 4.3.3]{Saito_QFA}}]\label{conjtheorem}
$$\Phi_{\mathrm{KOY}}(E_ix)=\xi_i.\Phi_{\mathrm{KOY}}(x),$$ for all $x\in U_q(\mathfrak{n}^+)$, where  $\xi_i.$ is the action of $\xi_i$ on $V_{w_0}$.
\end{fact}

We calculate the action of the elements of $\mathbb{Q}_q[G]_{\mathcal{S}}$ on $V_{w_0}$ other than $\xi_i$'s. More precisely, taking into account the right $U_q(\mathfrak{g})$-action, we investigate the action of the elements of $\mathbb{Q}_q[G/N^-]$, which is a right $U_q(\mathfrak{g})$-subalgebra of $\mathbb{Q}_q[G]$ defined as the set of invariants of $\mathbb{Q}_q[G]$ with respect to the left action of $U_q(\mathfrak{n}^-)$. We mainly use the method appearing in the reference \cite[Chapter 9]{Josbook}. 

The set $\mathcal{S}$ is also a (left and right) Ore multiplicative set in $\mathbb{Q}_q[G/N^-]$. The quotient ring is denoted by $\mathbb{Q}_q[G/N^-]_{\mathcal{S}}$. We construct a right $U_q(\mathfrak{g})$-algebra $C$ with the following properties:
\begin{itemize}
\item The right $U_q(\mathfrak{g})$-algebra $C$ is isomorphic to $\mathbb{Q}_q[G/N^-]$ (the multiplicative set in $C$ corresponding to $\mathcal{S}$ is denoted by $\mathcal{S}'$).
\item As a $\mathbb{Q}(q)$-algebra, the quotient ring $C_{\mathcal{S}'}$ is naturally regarded as the non-negative(=Borel) part $\check{U}^{\geqq 0}$ of (a variant of) the quantized enveloping algebra whose basis $\{K_{\lambda}\}_{\lambda\in P}$ of its Cartan part is indexed by the elements of the weight lattice $P$. 
\end{itemize}

The right $U_q(\mathfrak{g})$-algebra $C$ can be regarded as a subalgebra of $\check{U}^{\geqq 0}$ with a certain right $U_q(\mathfrak{g})$-algebra structure. Let $\{G_i\}_{i\in I}$ be the set of elements in $C_{\mathcal{S}'}(=\check{U}^{\geqq 0})$ corresponding to the positive Chevalley generators. Then, we have the following theorem:
\begin{theorem_i}\label{main1}
Let $\Upsilon$ be the isomorphism $\check{U}^{\geqq 0}\to\mathbb{Q}_q[G/N^-]_{\mathcal{S}}.$
$($Note that $\left.\Upsilon\right|_{C}:C\to \mathbb{Q}_q[G/N^-].)$ Then, 
$$\Upsilon (K_{\lambda})=c_{f_{\lambda}, v_{w_0\lambda}}^{\lambda}\ \text{and\ }\Upsilon(G_i)=\xi_i\  (\lambda\in P_+,\ i\in I).$$
\end{theorem_i}
Theorem \ref{main1} will be stated as Theorem \ref{maintool} below. Combining this theorem with Fact \ref{conjtheorem}, we can compute the action of the element of the form $c_{f_{\lambda}, v_{w_0\lambda}}^{\lambda}.x $ ($.x$ means the right action of $x\in U_q(\mathfrak{g})$) on $V_{w_0}$ by computing the image of $K_{\lambda}$ under the right action of $x$ in $C$.

In the latter half of this paper, we also deal with the canonical basis of $U_q(\mathfrak{n^+})$, which is defined by Lusztig \cite{Lus_canari1} and subsequently by Kashiwara \cite{Kas_originalcrys} under the different methods. The canonical basis of $U_q(\mathfrak{n}^+)$ is a basis of $U_q(\mathfrak{n^+})$ which is different from the PBW bases and does not depend on the choice of the reduced expressions of $w_0$. The elements of this basis have many nice properties, but it is difficult to calculate the explicit forms of them unlike the PBW bases, in general. 

For an element $G^{+}$ of the canonical basis of $U_q(\mathfrak{n}^+)$, we can write 
$$G^{+}=\sum_{{\bf c}}{_{{\bf i}}\zeta}_{{\bf c}}^{G^{+}}E_{{\bf i}}^{{\bf c}}\ \text{with\ } {_{{\bf i}}\zeta}_{{\bf c}}^{G^{+}}\in \mathbb{Z}[q^{\pm 1}].$$
Our aim is to investigate the coefficients ${_{{\bf i}}\zeta}_{{\bf c}}^{G^{+}}$ by using the representation $V_{w_0}$ of $\mathbb{Q}_q[G]$. 

Let us explain our method briefly. Fix an arbitrary reduced word ${\bf i}$ of $w_0$. Take the sufficiently ``large'' dominant integral weight $\lambda_0$ associated with $G^{+}$ and set $c_{G^{+}}:=c_{f_{\lambda_0}, v_{w_0\lambda_0}}^{\lambda_0}.\ast (G^{+})$, where $\ast$ is a certain $\mathbb{Q}(q)$-algebra anti-involution of $U_q(\mathfrak{g})$. (Definition \ref{automorphisms}). By Theorem \ref{main1} (and Fact \ref{conjtheorem}), we can show that the element $\Phi_{\mathrm{KOY}}^{-1}(c_{G^{+}}.|( 0 )\rangle_{{\bf i}})$ of $U_q(\mathfrak{n}^+)$ is equal to $G^{+}$ modulo the terms of the form (the sufficiently large powers of $q$)$\cdot E_{{\bf i}}^{{\bf c}}$ when we express this element using the PBW basis. 

On the other hand, it is possible in principle to compute the element $c_{G^{+}}.|(0)\rangle_{\bf i}$ according to the definition of the tensor product module $V_{w_0}$. In fact, it is difficult to obtain the explicit result, but we can compute it modulo the terms of the form (the sufficiently large powers of $q$)$\cdot|({\bf c})\rangle_{{\bf i}}$. Taking the identification via $\Phi_{\mathrm{KOY}}$ into consideration, we obtain the desired coefficients ${_{{\bf i}}\zeta}_{{\bf c}}^{G^{+}}$ and show that they are described by the structure constants for the comultiplication of $U_q(\mathfrak{g})$. In particular, when the Lie algebra $\mathfrak{g}$ is of type $ADE$, these structure constants are the elements of $\mathbb{N}[q^{\pm 1}]$ by Lusztig's theorem (\cite[Theorem 11.5]{Lus_quiper}). Using this fact, we obtain the following theorem:
\begin{theorem_i}\label{main2}
Assume that the Lie algebra $\mathfrak{g}$ is of type $ADE$. Take an arbitrary reduced word ${\bf i}$ of $w_0$ and an element $G^{+}$ of the canonical basis of $U_q(\mathfrak{n}^+)$. Then, we have
$$G^{+}=\sum_{{\bf c}}{_{{\bf i}}\zeta}_{{\bf c}}^{G^{+}}E_{{\bf i}}^{{\bf c}}\ \text{with\ } {_{{\bf i}}\zeta}_{{\bf c}}^{G^{+}}\in \mathbb{N}[q^{\pm 1}].$$
\end{theorem_i}
Theorem \ref{main2} will be stated as Theorem \ref{maintheorem} below. We remark that Theorem \ref{maintheorem} is actually the negative counterpart of Theorem \ref{main2}. However, the statement in the negative side can be immediately translated into that for the positive side and vice versa. 

In fact, this theorem itself is not a new result. It was originally proved by Lusztig in his original paper of the canonical bases (\cite[Corollary 10.7]{Lus_canari1}) in the case when PBW bases are associated with the ``adapted'' reduced words of $w_0$ (See \cite[4.7]{Lus_canari1}.), through his geometric realization of the elements of the canonical bases and PBW bases. Recently, this fact for arbitrary PBW bases was proved by Kato (\cite[Theorem 4.17]{KatoKLR}), through the categorification of PBW bases by using the Khovanov-Lauda-Rouquier algebras. Beyond type $ADE$, by the way, McNamara also established the categorification of PBW bases via the Khovanov-Lauda-Rouquier algebras for arbitrary finite types (\cite[Theorem 3.1]{McNamara_finite}, the dual PBW bases) and symmetric affine types (\cite[Theorem 24.4]{McNamara_symmaffine}). Hence, in fact, such positivity also holds for symmetric affine types (\cite[Theorem 24.10]{McNamara_symmaffine}) though we do not deal with them in this paper. (For nonsymmetric finite types, the ``canonical basis'' arising from the Khovanov-Lauda-Rouquier categorification does not coincide with the above-mentioned canonical basis and the positivity theorem fails in general.)

When we perform our calculation, we heavily use the properties of the canonical basis of $U_q(\mathfrak{n}^{\pm})$, the canonical and dual canonical bases of the highest weight integrable modules of $U_q(\mathfrak{g})$, since the coproduct of $c_{G^{+}}$ is described by these objects. One of the important properties, ``Similarity of the structure constants'', will be proved in Section \ref{dab2}. (Proposition \ref{similarity}) 

After this paper was submitted to the preprint server, Yoshiyuki Kimura pointed out to the author the existence of a much simpler proof of the positivity of the transition matrices from canonical bases to PBW bases. We explain the proof of this which is obtained from his comments in the first half
of Appendix. Moreover, this method provides, in fact, the same constants as in Section \ref{main} even when $\mathfrak{g}$ is of nonsymmetric (finite) type. We check this point in the latter half. Since the calculation in Section \ref{main} is slightly complicated, it might be helpful to read Subsection \ref{app1} before reading Section \ref{main}. At last, we also state the corollaries of this comparison. 
\addsubsection*{Acknowledgements}
The author is greatly indebted to Yoshiyuki Kimura, Yoshihisa Saito and Toshiyuki Tanisaki for many helpful discussions and comments. The appendix in this paper is attached essentially motivated by Kimura's comments. I would also like to thank Ryo Sato for many helpful discussions. This work was supported by the Program for Leading Graduate Schools, MEXT, Japan.
\section{Preliminaries I : the quantized enveloping algebras and the quantized function algebras}\label{dab1}
\subsection{Definitions of $U_q(\mathfrak{g})$ and $\mathbb{Q}_q[G]$}
Firstly, we review the definition of the quantized enveloping algebras and the quantized function algebras. 

\begin{notation}\label{notation}
Let $\mathfrak{g}$ be a finite dimensional complex simple Lie algebra, $\Delta $ the root system of $\mathfrak{g}$ with respect to a fixed Cartan subalgebra $\mathfrak{h}$, $Q(\subset \mathfrak{h}^{\ast})$ the root lattice, $Q^{\vee}(\subset \mathfrak{h})$ the coroot lattice. Fix a set of simple roots $\Pi = \{ \alpha_i \}_{i \in I} (\subset \Delta \subset \mathfrak{h}^{\ast})$ and denote the set of simple coroots by $\Pi^{\vee}:= \{ \alpha_i^{\vee} \}_{i \in I} (\subset \mathfrak{h})$. $(\langle \alpha_i^{\vee}, \alpha_j \rangle )_{i,j \in I}$ is called the Cartan matrix of $\mathfrak{g}$. We say that the Lie algebra $\mathfrak{g}$ is of type $ADE$ if the Cartan matrix of $\mathfrak{g}$ is symmetric. Let $W$ be the Weyl group, $e$ the unit of $W$, $s_i(\in W)$ the simple reflection corresponding to $\alpha_i$. (i.e. $s_i(\lambda)=\lambda-\langle \lambda, \alpha_i^{\vee}\rangle \alpha_i$ for $\lambda \in \mathfrak{h}^{\ast}$.), $w_0$ the longest element of $W$ and $l(w)$ the length of an element $w$ of $W$. Define the symmetric bilinear form $(\ ,\ )$ on $\mathfrak{h}^{\ast}$ by $2(\alpha_i, \alpha_j)/(\alpha_i, \alpha_i)=\langle \alpha_i^{\vee}, \alpha_j \rangle$ and $(\alpha_i, \alpha_i)=2$ for all short simple root $\alpha_i$. Define $Q_+:= \sum_{i\in I}\mathbb{Z}_{\geqq 0}\alpha_i$, $Q_-:=-Q_+$, $\dis \Delta_+ := \Delta\cap Q_+$ (the set of positive roots), $P:= \Set{\lambda \in \mathfrak{h}^{\ast}| \langle \lambda , \alpha_i^{\vee} \rangle \in \mathbb{Z} \text{\footnotesize\ for\ all}\ i \in I }$ (the integral weight lattice), and $P_+:= \Set{\lambda \in \mathfrak{h}^{\ast}| \langle \lambda , \alpha_i^{\vee} \rangle \in \mathbb{Z}_{\geqq 0} \text{\footnotesize\ for\ all}\ i \in I }$ (the set of dominant integral weights). $\varpi_i$ denotes the fundamental weight such that $\langle \varpi_i, \alpha_j^{\vee} \rangle=\delta_{i, j}$. We define the partial order $\leq$ on $Q$ by $\alpha \leq \beta \Leftrightarrow \beta -\alpha \in Q_+$. For an element $\alpha:=\sum_{i\in I}m_i\alpha_i\in Q$ $(m_i\in \mathbb{Z})$, we set $\height \alpha:=\sum_{i\in I}m_i$ (called the height of $\alpha$).

We set
\[
\begin{array}{l}
q_i:= q^{\frac{(\alpha_i, \alpha_i)}{2}}, \\
\dis [n]:= \frac{q^n-q^{-n}}{q-q^{-1}}\ \text{for\ }n\in \mathbb{Z},\\
\dis \left[ \begin{array}{c} n\\k \end{array} \right]:=\begin{cases}\dis \frac{[n][n-1]\cdots[n-k+1]}{[k][k-1]\cdots [1]}&\text{if\ }n\in \mathbb{Z}, k\in \mathbb{Z}_{>0},\\ 1 &\text{if\ }n\in \mathbb{Z}, k=0,\end{cases}\\
\dis [n]!:=[n][n-1]\cdots[1] \text{\ for\ } n\in \mathbb{Z}_{>0}, [0]!:=1.\\
\end{array}
\]
Note that $[n], \left[ \begin{array}{c} n\\k \end{array} \right]\in \mathbb{Z}[q^{\pm 1}]$ and $\dis \left[ \begin{array}{c} n\\k \end{array} \right]=\dis \frac{[n]!}{[k]![n-k]!}\ \text{if\ } n\geqq k\geqq 0$. For a rational function $X\in \mathbb{Q}(q)$, we define $X_i$ by the rational function obtained from $X$ by substituting $q$ by $q_i$ ($i\in I$).
\end{notation}
\begin{definition}\label{QEA}
The quantized enveloping algebra $U_q(\mathfrak{g})$ is the unital associative $\mathbb{Q}(q)$-algebra defined by the generators
\begin{center}
$E_i$, $F_i$ ($i\in I$), $K_h$ ($h \in Q^{\vee}$),
\end{center}
and the relations (i)-(vi) below.
\[
\begin{array}{ccl}
\text{(i)}&K_0=1,\ K_hK_{h'}=K_{h+h'} & \text{for\ all\ } h, h' \in Q^{\vee},\\
\text{\nor (ii)}&K_{h}E_i=q^{\langle h, \alpha_i \rangle}E_iK_h & \text{for\ all\ } h \in Q^{\vee}, i \in I,\\
\text{(iii)}&K_hF_i=q^{-\langle h, \alpha_i \rangle}F_iK_h & \text{for\ all\ } h \in Q^{\vee}, i \in I,\\
\text{(iv)}&\dis [E_i, F_j]=\delta_{ij}\frac{K_i-K_{-i}}{q_i-q_i^{-1}} & \text{for\ all\ } i, j \in I,\\
\multicolumn{3}{c}{\text{where\ }K_i:= K_{\frac{(\alpha_i, \alpha_i)}{2}\alpha_i^{\vee}}\ \text{and\ } K_{-i}:= K_{-\frac{(\alpha_i, \alpha_i)}{2}\alpha_i^{\vee}},}\\
\end{array}
\]
\[
\begin{array}{ccl}
\text{(v)}&\dis \sum_{k=0}^{1-\langle \alpha_i^{\vee}, \alpha_j \rangle}(-1)^kE_i^{(k)}E_jE_i^{(1-\langle \alpha_i^{\vee}, \alpha_j \rangle-k)}=0 &\text{for\ all\ } i, j \in I \ \text{with\ } i\neq j,\\
\text{(vi)}&\dis \sum_{k=0}^{1-\langle \alpha_i^{\vee}, \alpha_j \rangle}(-1)^kF_i^{(k)}F_jF_i^{(1-\langle \alpha_i^{\vee}, \alpha_j \rangle-k)}=0 &\text{for\ all\ } i, j \in I \ \text{with\ } i\neq j\\
\multicolumn{3}{c}{\dis \text{where\ }X_i^{(n)}:= \frac{X_i^n}{[n]_i!}.}
\end{array} 
\]

For $\dis \alpha=\sum_{i\in I}m_i\alpha_i \in Q$ ($m_i \in \mathbb{Z}$), we set $K_{\alpha}:= K_{\sum_{i\in I}\frac{m_i(\alpha_i, \alpha_i)}{2}\alpha_i^{\vee}}$. Note that $K_{\pm\alpha_i}=K_{\pm i}$.

The subalgebra of $U_q(\mathfrak{g})$ generated by $\{ E_i \}_{i\in I}$ (resp.~$\{ F_i \}_{i\in I}$,  $\{ K_{\alpha_i^{\vee}} \}_{i\in I}$
) is denoted by $U_q(\mathfrak{n}^+)$ (resp.~$U_q(\mathfrak{n}^-)$,  $U^0$
). 

For $\alpha \in Q$, we set
\[
\begin{array}{l}
U_q(\mathfrak{g})_{\alpha} := \Set{u \in U_q(\mathfrak{g})| K_huK_{-h}=q^{\langle\alpha, h \rangle}u \text{\footnotesize\ for\ all\ } h \in Q^{\vee}}.
\end{array}
\]
The elements of $U_q(\mathfrak{g})_{\alpha}$ are called homogeneous and said to have weights $\alpha$. For a homogeneous element $u\in U_q(\mathfrak{g})_{\alpha}$, we set $\weight u=\alpha$. $U_q(\mathfrak{n}^+)_{\alpha}$, $U_q(\mathfrak{n}^-)_{\alpha}
$ are defined similarly.

There exists a $\mathbb{Q}(q)$-linear isomorphism $U_q(\mathfrak{n}^-)\otimes U^0\otimes U_q(\mathfrak{n}^+)\to U_q(\mathfrak{g})$, $a\otimes b\otimes c\mapsto abc$. This is called the triangular decomposition of $U_q(\mathfrak{g})$. 

The algebra $U_q(\mathfrak{g})$ has a Hopf algebra structure given by the following comultiplication $\Delta$, counit $\varepsilon$ and antipode $S$:
\[
\begin{array}{ccc}
\Delta (E_i)=E_i\otimes 1+K_i\otimes E_i,& \varepsilon (E_i)=0,& S(E_i)=-K_{-i}E_i,\\
\Delta (F_i)=F_i\otimes K_{-i}+1\otimes F_i,& \varepsilon (F_i)=0,& S(F_i)=-F_iK_i,\\
\Delta (K_{\alpha_i^{\vee}})=K_{\alpha_i^{\vee}}\otimes K_{\alpha_i^{\vee}},& \varepsilon (K_{\alpha_i^{\vee}})=1,& S(K_{\alpha_i^{\vee}})=K_{-\alpha_i^{\vee}},
\end{array}
\]
for all $i \in I$.

\end{definition}
\begin{remark}
This $U_q(\mathfrak{g})$ is called of simply-connected type in \cite{Lusbook}.
\end{remark}
\begin{definition}\label{automorphisms}
We define the $\mathbb{Q}(q)$-algebra, anti-coalgebra involution $\omega : U_q(\mathfrak{g})\to U_q(\mathfrak{g})$ by
\[
\begin{array}{lll}
\omega(E_i)=F_i,& \omega(F_i)=E_i,& \omega(K_{\alpha_i^{\vee}})=K_{-\alpha_i^{\vee}},
\end{array}
\]
for $i\in I$.
We define the $\mathbb{Q}(q)$-algebra anti-involutions $\ast, \varphi, \omega' : U_q(\mathfrak{g})\to U_q(\mathfrak{g})$ by
\[
\begin{array}{lll}
\ast(E_i)=E_i,& \ast(F_i)=F_i, & \ast(K_{\alpha_i^{\vee}})=K_{-\alpha_i^{\vee}},\\
\varphi(E_i)=F_i,& \varphi(F_i)=E_i, & \varphi(K_{\alpha_i^{\vee}})=K_{\alpha_i^{\vee}},\\
\omega'(E_i)=K_iF_i,& \omega'(F_i)=E_iK_{-i}, & \omega'(K_{\alpha_i^{\vee}})=K_{\alpha_i^{\vee}},
\end{array}
\]
for $i\in I$. Note that $\omega'$ is a $\mathbb{Q}(q)$-coalgebra involution and $\omega=\ast\circ \varphi=\varphi\circ \ast$.

We define the $\mathbb{Q}$-algebra involution $\overline{(\cdot)} : U_q(\mathfrak{g})\to U_q(\mathfrak{g})$ by
\[
\begin{array}{llll}
\overline{E_i}=E_i,& \overline{F_i}=F_i,& \overline{K_{\alpha_i^{\vee}}}=K_{-\alpha_i^{\vee}},& \overline{q}=q^{-1},
\end{array}
\]
for $i\in I$.
\end{definition}
\begin{definition}\label{qderiv}
For $i\in I$, define the $\mathbb{Q}(q)$-linear maps $_ir$, $r_i:U_q(\mathfrak{n}^+)\to U_q(\mathfrak{n}^+)$ by
\[
\begin{array}{l}
_ir(xy)={_ir}(x)y+q_i^{\langle \weight x , \alpha_i^{\vee} \rangle}x {_ir}(y)\ \text{and}\ _ir(E_j)=\delta_{ij},\\
r_i(xy)=q_i^{\langle \weight y , \alpha_i^{\vee} \rangle}r_i(x)y+xr_i(y)\ \text{and}\ r_i(E_j)=\delta_{ij},
\end{array}
\]
for $j \in I$ and homogeneous elements $x, y \in U_q(\mathfrak{n^+})$.

Moreover, we set $_ie':=\left.\overline{(\cdot)}\circ \omega \circ r_i \circ \omega \circ \overline{(\cdot)}\right|_{U_q(\mathfrak{n}^-)}$, $e'_i:=\left.\overline{(\cdot)}\circ \omega \circ {_ir} \circ \omega \circ \overline{(\cdot)}\right|_{U_q(\mathfrak{n}^-)}: U_q(\mathfrak{n}^-)\to U_q(\mathfrak{n}^-)$.

We have $\left.\ast\circ {_ir}\circ\ast\right|_{U_q(\mathfrak{n}^+)}=r_i$ and $\left.\ast\circ {_ie'}\circ\ast\right|_{U_q(\mathfrak{n}^-)}=e'_i$ for all $i\in I$. 

For any homogeneous element $x \in U_q(\mathfrak{n}^+), y\in U_q(\mathfrak{n}^-)$ and $p\in \mathbb{Z}_{\geqq 0}$, we have 
\begin{align*}
\Delta(x)&=E_i^{(p)}K_{x - p\alpha_i}\otimes q_i^{-\frac{1}{2}p(p-1)}( _ir)^p(x)+\sum_{\substack{x'\in U_q(\mathfrak{n}^+):\text{homogeneous},\\ \weight x'\neq p\alpha_i}}x'K_{\weight x''}\otimes x'',\\
\Delta(x)&=q_i^{-\frac{1}{2}p(p-1)}(r_i)^p(x)K_{p\alpha_i}\otimes E_i^{(p)}+\sum_{\substack{x''\in U_q(\mathfrak{n}^+):\text{homogeneous},\\ \weight x''\neq p\alpha_i}}x'K_{\weight x''}\otimes x'',\\
\Delta(y)&=F_i^{(p)}\otimes q_i^{\frac{1}{2}p(p-1)}K_{-p\alpha_i}(e'_i)^p(y)+\sum_{\substack{y'\in U_q(\mathfrak{n}^-):\text{homogeneous},\\ \weight y'\neq -p\alpha_i}}y'\otimes K_{\weight y'}y'',\\
\Delta(y)&=q_i^{\frac{1}{2}p(p-1)}(_ie')^p(y)\otimes K_{\weight y+p\alpha_i}F_i^{(p)}+\sum_{\substack{y''\in U_q(\mathfrak{n}^-):\text{homogeneous},\\ \weight y''\neq -p\alpha_i}}y'\otimes K_{\weight y'}y''.\\
\end{align*}

For a homogeneous element $x\in U_q(\mathfrak{n}^+)$ (resp.~$U_q(\mathfrak{n}^-)$) with $\weight x\neq 0$, we have 
\begin{itemize}
\item[(1)] if $_ir(x)=0$ (resp.~$e'_i(x)=0$) for all $i\in I$, then $x=0$, and
\item[(2)] if $r_i(x)=0$ (resp.~$_ie'(x)=0$) for all $i\in I$, then $x=0$.
\end{itemize}

For a homogeneous element $x\in U_q(\mathfrak{n}^+)$ (resp.~$U_q(\mathfrak{n}^-)$), we have

\begin{align}
r_i(x)&=q_i^{\langle \mathrm{wt}x-\alpha_i, \alpha_i^{\vee} \rangle}\overline{_ir(\overline{x})} \label{qderiv1}\\
(\text{resp.~} e'_i(x)&=q_i^{\langle \mathrm{wt}x+\alpha_i, \alpha_i^{\vee} \rangle}\overline{_i{e'}(\overline{x})}) \label{qderiv2}
\end{align}
For $x\in U_q(\mathfrak{n}^+)$ and $y\in U_q(\mathfrak{n}^-)$, we have
\begin{itemize}
\item $\dis F_ix-xF_i=\frac{K_{-i}{_ir}(x)-r_i(x)K_i}{q_i-q_i^{-1}}$ in $U_q(\mathfrak{g})$, and
\item $\dis E_iy-yE_i=\frac{ _ie'(y)K_i-K_{-i}e'_i(y)}{q_i-q_i^{-1}}$ in $U_q(\mathfrak{g})$.
\end{itemize}
See \cite[Chapter 1, 3.1.5]{Lusbook} for details.
\end{definition}
\begin{definition}
Let $M$ be a left (resp.~right) $U_q(\mathfrak{g})$-module. For any $\lambda \in P$, we set
\begin{center}
$M_{\lambda}:= \Set{m \in M | K_h.m=q^{\langle h, \lambda \rangle}m\ (\text{\footnotesize resp.~ }m.K_h=q^{\langle h, \lambda \rangle}m) \text{\footnotesize\ for\ all\ } h \in Q^{\vee}}$.
\end{center}
We say that $M$ is integrable if $M$ satisfies 

\begin{itemize}
\item $\dis M=\bigoplus_{\lambda \in P}M_{\lambda}$,
\item The action of $E_i$ and $F_i$ are locally nilpotent on $M$ for all $i \in I$.
\end{itemize}

Let $\mathcal{O}_{\mathrm{int}}(\mathfrak{g})$ (resp.~$\mathcal{O}_{\mathrm{int}}(\mathfrak{g}^{\mathrm{opp}})$) be the category of finite dimensional integrable left (resp.~right) $U_q(\mathfrak{g})$-modules. The category $\mathcal{O}_{\mathrm{int}}(\mathfrak{g})$ (resp.~ $\mathcal{O}_{\mathrm{int}}(\mathfrak{g}^{\mathrm{opp}})$) is a semisimple category and its irreducible objects are isomorphic to exactly one $V(\lambda)$ (resp.~$V^{r}(\lambda)$) for some $\lambda \in P_+$, where $V(\lambda)$ (resp.~$V^{r}(\lambda)$) is the irreducible integrable left (resp.~right) $U_q(\mathfrak{g})$-module with highest weight $\lambda$. (See. \cite[Chapter 5]{Janbook})
\end{definition}
\begin{remark}
A highest weight vector of a right $U_q(\mathfrak{g})$-module is a vector which vanishes by the action  of $F_i$'s ($i \in I$).

The dual space $V(\lambda)^{\ast}$ of $V(\lambda)$ has a natural right $U_q(\mathfrak{g})$-module structure, and $V(\lambda)^{\ast}$ is isomorphic to $V^{r}(\lambda)$ as a right $U_q(\mathfrak{g})$-module.
\end{remark}
\begin{definition}
The dual space $U_q(\mathfrak{g})^{\ast}$ of $U_q(\mathfrak{g})$ has a natural $U_q(\mathfrak{g})$-bimodule structure. Hence, we can define the subspace $\mathbb{Q}_q[G]$ of $U_q(\mathfrak{g})^{\ast}$ by
\[
\begin{array}{l}
\Set{ f \in U_q(\mathfrak{g})^{\ast} | U_q(\mathfrak{g}).f \text{\footnotesize\ belongs\ to\ } \mathcal{O}_{\mathrm{int}}(\mathfrak{g}) \text{\footnotesize\ and\ } f.U_q(\mathfrak{g})\text{\footnotesize\ belongs\ to\ } \mathcal{O}_{\mathrm{int}}(\mathfrak{g}^{\mathrm{opp}})},
\end{array}
\]
Then, $\mathbb{Q}_q[G]$ has a Hopf algebra structure (\cite[Chapter 7]{Janbook}) induced from one of $U_q(\mathfrak{g})$, and a left and right $U_q(\mathfrak{g})$-algebra structure. Note that $U_q(\mathfrak{g})^{\ast}$ has a natural $\mathbb{Q}(q)$-algebra structure but does not have the bialgebra (hence, the Hopf algebra) structure. This Hopf algebra $\mathbb{Q}_q[G]$ is called the quantized function algebra. By abuse of notation, the coproduct (resp.~the counit, the antipode) of $\mathbb{Q}_q[G]$ is also denoted by $\Delta$ (resp.~$\varepsilon$, $S$).
\end{definition}
The Hopf algebra $\mathbb{Q}_q[G]$ is a quantum analogue of the algebra of regular functions on $G$ where $G$ is the connected simply-connected simple complex algebraic group whose Lie algebra is $\mathfrak{g}$. There is the $q$-analogue of the Peter-Weyl theorem(\cite[Proposition 7.2.2]{Kasglo}):
\begin{proposition}
For $\lambda \in P_+$, we define the $U_q(\mathfrak{g})$-bimodule homomorphism $\Psi_{\lambda}: V(\lambda)^{\ast}\otimes V(\lambda) \to \mathbb{Q}_q[G]$ by $f\otimes v \mapsto (u\mapsto \langle f, u.v \rangle)$.
Then, 
\begin{center}
$\dis \bigoplus_{\lambda \in P_+}\Psi_{\lambda}: \bigoplus_{\lambda \in P_+} V(\lambda)^{\ast}\otimes V(\lambda) \to \mathbb{Q}_q[G]$
\end{center}
is an isomorphism of $U_q(\mathfrak{g})$-bimodules.
\end{proposition}
\begin{notation}\label{normalization}
For $f \in V(\lambda)^{\ast}$ and $v \in V(\lambda)$, we set
\begin{center}
$c_{f, v}^{\lambda}:= \Psi_{\lambda}(f\otimes v)$.
\end{center}
For each $\lambda \in P_+$, we fix a highest weight vector of $V(\lambda)$ (resp.~$V^{\ast}(\lambda)$), denoted by $v_{\lambda}$ (resp.~$f_{\lambda}$). We make an assumption that $\langle f_{\lambda}, v_{\lambda} \rangle=1$. For $w=s_{i_{l(w)}}\cdots s_{i_1} \in W$, we set
\begin{center}
$v_{w\lambda}:= F_{i_{l(w)}}^{(\langle s_{i_{l(w)-1}}\cdots s_{i_1}\lambda, \alpha_{i_{l(w)}}^{\vee} \rangle)}\cdots F_{i_2}^{(\langle s_{i_1}\lambda, \alpha_{i_2}^{\vee} \rangle)}F_{i_1}^{(\langle \lambda, \alpha_{i_1}^{\vee} \rangle)}.v_{\lambda}$,\\
$f_{w\lambda}:= f_{\lambda}. E_{i_1}^{(\langle \lambda, \alpha_{i_1}^{\vee} \rangle)}E_{i_2}^{(\langle s_{i_1}\lambda, \alpha_{i_2}^{\vee} \rangle)}\cdots E_{i_{l(w)}}^{(\langle s_{i_{l(w)-1}}\cdots s_{i_1}\lambda, \alpha_{i_{l(w)}}^{\vee} \rangle)}$.
\end{center}
It is well known that $v_{w\lambda}$ and $f_{w\lambda}$ depend only on the weight $w\lambda$ (i.e. not on the choice of $w$ and its reduced expression. See \cite[Proposition 39.3.7]{Lusbook}). Moreover, $\langle f_{w\lambda}, v_{w\lambda} \rangle=1$.
\end{notation}
\subsection{A review of the representation theory of the quantized function algebras}
Let us recall some basic facts on simple $\mathbb{Q}_q[G]$-modules due to Soibelman et al.
\begin{definition}[{The $\mathbb{Q}_q[G]$-modules $V_{s_i}$}]\label{vsigma}
First, we construct a certain simple $\mathbb{Q}_q[SL_2]$-module. 
The algebra ${\mathbb{Q}_q[{SL}_2]}$ is generated by $c_{ij}:= F^{\delta_{2,j}}.c^{\varpi}_{f_{\varpi}, v_{\varpi}}.E^{\delta_{2,i}}$ ($i, j \in \{1, 2\}$). (Since $I=\{ \ast \}$, $E_i$, $F_i$, $\varpi_i$ ($i\in I$) are simply denoted by $E$, $F$, $\varpi$, respectively.) Moreover, $c_{ij}$'s satisfies the following relations and, in fact, the relations of $c_{ij}$'s are exhausted by them:
\[
\begin{array}{ll}
c_{11}c_{12}=qc_{12}c_{11},& c_{21}c_{22}=qc_{22}c_{21},\\
c_{11}c_{21}=qc_{21}c_{11},& c_{12}c_{22}=qc_{22}c_{12},\\
\left[c_{12},c_{21}\right]=0,& \left[c_{11},c_{22}\right]=(q-q^{-1})c_{12}c_{21},\\
\multicolumn{2}{c}{c_{11}c_{22}-qc_{12}c_{21}=1.}
\end{array} 
\]
Let $\dis V := \bigoplus_{m \in \mathbb{Z}_{\geqq 0}} \mathbb{Q}(q) \ket{m}$ be an infinite dimensional $\mathbb{Q}(q)$-vector space with a basis indexed by non-negative integers. We define a $\mathbb{Q}_q[SL_2]$-module structure on $V$ by
\[
\begin{array}{ll}
c_{11}. :\ket{m} \longmapsto 
\begin{cases}
0 \\
\ket{m-1}\\
\end{cases}
&\begin{tabular}{l}
{\rm if}\ $m=0$, \\
{\rm if}\ $m \in \mathbb{Z}_{>0},$\\
\end{tabular}\\
c_{12}. :\ket{m} \longmapsto q^m \ket{m}&{\rm \ \ for}\ m \in \mathbb{Z}_{\geqq 0},\\
c_{21}. :\ket{m} \longmapsto -q^{m+1} \ket{m}&{\rm \ \ for}\ m \in \mathbb{Z}_{\geqq 0},\\
c_{22}. :\ket{m} \longmapsto (1-q^{2(m+1)})\ket{m+1}&{\rm \ \ for}\ m \in \mathbb{Z}_{\geqq 0}.
\end{array}
\]
By the construction, it is easy to see that this is a simple ${\mathbb{Q}_q[SL_2]}$-module. The algebra homomorphism $\mathbb{Q}_q[SL_2] \to \End_{\mathbb{Q}(q)}(V)$ corresponding to this infinite dimensional $\mathbb{Q}_q[SL_2]$-module is denoted by $\pi$.

Next, we construct simple representations of $\mathbb{Q}_q[G]$ using this representation $\pi$.

For $i \in I$, there exists the injective Hopf algebra homomorphism $\phi_i: U_{q_i}(\mathfrak{sl}_2) \to U_q(\mathfrak{g})$ defined by $E \mapsto E_i$, $F \mapsto F_i$ and $K_{\alpha}(=K_{\alpha^{\vee}})\mapsto K_i$. The image of $\phi_i$ is denoted by $U_{q_i}(\mathfrak{sl}_{2, i})$. Hence, we have the dual surjective Hopf algebra homomorphism $\phi_i^{\ast}: \mathbb{Q}_q[G] \to \mathbb{Q}_{q_i}[SL_2]$, $f\mapsto \left. f\right|_{U_{q_i}(\mathfrak{sl}_{2, i})}\circ\phi_i$.

Therefore, we obtain the representations $\left. \pi \right|_{q=q_i} \circ\phi_i$ of $\mathbb{Q}_q[G]$ ($i\in I$). The simple $\mathbb{Q}_q[G]$-module corresponding to this representation will be denoted by $V_{s_i}$ ($=V$ as a vector space) and its natural basis will be denoted by ${\big \{}\ket{m}_i {\big \}}_{m\in \mathbb{Z}_{\geqq 0}}$. The trivial $\mathbb{Q}_q[G]$-module will be denoted by $V_e$.
\end{definition}
The following theorem is known as the tensor product theorem.
\begin{theorem}[{
\cite[5.4]{Soirep}}]\label{soi}
Let $w \in W$. Then, for any reduced expression $w =s_{i_1}\cdots s_{i_{l(w)}}$, the $\mathbb{Q}_q[G]$-module $V_{s_{i_1}}\otimes \cdots \otimes V_{s_{i_{l(w)}}}$ is simple and its isomorphism class does not depend on the choice of the reduced expressions of $w$. Moreover, for any two reduced expression $w =s_{i_1}\cdots s_{i_{l(w)}}=s_{j_1}\cdots s_{j_{l(w)}}$, we have the $\mathbb{Q}_q[G]$-module isomorphism $\Theta_{{\bf j}, {\bf i}}: V_{s_{i_1}}\otimes \cdots \otimes V_{s_{i_{l(w)}}}\to V_{s_{j_1}}\otimes \cdots \otimes V_{s_{j_{l(w)}}}$ given by $\ket{0}_{i_1}\otimes \cdots \otimes\ket{0}_{i_{l(w)}}\mapsto \ket{0}_{j_1}\otimes \cdots \otimes\ket{0}_{j_{l(w)}}$, where ${\bf i}=(i_1,\dots,i_{l(w)})$ and ${\bf j}=(j_1,\dots ,j_{l(w)})$.

Hence, we denote this module by $V_w$. 
\end{theorem}
The module $V_w$ is a ``highest weight module'' in some sense. Let us explain this point.
\begin{definition}
We define the subspaces $\mathbb{Q}_q[G/N^+]$, $\mathbb{Q}_q[G/N^-]$ of $\mathbb{Q}_q[G]$ by
\begin{align*}
\mathbb{Q}_q[G/N^+]&:= \spn_{\mathbb{Q}(q)}\Set{c_{f, v}^{\lambda}| f \in V(\lambda)^{\ast}, v \in V(\lambda)_{\lambda}},\\
\mathbb{Q}_q[G/N^-]&:= \spn_{\mathbb{Q}(q)}\Set{c_{f, v}^{\lambda}| f \in V(\lambda)^{\ast}, v \in V(\lambda)_{w_0\lambda}}.
\end{align*}

Then, $\mathbb{Q}_q[G/N^+]$ and $\mathbb{Q}_q[G/N^-]$ are the right $U_q(\mathfrak{g})$-subalgebras of $\mathbb{Q}_q[G]$. 
\end{definition}
The following lemma is the ``triangular decomposition'' of $\mathbb{Q}_q[G]$. (\cite[Chapter3]{KroSoi})
\begin{lemma}\label{decomposition}
The multiplication $\mathbb{Q}_q[G/N^-] \otimes_{\mathbb{Q}(q)} \mathbb{Q}_q[G/N^+]\to \mathbb{Q}_q[G]$, $c\otimes c' \mapsto cc'$ is surjective.
\end{lemma}
We have the following proposition. (\cite[Proposition 10.1.1, Lemma 10.1.7, Theorem 10.1.18]{Josbook})
\begin{proposition}\label{hw}
Let $w \in W$ and $w =s_{i_1}\cdots s_{i_{l(w)}}$ be its reduced expression.
Set $\ket{(0)}_w:= \ket{0}_{i_1} \otimes \cdots \otimes \ket{0}_{i_{l(w)}} \in V_{s_{i_1}}\otimes \cdots \otimes V_{s_{i_{l(w)}}}$.

Then, 
\begin{itemize}
\item[{\nor (i)}] $\mathbb{Q}_q[G/N^+].\ket{(0)}_w=\mathbb{Q}(q)\ket{(0)}_w$, and
\item[{\nor (ii)}] if $c_{f, v_{\lambda}}^{\lambda}.\ket{(0)}_w\neq 0$ for some weight vector $f \in V(\lambda)^{\ast}$, then $\weight f=w \lambda$.
\end{itemize}
Moreover,
\begin{itemize}
\item[{\nor (I)}] if $c_{f, v_{\lambda}}^{\lambda}.V_w\neq 0$ for some $\lambda \in P_+$ and a weight vector $f \in V(\lambda)^{\ast}$, then $\weight f \geq w \lambda$,
\item[{\nor (II)}] if $c_{f, v_{w_0\lambda}}^{\lambda}.V_w\neq 0$ for some $\lambda \in P_+$ and a weight vector $f \in V(\lambda)^{\ast}$, then $\weight f \leq ww_0 \lambda$.
\end{itemize}
\end{proposition}
\subsection{Kuniba-Okado-Yamada's theorem}
\hfill Kuniba-Okado-Yamada's theorem\\
points out the relation between the $\mathbb{Q}(q)$-vector space $U_q(\mathfrak{n^+})$ and the simple $\mathbb{Q}_q[G]$-module $V_{w_0}$. To explain this, we introduce the PBW bases of $U_q(\mathfrak{n^+})$. (\cite[Chapter 37,40]{Lusbook})
\begin{definition}[Symmetries $T'_{i, \epsilon}$, $T''_{i, \epsilon}$]\label{braidaction}
For $i\in I$ and $\epsilon \in \{\pm 1 \}$, there exist $\mathbb{Q}(q)$-algebra automorphisms $T'_{i, \epsilon}$, $T''_{i, \epsilon}: U_q(\mathfrak{g})\to U_q(\mathfrak{g})$ defined by
\[
\begin{array}{l}
T'_{i, \epsilon}(E_i)=-K_{\epsilon i}F_i, T'_{i, \epsilon}(F_i)=-E_iK_{-\epsilon i},\\
\dis T'_{i, \epsilon}(E_j)=\sum_{r+s=-\langle \alpha_i^{\vee}, \alpha_j \rangle}(-1)^rq_i^{\epsilon r}E_i^{(r)}E_jE_i^{(s)} \mathrm{for\ } j\neq i,\\
\dis T'_{i, \epsilon}(F_j)=\sum_{r+s=-\langle \alpha_i^{\vee}, \alpha_j \rangle}(-1)^rq_i^{-\epsilon r}F_i^{(s)}F_jF_i^{(r)} \mathrm{for\ } j\neq i,\\
T'_{i, \epsilon}(K_h)=K_{h-\langle \alpha_i, h \rangle\alpha_i^{\vee}},\\
T''_{i, \epsilon}(E_i)=-F_iK_{\epsilon i}, T''_{i, \epsilon}(F_i)=-K_{-\epsilon i}E_i,\\
\dis T''_{i, \epsilon}(E_j)=\sum_{r+s=-\langle \alpha_i^{\vee}, \alpha_j \rangle}(-1)^rq_i^{-\epsilon r}E_i^{(s)}E_jE_i^{(r)} \mathrm{for\ } j\neq i,\\
\dis T''_{i, \epsilon}(F_j)=\sum_{r+s=-\langle \alpha_i^{\vee}, \alpha_j \rangle}(-1)^rq_i^{\epsilon r}F_i^{(r)}F_jF_i^{(s)} \mathrm{for\ } j\neq i,\\
T''_{i, \epsilon}(K_h)=K_{h-\langle \alpha_i, h \rangle\alpha_i^{\vee}}.\\
\end{array}
\]
Note that 
\[
\begin{array}{ll}
T'_{i, \epsilon}=(T''_{i, -\epsilon})^{-1} & \omega\circ T'_{i, \epsilon}\circ \omega=T''_{i, \epsilon},\\
\ast\circ T'_{i, \epsilon}\circ \ast=T''_{i, -\epsilon},& \overline{(\cdot )}\circ T^{\#}_{i, \epsilon}\circ\overline{(\cdot )}=T^{\#}_{i, -\epsilon} 
\end{array}
\]
where $\#\in\{\ ',\ ''\}$.
\end{definition}
\begin{definition}[PBW bases]\label{PBW}
Let $\epsilon \in \{ \pm 1 \}$ and ${\bf i}=(i_1, i_2,\dots, i_{l(w_0)})$ be a reduced word of the longest element $w_0$ of $W$. (i.e. $w_0=s_{i_1}s_{i_2}\cdots s_{i_{l(w_0)}}$.)

Then, the vectors
$$\Set{E_{i_1}^{(c_1)}T'_{i_1, \epsilon}(E_{i_2}^{(c_2)})\cdots T'_{i_1, \epsilon}T'_{i_2, \epsilon}\cdots T'_{i_{l(w_0)-1}, \epsilon}(E_{i_{l(w_0)}}^{(c_{l(w_0)})})|(c_1, c_2,\dots,c_{l(w_0)})\in (\mathbb{Z}_{\geqq 0})^{l(w_0)}}$$
form a basis of $U_q(\mathfrak{n^+})$. Moreover, the vectors
$$\Set{E_{i_1}^{(c_1)}T''_{i_1, \epsilon}(E_{i_2}^{(c_2)})\cdots T''_{i_1, \epsilon}T''_{i_2, \epsilon}\cdots T''_{i_{l(w_0)-1}, \epsilon}(E_{i_{l(w_0)}}^{(c_{l(w_0)})})|(c_1, c_2,\dots,c_{l(w_0)})\in (\mathbb{Z}_{\geqq 0})^{l(w_0)}}$$
also form a basis of $U_q(\mathfrak{n}^+)$. They are called the PBW bases of $U_q(\mathfrak{n}^+)$. (See, for instance, \cite[Chapter 8]{Janbook} for details.)
\end{definition}
\begin{remark}\label{positivelabel}
For any reduced word ${\bf i}=(i_1, i_2,\dots, i_{l(w_0)})$ of $w_0$, we have
\begin{center}
$\Delta_+=\{ \beta_{{\bf i}}^1, \beta_{{\bf i}}^2,\dots, \beta_{{\bf i}}^{l(w_0)} \}$ where $\beta_{{\bf i}}^k:=s_{i_1}\cdots s_{i_{k-1}}(\alpha_{i_k})$.
\end{center}
For any $k=1,\dots, l(w_0)$, we have
\begin{center}
$T^{\#}_{i_1, \epsilon}T^{\#}_{i_2, \epsilon}\cdots T^{\#}_{i_{k-1}, \epsilon}(E_{i_k})\in U_q(\mathfrak{n}^+)_{\beta_{{\bf i}}^k}$,
\end{center}
and if $\beta_{{\bf i}}^k=\alpha_i$ for some $i\in I$, then
\begin{center}
$T^{\#}_{i_1, \epsilon}T^{\#}_{i_2, \epsilon}\cdots T^{\#}_{i_{k-1}, \epsilon}(E_{i_k})=E_i$,
\end{center}
where $\#\in \{\ ',\ ''\}$. See for instance \cite[Proposition 8.20]{Janbook}.
\end{remark}
\begin{definition}\label{E_i^c}
In the setting of Theorem \ref{PBW}, we set
$$E_{{\bf i}}^{{\bf c}}:= E_{i_1}^{(c_1)}T'_{i_1, 1}(E_{i_2}^{(c_2)})\cdots T'_{i_1, 1}T'_{i_2, 1}\cdots T'_{i_{l(w_0)-1}, 1}(E_{i_{l(w_0)}}^{(c_{l(w_0)})}),$$
where ${\bf c}=(c_1, c_2,\dots,c_{l(w_0)})\in (\mathbb{Z}_{\geqq 0})^{l(w_0)}$.
\end{definition}
Now, we can state Kuniba-Okado-Yamada's theorem. 
Let ${\bf i}=(i_1, i_2,\dots, i_{l(w_0)})$ be a reduced word of $w_0$.
For ${\bf c}=(c_1, c_2,\dots,c_{l(w_0)})\in (\mathbb{Z}_{\geqq 0})^{l(w_0)}$, we set
\begin{align*}
| ({\bf c}) \rangle_{{\bf i}}:=|c_1\rangle_{i_1}\otimes\cdots\otimes|c_{l(w_0)}\rangle_{i_{l(w_0)}}
\end{align*}

Define the $\mathbb{Q}(q)$-linear isomorphism $\Phi_{{\bf i}}:U_q(\mathfrak{n}^+)\to V_{w_0}$ by $E_{{\bf i}}^{{\bf c}} \mapsto |({\bf c})\rangle_{{\bf i}}$.

\begin{theorem}[{
\cite[Theorem 5]{KOY}}]\label{KOY}
For any two reduced words ${\bf i}$, ${\bf j}$ of $w_0$, we have
$$\Phi_{{\bf j}}=\Theta_{{\bf j}, {\bf i}}\circ \Phi_{{\bf i}}.$$
\end{theorem}
Now, $\Theta_{{\bf j}, {\bf i}}$ can be regarded as the identity map of $V_{w_0}$. (i.e. We can identify  $V_{s_{i_1}}\otimes \cdots \otimes V_{s_{i_{l(w_0)}}}$ with $V_{s_{j_1}}\otimes \cdots \otimes V_{s_{j_{l(w_0)}}}$ via $\Theta_{{\bf j}, {\bf i}}$.) Hence, this theorem says that $U_q(\mathfrak{n}^+)$ and $V_{w_0}$ have the ``same'' natural bases and the map $\Phi_{{\bf i}}$ is determined independently of the choice of ${\bf i}$. In this sense, we denote this map $U_q(\mathfrak{n}^+)\to V_{w_0}$ by $\Phi_{\mathrm{KOY}}$. We will investigate the canonical basis of $U_q(\mathfrak{n}^+)$ using $\Phi_{\mathrm{KOY}}$ in Section \ref{main}. 
\section{The isomorphism $\mathbb{Q}_q[G/N^+]\simeq \dis \bigoplus_{\lambda \in P_+}F(\lambda )K_{\lambda}$ }\label{isom}
By Lemma \ref{decomposition} and Proposition \ref{hw}, we have $V_{w_0}=\mathbb{Q}_q[G/N^-].|(0)\rangle_{{\bf i}}$. In this section, we first establish the isomorphism of the right $U_q(\mathfrak{g})$-algebras between $\mathbb{Q}_q[G/N^-]$ and the specific subalgebra of a variant of the quantized enveloping algebra. Next, we check the compatibility of Kuniba-Okado-Yamada's conjecture and this isomorphism. (Theorem \ref{maintool}) Kuniba-Okado-Yamada's conjecture was recently proved by Saito (\cite{Saito_QFA}) and Tanisaki (\cite{Tanisaki_QFA}). This compatibility is one of the main tools of this paper.

The above-mentioned isomorphism is essentially written in the reference (\cite[Lemma 9.1.7]{Josbook}). However, our convention is slightly different from the one in that book (coproduct, left or right,\dots). So, we attach the details for the reader's convenience.

\begin{definition}\label{filt}
We consider the right adjoint action $\adr $ of $U_q(\mathfrak{g})$ on itself defined by $\adr (u)x=\sum_{(u)}S(u_{(1)})xu_{(2)}$ for $u, x \in U_q(\mathfrak{g})$, where $\Delta(u)=\sum_{(u)}u_{(1)}\otimes u_{(2)}$. Note that this is a right action.
We have the following equalities:
\[
\begin{array}{l}
\adr (E_i)x=-K_{-i}E_ix+K_{-i}xE_i,\\
\adr (F_i)x=-F_iK_ixK_{-i}+xF_i,\\
\adr (K_{\alpha_i^{\vee}})x=K_{-\alpha_i^{\vee}}xK_{\alpha_i^{\vee}},\\
\end{array}
\]
for $i \in I$ and $x \in U_q(\mathfrak{g})$.

Now, we can define a filtration $\mathcal{F}$ on $U_q(\mathfrak{g})$ in which $E_i$ has degree $(\alpha_i, \alpha_i)/2$, $F_i$ has degree 0 and $K_{\alpha_i^{\vee}}$ has degree 1 ($i \in I$).
By the equalities above, this filtration is $\adr(U_q(\mathfrak{g}))$-invariant. Hence, we obtain the right action of $U_q(\mathfrak{g})$ on $\gra_{\mathcal{F}}U_q(\mathfrak{g}) (:= \dis \bigoplus_{i \in \mathbb{Z}}\mathcal{F}_i(U_q(\mathfrak{g}))/\mathcal{F}_{i-1}(U_q(\mathfrak{g})))$ from the right adjoint action $\adr $. (This induced action will be again denoted by $\adr $.)

Let us denote by $U^+$ the subalgebra of $U_q(\mathfrak{g})$ generated by $\{ -K_{-i}E_i(=: G_i )\}_{i \in I}$. Then, $U^+$ is isomorphic to $U_q(\mathfrak{n^+})$ as a $\mathbb{Q}(q)$-algebra.
Now, $G_i$ has degree 0 with respect to $\mathcal{F}$. Therefore, $U^+$ can be naturally embedded 
into $\gra_{\mathcal{F}}U_q(\mathfrak{g})$ as a $\mathbb{Q}(q)$-algebra and its image will be denoted by $\check{U}^+$.
\end{definition}
\begin{lemma}\label{submod}
$\check{U}^+$ is a right $U_q(\mathfrak{g})$-submodule of $\gra_{\mathcal{F}}U_q(\mathfrak{g})$.
\end{lemma}
\begin{proof}
It is easy to check that $\adr (E_i)$ and $\adr (K_{\alpha_i^{\vee}})$ preserve $\check{U}^+$. So, we have only to prove that $\adr (F_i)$ preserves $\check{U}^+$ ($i \in I$).

we have
\[
\adr (F_i)(xy)=\adr (F_i)(x)\adr (K_{-i})(y)+x\adr (F_i)(y)
\]
for $x, y \in \check{U}^+$ and
\begin{align*}
\adr (F_i)(G_j)&=-F_iK_i(-K_{-j}E_j)K_{-i}+(-K_{-j}E_j)F_i\\
&=K_{-j}(F_iE_j-E_jF_i)=\dis \delta_{ij}\frac{-1+K_{-i}^2}{q_i-q_i^{-1}}=\delta_{ij}\frac{-1}{q_i-q_i^{-1}}.
\end{align*}
The last equality follows because the degree of $K_{-i}^2$ is $-(\alpha_i, \alpha_i)$ ($<0$).

These equalities complete the proof.
\end{proof}

\begin{lemma}\label{rightmod}
Let $K_{\lambda}$ be the symbol indexed by an element $\lambda$ of $P_+$ and $\check{U}^+K_{\lambda}$ be the $\mathbb{Q}(q)$-vector space such that 
\begin{center}
$a(xK_{\lambda})=(ax)K_{\lambda}$ and $xK_{\lambda}+yK_{\lambda}=(x+y)K_{\lambda}$
\end{center}
for $a \in \mathbb{Q}(q)$ and $x, y \in \check{U}^+$.
Then,
\begin{itemize}
\item[{\nor (I)}] The vector space $\check{U}^+K_{\lambda}$ has the right $U_q(\mathfrak{g})$-module structure (the action of $U_q(\mathfrak{g})$ is denoted by $\adr _{\lambda}$) defined by the Leibnitz rule, the given right adjoint action of $U_q(\mathfrak{g})$ on $\check{U}^+$ and,
\[
\begin{array}{ll}
\multicolumn{2}{c}{\adr _{\lambda}(E_i)K_{\lambda}=(1-q_i^{\langle 2\lambda, \alpha_i^{\vee} \rangle})G_iK_{\lambda},}\\
\adr _{\lambda}(F_i)K_{\lambda}=0,& \adr _{\lambda}(K_{h})K_{\lambda}=q^{\langle \lambda, h \rangle}K_{\lambda}
\end{array}
\]
for $i \in I$ and $h \in Q^{\vee}$.\\
\item[{\nor (II)}] The right $U_q(\mathfrak{g})$-module $\check{U}^+K_{\lambda}$ is isomorphic to the graded dual $M(\lambda)^{\ast, \gra}$ of the (left) Verma module $M(\lambda)$ with highest weight $\lambda$, where the graded dual is defined by $\dis \bigoplus_{\gamma \in P}\Hom_{\mathbb{Q}(q)}(M(\lambda)_{\gamma}, \mathbb{Q}(q))$ with a natural right $U_q(\mathfrak{g})$-module structure.\\
\end{itemize}
\end{lemma}
\begin{proof}
To prove (I), it suffices to show that the relations (i)-(vi) in \ref{QEA} are satisfied as the operators on $\check{U}^+K_{\lambda}$. The relations (i), (ii), (iii), (vi) are easily checked. The relation (v) is satisfied since the action of $E_i$ on $K_{\lambda}$ is essentially the same as the right adjoint action of $E_i$ on the element $K_{\sum_{i\in I}2(\lambda, \alpha_i)\varpi_i^{\vee}}$ in the variant of the quantized enveloping algebra whose basis of the Cartan part is indexed by the coweights (called the adjoint type in \cite{Lusbook}), where $\langle\varpi_i^{\vee}, \alpha_j\rangle=\delta_{ij} (i, j\in I)$.
 
The relation (iv) can be checked by straightforward calculation. Hence, (I) follows.

For (II), we have only to prove that the graded dual $(\check{U}^+K_{\lambda})^{\ast, \gra}$ of $\check{U}^+K_{\lambda}$ is isomorphic to $M(\lambda)$. Let $l_{\lambda}$ be  the element of $(\check{U}^+K_{\lambda})^{\ast, \gra}_{\lambda}$ (1 dimensional) such that $l_{\lambda}(K_{\lambda})=1$. Then, we have $E_i.l_{\lambda}=0$ for all $i\in I$. 

Now, by the proof of Lemma \ref{submod}, we can identify $-(q_i-q_i^{-1})\adr (F_i)$ on $\check{U}^+$ as $r_i$ on $U_q(\mathfrak{n}^+)$ via the $\mathbb{Q}(q)$-algebra isomorphism $\check{U}^+\rightarrow U_q(\mathfrak{n}^+), G_i\mapsto E_i$. Hence, it follows from Definition \ref{qderiv} (2) that $(\check{U}^+K_{\lambda})^{\ast, \gra}=U_q(\mathfrak{g}).l_{\lambda}$. Therefore, there exists a surjective (left) $U_q(\mathfrak{g})$-module homomorphism $M(\lambda)\to (\check{U}^+K_{\lambda})^{\ast, \gra}$. Moreover, these modules have the same formal characters. Hence, these modules are isomorphic and this completes the proof of (II).
\end{proof}

We can deduce from Lemma \ref{rightmod} (II) that $\adr _{\lambda}(U_q(\mathfrak{g}))(K_{\lambda})$ is a finite dimensional right $U_q(\mathfrak{g})$-module isomorphic to $V(\lambda)^{\ast}$.
\begin{notation}
For $\lambda \in P_+$, we set $F(\lambda )K_{\lambda}:= \adr _{\lambda}(U_q(\mathfrak{g}))(K_{\lambda})$ ($F(\lambda )\subset \check{U}^+$).
\end{notation}
\begin{lemma}\label{algstr}
For $\lambda, \mu \in P_+$, we have $F(\lambda)F(\mu)=F(\lambda+\mu)$.
\end{lemma}
\begin{proof}
It suffices to show that $F(\lambda)F(\mu)K_{\lambda+\mu}$ is a right $U_q(\mathfrak{g})$-submodule of  $\check{U}^+K_{\lambda+\mu}$. It is clear that the actions of $F_i$ ($i\in I$) and $K_h$ ($h\in Q^{\vee}$) preserve the space $F(\lambda)F(\mu)K_{\lambda+\mu}$. Consider the action of $E_i$ ($i\in I$). Recall that $\adr (E_i)(x)=G_ix-\adr (K_i)(x)G_i$ for all $x\in U_q(\mathfrak{g})$. For $G' \in F(\lambda)$ and $G''\in F(\mu)$, we have 
\begin{align*}
&\adr _{\lambda+\mu}(E_i)(G'G''K_{\lambda+\mu})\\
&=\adr (E_i)(G')G''K_{\lambda+\mu}+\adr (K_i)(G')\adr (E_i)(G'')K_{\lambda+\mu}\\
&\hspace{20pt}+\adr (K_i)(G')\adr (K_i)(G'')\adr _{\lambda+\mu}(E_i)(K_{\lambda+\mu})\\
&=\left\{ \adr (E_i)(G')+(1-q_i^{\langle 2\lambda, \alpha_i^{\vee} \rangle})\adr (K_i)(G')G_i \right\}G''K_{\lambda+\mu}\\
&\hspace{20pt}+q_i^{\langle 2\lambda, \alpha_i^{\vee} \rangle}\adr (K_i)(G')\left\{\adr (E_i)(G'')+(1-q_i^{\langle 2\mu, \alpha_i^{\vee} \rangle})\adr (K_i)(G'')G_i \right\}K_{\lambda+\mu}.
\end{align*}
By the definition of $F(\lambda)$ and $F(\mu)$, $\adr (E_i)(G')+(1-q_i^{\langle 2\lambda, \alpha_i^{\vee} \rangle})\adr (K_i)(G')G_i$ is an element of $F(\lambda)$ and $\adr (E_i)(G'')+(1-q_i^{\langle 2\mu, \alpha_i^{\vee} \rangle})\adr (K_i)(G'')G_i$ is an element of $F(\mu)$. This completes the proof.
\end{proof}
By Lemma \ref{algstr}, we can conclude that the vector space $C:= \bigoplus_{\lambda \in P_+}F(\lambda )K_{\lambda}$ has a $\mathbb{Q}(q)$-algebra structure, where $K_{\lambda}K_{\mu}=K_{\lambda+\mu}$ and $K_{\lambda}G_i=q^{(\lambda, \alpha_i)}G_iK_{\lambda}$ for $\lambda, \mu \in P_+$ and $i \in I$. The following proposition is the main proposition of this section. 
\begin{proposition}\label{mainisom}
\begin{itemize}
\item[{\nor (i)}] The algebra $C$ is a right $U_q(\mathfrak{g})$-algebra.
\item[{\nor (ii)}] The map $\Upsilon : C \to \mathbb{Q}_q[G/N^-], K_{\lambda}\mapsto c_{f_{\lambda}, v_{w_0\lambda}}^{\lambda} (\lambda \in P_+)$ is an isomorphism of right $U_q(\mathfrak{g})$-algebras.
\end{itemize}
\end{proposition}
\begin{proof}
Direct calculation shows that, for $i\in I$ and homogeneous elements $G', G'' \in \check{U}^+$,
\begin{align*}
&\adr _{\lambda}(E_i)(G'K_{\lambda})G''K_{\mu}+\adr _{\lambda}(K_i)(G'K_{\lambda})\adr _{\mu}(E_i)(G''K_{\mu})\\
&=q^{( \lambda, \weight G'' )}\adr _{\lambda+\mu}(E_i)(G'G''K_{\lambda+\mu}),\\
&\adr _{\lambda}(F_i)(G'K_{\lambda})\adr _{\mu}(K_{-i})(G''K_{\mu})+G'K_{\lambda}\adr _{\mu}(F_i)(G''K_{\mu})\\
&=q^{( \lambda, \weight G'' )}\adr _{\lambda+\mu}(F_i)(G'G''K_{\lambda+\mu}),\\
&\adr _{\lambda}(K_{\alpha_i^{\vee}})(G'K_{\lambda})\adr _{\mu}(K_{\alpha_i^{\vee}})(G''K_{\mu})\\
&=q^{( \lambda, \weight G'' )}\adr _{\lambda+\mu}(K_{\alpha_i^{\vee}})(G'G''K_{\lambda+\mu}).
\end{align*}
This completes the proof of (i). For the proof of (ii), we recall the following fact. (See \cite[9.1.6]{Josbook}.)
\begin{fact}\label{genrel}
Put $V := \bigoplus_{i \in I}V(\varpi_i)^{\ast}$ and let $T(V)$ be the tensor algebra of $V$.
The algebra $T(V)$ is a right $U_q(\mathfrak{g})$-algebra. For all $k \in \mathbb{Z}_{\geqq 0}$ and all $i_1, \dots , i_k \in I$, we have the right $U_q(\mathfrak{g})$-module homomorphism
\begin{center}
$V(\varpi_{i_1})^{\ast}\otimes \dots \otimes V(\varpi_{i_k})^{\ast}\to V(\varpi_{i_1}+\dots +\varpi_{i_k})^{\ast}, f_{\varpi_{i_1}}\otimes \dots \otimes f_{\varpi_{i_k}}\mapsto f_{\varpi_{i_1}+\dots +\varpi_{i_k}}$,
\end{center}
and denote its kernel by $K_{i_1, \dots , i_k}$. For $i, j \in I$, we set
\begin{center}
$E_{i, j}:= (f_{\varpi_i}\otimes f_{\varpi_j}-f_{\varpi_j}\otimes f_{\varpi_i}).U_q(\mathfrak{g})(\subset T(V)).$
\end{center} 
Let $K$ denote the two-sided ideal of $T(V)$ generated by the all $K_{i_1, \dots , i_k}$ and the all $E_{i, j}$.
Then, there exists the right $U_q(\mathfrak{g})$-algebra isomorphism $T(V)/K\to \mathbb{Q}_q[G/N^-]$ defined by $f_{\varpi_i}\mapsto c_{f_{\varpi_i}, v_{w_0\varpi_i}}^{\varpi_i}$ $(i \in I)$. Moreover, this isomorphism maps $f_{\varpi_{i_1}}\otimes \dots \otimes f_{\varpi_{i_k}}$ to $c_{f_{\varpi_{i_1}+\dots +\varpi_{i_k}}, v_{w_0(\varpi_{i_1}+\dots +\varpi_{i_k})}}^{\varpi_{i_1}+\dots +\varpi_{i_k}}$ for all $k \in \mathbb{Z}_{\geqq 0}$ and all $i_1, \dots , i_k \in I$.
\end{fact}
Now, we have the right $U_q(\mathfrak{g})$-algebra homomorphism $T(V)\to C$ defined by $f_{\varpi_i}\mapsto K_{\varpi_i}$. By Lemma \ref{algstr}, this is a surjective homomorphism. Moreover, it is easy to check that this homomorphism factors through $T(V)/K$ and the induced map $T(V)/K\to C$ is the isomorphism of the right $U_q(\mathfrak{g})$-algebras. (Note that, by Lemma \ref{genrel}, both $T(V)/K$ and $C$ are isomorphic to $\bigoplus_{\lambda \in P_+}V(\lambda)^{\ast}$ as right $U_q(\mathfrak{g})$-modules.) Combining this with Fact \ref{genrel}, we obtain Proposition \ref{mainisom} (ii).
\end{proof}
Set $\mathcal{S}:= \Set{c_{f_{\lambda}, v_{w_0\lambda}}^{\lambda} | \lambda \in P_+ }$. Then, the set $\mathcal{S}$ is a multiplicative set and is (left and right) Ore in $\mathbb{Q}_q[G/N^-]$. So, we can define the localization of $\mathbb{Q}_q[G/N^-]$ by $\mathcal{S}$ (denoted by $\mathbb{Q}_q[G/N^-]_{\mathcal{S}}$). Under the $\mathbb{Q}(q)$-algebra isomorphism $\Upsilon^{-1}$, the set $\mathcal{S}$ corresponds to the set $\Set{K_{\lambda} | \lambda \in P_+}$ (denoted by $\mathcal{S}'$), and the algebra $C_{\mathcal{S}'}$ is isomorphic to the algebra $\check{U}^{\geqq 0}:= \bigoplus_{\lambda \in P}\check{U}^+K_{\lambda}$, where the relations concerning to $K_{\lambda}$'s ($\lambda \in P$) are the following 
\[
\begin{array}{c}
K_{\lambda}K_{\mu}=K_{\lambda+\mu},\ \text{and}\\
K_{\lambda}G_i=q^{(\lambda, \alpha_i)}G_iK_{\lambda}\ \text{for}\ \lambda, \mu \in P\ \text{and}\ i \in I. 
\end{array}
\]
By Proposition \ref{hw} (II), we have 
\begin{align*}
c_{f_{\lambda}, v_{w_0\lambda}}^{\lambda}.|(0)\rangle_{{\bf i}}&=c_{f_{\lambda}, v_{s_{i_1}\lambda}}^{\lambda}.|(0)\rangle_{i_1} \otimes \cdots \otimes c_{f_{s_{i_{l(w_0)-1}}\cdots s_{i_1}\lambda}, v_{w_0\lambda}}^{\lambda}.|(0)\rangle_{i_{l(w_0)}} \\
&=|(0)\rangle_{i_1} \otimes \cdots \otimes |(0)\rangle_{i_{l(w_0)}}=|(0)\rangle_{{\bf i}},
\end{align*}
for any reduced word ${\bf i}$ of $w_0$. Moreover, the elements of $\mathcal{S}$ are $q$-central in $\mathbb{Q}_q[G/N^-]$. Hence, the $\mathbb{Q}_q[G]$-module $V_{w_0}$ can be regarded as a $\mathbb{Q}_q[G/N^-]$-module by restriction and this $\mathbb{Q}_q[G/N^-]$-module can be extended to the $\mathbb{Q}_q[G/N^-]_{\mathcal{S}}$-module (denoted again by $V_{w_0}$). 

The following proposition was conjectured by Kuniba, Okado and Yamada (\cite[Conjecture 1]{KOY}) and recently proved by Saito (\cite{Saito_QFA}) and Tanisaki (\cite{Tanisaki_QFA}). We define the $\mathbb{Q}_q[G]$(or $\mathbb{Q}_q[G/N^-]$, $\mathbb{Q}_q[G/N^-]_{\mathcal{S}}$ )-module structure on $U_q(\mathfrak{n}^+)$ via $\Phi_{\mathrm{KOY}}$. Set 
$$\xi_i:=(1-q_i^2)^{-1}(c_{f_{\varpi_i}, v_{w_0\varpi_i}}^{\varpi_i}.E_i)({c_{f_{\varpi_i}, v_{w_0\varpi_i}}^{\varpi_i}})^{-1}(\in \mathbb{Q}_q[G/N^-]_{\mathcal{S}}),$$ 
for $i\in I$.
\begin{proposition}[{\cite[Corollary 4.3.3]{Saito_QFA} \cite[Proposition 7.6]{Tanisaki_QFA}}]\label{Saito}
The $\mathbb{Q}_q[G/N^-]_{\mathcal{S}}$-module structure on $U_q(\mathfrak{n}^+)$ is determined by the $\mathbb{Q}(q)$-algebra homomorphism $\tilde{\rho}_{w_0}:\mathbb{Q}_q[G/N^-]_{\mathcal{S}}\to \End_{\mathbb{Q}(q)}(U_q(\mathfrak{n}^+))$ defined by 
\[
\begin{array}{l}
\xi_i\mapsto (x\mapsto E_ix),\\
c_{f_{\varpi_i}, v_{w_0\varpi_i}}^{\varpi_i}\mapsto (x\mapsto q^{(\varpi_i, \weight x)}x \text{\ for\ any\ homogeneous\ element\ }x).
\end{array}
\]
\end{proposition}
Combining Proposition \ref{mainisom} with Proposition \ref{Saito}, we obtain the following theorem.
\begin{theorem}\label{maintool}
The composition map $\check{U}^{\geqq 0}\xrightarrow[ ]{\Upsilon}\mathbb{Q}_q[G/N^-]_{\mathcal{S}}\xrightarrow[ ]{\tilde{\rho}_{w_0}} \End_{\mathbb{Q}(q)}(U_q(\mathfrak{n}^+))$ is given by
\[
\begin{array}{l}
G_i\mapsto (x\mapsto E_ix),\\
K_{\varpi_i}\mapsto (x\mapsto q^{(\varpi_i, \weight x)}x \text{\ for\ any\ homogeneous\ element\ }x).
\end{array}
\]
\end{theorem}
By Theorem \ref{maintool}, we can compute the action of the element of the form $c_{f_{\lambda}, v_{w_0\lambda}}^{\lambda}.x $ ($x\in U_q(\mathfrak{g})$) on $V_{w_0}$ by computing the element $\adr_{\lambda}(x)K_{\lambda}$ in $C$.
\begin{example}\label{sl2ex}
Set $\mathfrak{g}:= \mathfrak{sl}_2$. In this case, $P_+$ is naturally identified with $\mathbb{Z}_{\geqq 0}$ and $w_0=s_1$. Then, for any $n, k \in \mathbb{Z}_{\geqq 0}$ with $k \leqq n$, we have
\begin{align*}
\left( c_{f_n, v_{-n}}^{n}.E^{(k)} \right).\Phi_{\mathrm{KOY}}^{-1}(|(0) \rangle)&=\Upsilon\left( \adr_n(E^{(k)})K_n \right).\Phi_{\mathrm{KOY}}^{-1}(|(0) \rangle)\\
&=\Upsilon\left( \prod_{i=1}^{k}(1-q^{2(n-i+1)})G^{(k)}K_n \right).\Phi_{\mathrm{KOY}}^{-1}(|(0) \rangle)\\
&=\prod_{i=1}^{k}(1-q^{2(n-i+1)})E^{(k)}.
\end{align*}
Therefore, via $\Phi_{\mathrm{KOY}}$, we have 
$$\left( c_{f_n, v_{-n}}^{n}.E^{(k)} \right).|(0) \rangle=\Phi_{\mathrm{KOY}}(\prod_{i=1}^{k}(1-q^{2(n-i+1)})E^{(k)})=\prod_{i=1}^{k}(1-q^{2(n-i+1)})|(k) \rangle,$$
in $V_{s_1}$.
\end{example}
We remark that if $n$ is sufficiently larger than $k$ then the ``output'' element $\prod_{i=1}^{k}(1-	q^{2(n-i+1)})E^{(k)}$ is equal to the ``input'' element $\ast(E^{(k)})=E^{(k)}$ modulo the terms of the form (the sufficiently large powers of $q$ )$\cdot E^{(k)}$. In general, for $i\in I$, $\lambda\in P_+$ and a homogeneous element $\tilde{G}\in \check{U}^+$, we have
\begin{align}\label{act}
\adr_{\lambda}(E_i)(\tilde{G}K_{\lambda})=(G_i\tilde{G}-q^{(2\lambda-\weight \tilde{G}, \alpha_i)}\tilde{G}G_i)K_{\lambda}.
\end{align}
Hence, we obtain the following corollary. 
\begin{corollary}\label{converge}
For $x\in U_q(\mathfrak{n}^+), \lambda\in P_+$ and a reduced word ${\bf i}$ of $w_0$, we write 
$$x=\sum_{{\bf c}\in (\mathbb{Z}_{\geqq 0})^{l(w_0)}}{_{{\bf i}}\zeta}_{{\bf c}}^{x}E_{{\bf i}}^{{\bf c}}\ \text{with\ } {_{{\bf i}}\zeta}_{{\bf c}}^{x}\in \mathbb{Q}(q),\ \text{and}$$
$$(c_{f_{\lambda}, v_{w_0\lambda}}^{\lambda}.\ast(x)).|(0) \rangle_{{\bf i}}=\sum_{{\bf c}\in (\mathbb{Z}_{\geqq 0})^{l(w_0)}}{_{{\bf i}}\zeta}_{{\bf c}}^{\lambda, x}|({\bf c}) \rangle_{{\bf i}}\ \text{with\ } {_{{\bf i}}\zeta}_{{\bf c}}^{\lambda, x}\in \mathbb{Q}(q).$$
When $\lambda \in P_+$ tends to $\infty$ in the sense that $\langle \lambda, \alpha_i^{\vee}\rangle$ tends to $\infty$ for all $i\in I$, ${_{{\bf i}}\zeta}_{{\bf c}}^{\lambda, x}$ converges to ${_{{\bf i}}\zeta}_{{\bf c}}^{x}$ in the complete discrete valuation field $\mathbb{Q}((q))$.
\end{corollary}
This point of view is important in the argument in Section \ref{prfidea}. By the way, we also have the following corollary.
\begin{corollary}\label{converge'}
For $x\in U_q(\mathfrak{n}^+), \lambda\in P_+, w\in W$ and a reduced word ${\bf i}_w$ of $w$, we set
$$(c_{f_{\lambda}, v_{w^{-1}\lambda}}^{\lambda}.\ast(x)).|(0) \rangle_{{\bf i}_w}=\sum_{{\bf c}\in (\mathbb{Z}_{\geqq 0})^{l(w)}}{_{{\bf i}_w}\zeta}_{{\bf c}}^{\lambda, x}|({\bf c}) \rangle_{{\bf i}_w}\ \text{with\ } {_{{\bf i}_w}\zeta}_{{\bf c}}^{\lambda, x}\in \mathbb{Q}(q).$$
{\nor (}$|({\bf c}) \rangle_{{\bf i}_w}$ is defined in the same way as the $w=w_0$ case.{\nor )}

When $\lambda \in P_+$ tends to $\infty$, ${_{{\bf i}_w}\zeta}_{{\bf c}}^{\lambda, x}$ converges to ${_{{\bf i}_w{\bf i}'_w}\zeta}_{({\bf c}, 0,\dots, 0)}^{x}$ in $\mathbb{Q}((q))$, where ${\bf i}'_w$  is a reduced word of $w^{-1}w_0$. {\nor (}For ${\bf i}=(i_1,\dots, i_s)$ and ${\bf i}'=(i_{s+1},\dots, i_t)$, ${\bf i}{\bf i}'$ denotes $(i_1,\dots, i_s, i_{s+1},\dots, i_t)$. {\nor )}
\end{corollary}
\begin{proof}
Set $w'=w^{-1}w_0$. By Theorem \ref{soi}, $V_w\otimes V_{w'}\to V_{w_0}, |(0) \rangle_{{\bf i}_w}\otimes |(0) \rangle_{{\bf i}'_w}\mapsto |(0) \rangle_{{\bf i}_w{\bf i}'_w}$ is an isomorphism of $\mathbb{Q}_q[G]$-modules. Moreover, we have
\begin{align*}
&(c_{f_{\lambda}, v_{w_0\lambda}}^{\lambda}.\ast(x)).(|(0) \rangle_{{\bf i}_w}\otimes|(0) \rangle_{{\bf i}'_w})\\
&=(c_{f_{\lambda}, v_{w^{-1}\lambda}}^{\lambda}.\ast(x)).|(0) \rangle_{{\bf i}_w}\otimes c_{f_{w^{-1}\lambda}, v_{w_0\lambda}}^{\lambda}.|(0) \rangle_{{\bf i}'_w}\\
&\hspace{15pt}+\sum_{j; v_j\neq v_{w^{-1}\lambda}} (c_{f_{\lambda}, v_j}^{\lambda}.\ast(x)).|(0) \rangle_{{\bf i}_w}\otimes c_{f_j, v_{w_0\lambda}}^{\lambda}.|(0) \rangle_{{\bf i}'_w},
\end{align*}
where $\{ v_j \}_j$ is a basis of $V(\lambda)$ consisting of weight vectors ond containing $v_{w^{-1}\lambda}$, and $\{ f_j \}_j$ is its dual basis of $V(\lambda)^{\ast}$. It is easy to check that $c_{f_{w^{-1}\lambda}, v_{w_0\lambda}}^{\lambda}.|(0) \rangle_{{\bf i}'_w}=|(0) \rangle_{{\bf i}'_w}$ by Proposition \ref{hw}.

Combining these argument with Corollary \ref{converge}, we have only to show that 
$$c_{f_j, v_{w_0\lambda}}^{\lambda}.|(0) \rangle_{{\bf i}'_w}=\sum_{{\bf c}}\eta_{{\bf c}}^{(j)}|({\bf c}) \rangle_{{\bf i}'_w}\ \text{with\ }\eta_{(0,\dots,0)}^{(j)}\neq 0\ \text{only\ if\ } f_j=f_{w^{-1}\lambda}.$$

This statement easily follows from the computation in Example \ref{sl2ex}.
\end{proof}
\section{Preliminaries II : the canonical and dual canonical bases}\label{dab2}
We collect the definitions of the canonical and dual canonical bases and their properties. Our main result of this section is Proposition \ref{similarity}. All statements in this section are valid even when $\mathfrak{g}$ is a symmetrizable infinite Kac-Moody Lie algebra.
\subsection{The category of crystals}
We recall the definition of the abstract crystal introduced by Kashiwara \cite{KasDem} associated with the root datum $(P, \Delta, Q^{\vee}, \Delta^{\vee})$.
\begin{definition}\label{abstractcrystal}
A crystal (associated with the root datum $(P, \Delta, Q^{\vee}, \Delta^{\vee})$) is a set $B$ endowed with maps $\weight : B\to P, \varepsilon_i, \varphi_i: B\to \mathbb{Z}\coprod \{-\infty \}, \tilde{e}_i, \tilde{f}_i: B\to B\coprod \{0 \} (i\in I)$ satisfying the following conditions:
\begin{itemize}
\item[(i)] $\varphi_i(b)=\varepsilon_i(b)+\langle \alpha_i^{\vee}, \weight b \rangle$ for $i\in I$ and $b\in B$,
\item[(ii)] for $i\in I$, $b \in B$ with $\tilde{e}_ib \in B$,  $\weight\tilde{e}_ib=\weight b+\alpha_i$, $\varepsilon_i(\tilde{e}_ib)=\varepsilon_i(b)-1$, $\varphi_i(\tilde{e}_ib)=\varphi_i(b)+1$,
\item[(iii)] for $i\in I$, $b \in B$ with $\tilde{e}_ib \in B$, $\weight\tilde{f}_ib=\weight b-\alpha_i$, $\varepsilon_i(\tilde{f}_ib)=\varepsilon_i(b)+1$, $\varphi_i(\tilde{f}_ib)=\varphi_i(b)-1$,
\item[(iv)] for $b, b' \in B$ and $i\in I$, $b'=\tilde{e}_ib$ if and only if $b=\tilde{f}_ib'$,
\item[(v)] if $\varphi_i(b)=-\infty$ for $i\in I$, $b \in B$, then $\tilde{e}_ib=\tilde{f}_ib=0$.
\end{itemize} 
For two crystals $B_1, B_2$, a map $\psi: B_1\coprod\{ 0 \}\to B_2\coprod\{ 0 \}$ is called a morphism of crystals from $B_1$ to $B_2$ if it satisfies the following properties:
\begin{itemize}
\item[(i)] $\psi(0)=0$,
\item[(ii)] $\weight(\psi(b))=\weight(b)$, $\varepsilon_i(\psi(b))=\varepsilon_i(b)$, $\varphi_i(\psi(b))=\varphi_i(b)$ for all $i\in I$ and $b\in B_1$ with $\psi(b)\in B_2$,
\item[(iii)] $\tilde{e}_i\psi(b)=\psi(\tilde{e}_ib)$ if $\psi(b), \psi(\tilde{e}_ib)\in B_2$,
\item[(iv)] $\tilde{f}_i\psi(b)=\psi(\tilde{f}_ib)$ if $\psi(b), \psi(\tilde{f}_ib)\in B_2$.
\end{itemize}
\end{definition}
\subsection{The canonical bases of $U_q(\mathfrak{n}^-)$}
We define the canonical basis of $U_q(\mathfrak{n}^-)$ according to Kashiwara's method \cite{Kas_originalcrys}. Let $\mathcal{A}_0$ be the subring of $\mathbb{Q}(q)$ consisting of rational functions without poles at $q=0$.
\begin{definition}\label{lowercrystal}
For $i\in I$, we have $U_q(\mathfrak{n}^-)=\bigoplus_{n\in \mathbb{Z}_{\geqq 0}}F_i^{(n)}\Ker e'_i$ (\cite[3.5]{Kas_originalcrys}). Hence, we can define the $\mathbb{Q}(q)$-linear maps $\tilde{e}_i, \tilde{f}_i:U_q(\mathfrak{n}^-)\to U_q(\mathfrak{n}^-)$ by
\begin{center}
$\tilde{e}_i(F_i^{(n)}u)=F_i^{(n-1)}u$, $\tilde{f}_i(F_i^{(n)}u)=F_i^{(n+1)}u$
\end{center}
for $u\in \Ker e'_i$ where $F_i^{(-1)}u:=0$. We have $\tilde{e}_i\circ \tilde{f}_i=\mathrm{id}_{U_q(\mathfrak{n}^-)}$.
\\

We set
\[
\begin{array}{l}
\dis L(\infty):=\sum_{l\geqq 0, i_1,\dots , i_l\in I}\mathcal{A}_0\tilde{f}_{i_1}\cdots\tilde{f}_{i_l}1(\subset U_q(\mathfrak{n}^-)),\\
B(\infty):=\Set{\tilde{f}_{i_1}\cdots\tilde{f}_{i_l}1\ \mathrm{mod}\ qL(\infty)|l\geqq 0, i_1,\dots , i_l\in I} (\subset L(\infty)/qL(\infty)).
\end{array}
\]
Then, $(L(\infty), B(\infty))$ satisfies the following properties(\cite[Theorem 4]{Kas_originalcrys}):
\begin{itemize}
\item[(i)] The $\mathcal{A}_0$-submodule $L(\infty)$ of $U_q(\mathfrak{n}^-)$ is free and $\mathbb{Q}(q)\otimes_{\mathcal{A}_0}L(\infty)\simeq U_q(\mathfrak{n}^-)$,
\item[(ii)] The set $B(\infty)$ is a basis of the $\mathbb{Q}$-vector space $L(\infty)/qL(\infty)$,
\item[(iii)] $\tilde{e}_iL(\infty)\subset L(\infty)$ and $\tilde{f}_iL(\infty)\subset L(\infty)$ for all $i\in I$,
\item[(iv)] $\tilde{e}_i: B(\infty)\to B(\infty)\coprod \{ 0\}$ and $\tilde{f}_i: B(\infty)\to B(\infty)$ for all $i\in I$,
\item[(v)] For $b\in B(\infty)$ with $\tilde{e}_ib\in B(\infty)$, we have $b=\tilde{f}_i\tilde{e}_ib$.
\end{itemize}
This pair $(L(\infty), B(\infty))$ is called the lower crystal basis of $U_q(\mathfrak{n}^-)$. 

We define the maps $\varepsilon_i, \varphi_i:B(\infty)\to \mathbb{Z}\coprod \{-\infty \}$ ($i\in I$) by
\begin{center}
$\varepsilon_i(b)=\max\Set{ k\in \mathbb{Z}_{\geqq 0}| \tilde{e}_i^k b\in B(\infty) }$, $\varphi_i(b)=\varepsilon_i(b)+\langle \alpha_i^{\vee}, \weight b\rangle$,
\end{center}
for $b\in B(\infty)$. (The images of $\varepsilon_i, \varphi_i$'s are in $\mathbb{Z}$.) 

Then, the sextuple $(B(\infty); \weight, \{\tilde{e}_i\}_{i\in I}, \{\tilde{f}_i\}_{i\in I}, \{ \varepsilon_i\}_{i\in I}, \{\varphi_i\}_{i\in I})$ is a crystal.
\\

Moreover, we have $\ast(L(\infty))=L(\infty)$ and $\ast(B(\infty))=B(\infty)$ (\cite[Proposition 5.2.4]{Kas_originalcrys}, \cite[Theorem 2.1.1]{KasDem}). So, we can define the maps $\tilde{e}_i^{\ast}:=\ast\circ\tilde{e}_i\circ\ast, \tilde{f}_i^{\ast}:=\ast\circ\tilde{f}_i\circ\ast:B(\infty)\to B\coprod \{0 \} (i\in I)$. We set $\varepsilon_i^{\ast}:=\varepsilon_i\circ\ast, \varphi_i^{\ast}:=\varphi_i\circ\ast :B(\infty)\to \mathbb{Z}\coprod \{-\infty \}$. Note that $\varepsilon_i^{\ast}(b)=\max\Set{ k\in \mathbb{Z}_{\geqq 0}| \tilde{e}_i^{\ast k}b\in B(\infty) }$. 

Then, the sextuple $(B(\infty); \weight, \{\tilde{e}_i^{\ast}\}_{i\in I}, \{\tilde{f}_i^{\ast}\}_{i\in I}, \{ \varepsilon_i^{\ast}\}_{i\in I}, \{\varphi_i^{\ast}\}_{i\in I})$ is again a crystal.
\end{definition}
\begin{definition}\label{lowerglobal}
Let $U_q(\mathfrak{n}^-)_{\mathbb{Q}}$ be the $\mathbb{Q}[q^{\pm 1}]$-subalgebra of $U_q(\mathfrak{n}^-)$ generated by $\{ F_i^{(n)} \}_{i\in I, n\in \mathbb{Z}_{\geqq 0}}$. Then, the canonical map
\begin{center}
$E:=L(\infty)\cap \overline{L(\infty)}\cap U_q(\mathfrak{n}^-)_{\mathbb{Q}}\to L(\infty)/qL(\infty)$ 
\end{center}
is an isomorphism of $\mathbb{Q}$-vector spaces  (\cite[Theorem 6]{Kas_originalcrys}). Let $G^{\mathrm{low}}:L(\infty)/qL(\infty)\to E$ be the inverse of this map.

Now, the set $\Set{ G^{\mathrm{low}}(b) | b\in B(\infty)}$ is a $\mathbb{Q}[q^{\pm 1}]$-basis of $U_q(\mathfrak{n}^-)_{\mathbb{Q}}$ and this is called the canonical basis of $U_q(\mathfrak{n}^-)$. (This forms also a $\mathcal{A}_0$-basis of $L(\infty)$.) We have $\ast(G^{\mathrm{low}}(b))=G^{\mathrm{low}}(\ast b)$ for $b\in B(\infty)$. (Recall the last part of Definition \ref{lowercrystal}.)
\end{definition}
\subsection{The canonical and dual canonical bases of $V(\lambda)$}
In this subsection, we recall the definition of the canonical and dual canonical bases of the module $V(\lambda)$ ($\lambda\in P_+$). The main references are \cite{Kas_originalcrys} and \cite{Kasglo}.
\begin{definition}\label{crystalsinrepresentation}
Let $\dis M=\bigoplus_{\lambda \in P}M_{\lambda}$ be an object of $\mathcal{O}_{\mathrm{int}}(\mathfrak{g})$. For $i\in I$, we define the $\mathbb{Q}(q)$-linear maps $\tilde{e}_i^{\mathrm{low}}, \tilde{f}_i^{\mathrm{low}}, \tilde{e}_i^{\mathrm{up}}, \tilde{f}_i^{\mathrm{up}}: M\to M$ by
\[
\begin{array}{ll}
\tilde{e}_i^{\mathrm{low}}(F_i^{(n)}u)=F_i^{(n-1)}u,& \tilde{f}_i^{\mathrm{low}}(F_i^{(n)}u)=F_i^{(n+1)}u\\
\dis \tilde{e}_i^{\mathrm{up}}(F_i^{(n)}u)=\frac{\left[\langle \alpha_i^{\vee}, \lambda \rangle-n+1\right]_i}{[n]_i}F_i^{(n-1)}u,& \dis \tilde{f}_i^{\mathrm{up}}(F_i^{(n)}u)=\frac{[n+1]_i}{\left[\langle \alpha_i^{\vee}, \lambda \rangle-n \right]_i}F_i^{(n+1)}u, 
\end{array}
\]
for $u\in \Ker (E_i.) \cap M_{\lambda}$, where $F_i^{(-1)}u:=0$.

Now, we can define a unique symmetric non-degenerate bilinear form $(\ ,\ )_{\lambda}:V(\lambda)\otimes V(\lambda)\to \mathbb{Q}(q)$ ($\lambda\in P_+$) by
\[
\begin{array}{c}
(xu, v)_{\lambda}=(u, \varphi (x)v)_{\lambda} \text{\ for\ all\ } u, v \in V(\lambda) \text{\ and\ } x\in U_q(\mathfrak{g}), \text{and}\\
(v_{\lambda}, v_{\lambda})_{\lambda}=1. \\
\end{array}
\]
Then, we have
\begin{center}
$(\tilde{e}_i^{\mathrm{up}}u, v)_{\lambda}=(u, \tilde{f}_i^{\mathrm{low}}v )_{\lambda},\  (\tilde{f}_i^{\mathrm{up}}u, v)_{\lambda}=(u, \tilde{e}_i^{\mathrm{low}}v )_{\lambda}$
\end{center}
for $u, v\in V(\lambda)$ and $i\in I$.

For $\lambda\in P_+$, we set

\begin{align*}
\dis L^{\mathrm{low}}(\lambda)&:=\sum_{l\geqq 0, i_1,\dots , i_l\in I}\mathcal{A}_0\tilde{f}_{i_1}^{\mathrm{low}}\cdots\tilde{f}_{i_l}^{\mathrm{low}}v_{\lambda}\ (\subset V(\lambda)),\\
B^{\mathrm{low}}(\lambda)&:=\Set{\tilde{f}_{i_1}^{\mathrm{low}}\cdots\tilde{f}_{i_l}^{\mathrm{low}}v_{\lambda}\  \mathrm{mod}\ qL^{\mathrm{low}}(\lambda)|l\geqq 0, i_1,\dots , i_l\in I}\setminus \{ 0 \}\\
&(\subset L^{\mathrm{low}}(\lambda)/qL^{\mathrm{low}}(\lambda)).
\end{align*}

Then, $(L^{\mathrm{low}}(\lambda), B^{\mathrm{low}}(\lambda))$ satisfies the following properties(\cite[Theorem 2]{Kas_originalcrys}):
\begin{itemize}
\item[(i)] The $\mathcal{A}_0$-submodule $L^{\mathrm{low}}(\lambda)$ of $V(\lambda)$ is  free and $\mathbb{Q}(q)\otimes_{\mathcal{A}_0}L^{\mathrm{low}}(\lambda)\simeq V(\lambda)$,
\item[(ii)] The set $B^{\mathrm{low}}(\lambda)$ is a basis of the $\mathbb{Q}$-vector space $L^{\mathrm{low}}(\lambda)/qL^{\mathrm{low}}(\lambda)$,
\item[(iii)] $L^{\mathrm{low}}(\lambda)=\bigoplus_{\mu\in P}(L^{\mathrm{low}}(\lambda)\cap V(\lambda)_{\mu})(=: L^{\mathrm{low}}(\lambda)_{\mu})$ and \\
$B^{\mathrm{low}}(\lambda)=\coprod_{\mu\in P}(B^{\mathrm{low}}(\lambda)\cap L^{\mathrm{low}}(\lambda)_{\mu}/qL^{\mathrm{low}}(\lambda)_{\mu})(=: B^{\mathrm{low}}(\lambda)_{\mu})$,
\item[(iv)] $\tilde{e}_i^{\mathrm{low}}L^{\mathrm{low}}(\lambda)\subset L^{\mathrm{low}}(\lambda)$ and $\tilde{f}_i^{\mathrm{low}}L^{\mathrm{low}}(\lambda)\subset L^{\mathrm{low}}(\lambda)$ for all $i\in I$,
\item[(v)] $\tilde{e}_i^{\mathrm{low}}: B^{\mathrm{low}}(\lambda)\to B^{\mathrm{low}}(\lambda)\coprod \{ 0\}$ and $\tilde{f}_i^{\mathrm{low}}: B^{\mathrm{low}}(\lambda)\to B^{\mathrm{low}}(\lambda)\coprod \{ 0\}$ for all $i\in I$,
\item[(vi)] For $b, b'\in B^{\mathrm{low}}(\lambda)$, we have $b'=\tilde{f}_i^{\mathrm{low}}b$ if and only if $b=\tilde{e}_i^{\mathrm{low}}b'$.
\end{itemize}
This pair $(L^{\mathrm{low}}(\lambda), B^{\mathrm{low}}(\lambda))$ is called the lower crystal basis of $V(\lambda)$. We define the maps $\varepsilon_i^{\lambda,\mathrm{low}}, \varphi_i^{\lambda, \mathrm{low}}:B^{\mathrm{low}}(\lambda)\to \mathbb{Z}\coprod \{-\infty \}$ ($i\in I$) by
\begin{align*}
\varepsilon_i^{\lambda,\mathrm{low}}(b)&=\max\Set{ k\in \mathbb{Z}_{\geqq 0}| (\tilde{e}_i^{\mathrm{low}})^k b\in B^{\mathrm{low}}(\lambda) },\\
\varphi_i^{\lambda,\mathrm{low}}(b)&=\max\Set{ k\in \mathbb{Z}_{\geqq 0}| (\tilde{f}_i^{\mathrm{low}})^k b\in B^{\mathrm{low}}(\lambda) },
\end{align*}
for $b\in B^{\mathrm{low}}(\lambda)$. (The images of $\varepsilon_i^{\lambda,\mathrm{low}}, \varphi_i^{\lambda,\mathrm{low}}$'s are in $\mathbb{Z}$.) Then, the sextuple $(B^{\mathrm{low}}(\lambda); \weight, \{\tilde{e}_i^{\mathrm{low}}\}_{i\in I}, \{\tilde{f}_i^{\mathrm{low}}\}_{i\in I}, \{ \varepsilon_i^{\lambda,\mathrm{low}}\}_{i\in I}, \{\varphi_i^{\lambda,\mathrm{low}}\}_{i\in I})$ is a crystal.

We can define the $\mathbb{Q}$-linear automorphism $\overline{(\cdot )}:V(\lambda)\to V(\lambda)$ ($\lambda\in P_+$) by $\overline{xv_{\lambda}}=\overline{x}v_{\lambda}$ for all $x\in U_q(\mathfrak{g})$.

For $\lambda\in P_+$, we set $V(\lambda)_{\mathbb{Q}}^{\mathrm{low}}:=U_q(\mathfrak{n}^-)_{\mathbb{Q}}v_{\lambda}$. Then, the canonical map 
\begin{center}
$E^{\mathrm{low}}_{\lambda}:=L^{\mathrm{low}}(\lambda)\cap \overline{L^{\mathrm{low}}(\lambda)}\cap V(\lambda)_{\mathbb{Q}}^{\mathrm{low}}\to L^{\mathrm{low}}(\lambda)/qL^{\mathrm{low}}(\lambda)$
\end{center}
is an isomorphism of $\mathbb{Q}$-vector spaces (\cite[Theorem 6]{Kas_originalcrys}). 

Let $G_{\lambda}^{\mathrm{low}}:L^{\mathrm{low}}(\lambda)/qL^{\mathrm{low}}(\lambda)\to E^{\mathrm{low}}_{\lambda}$ be the inverse of this map.

Now, the set $\Set{ G_{\lambda}^{\mathrm{low}}(b) | b\in B^{\mathrm{low}}(\lambda)}$ is an $\mathbb{Q}[q^{\pm 1}]$-basis of $V(\lambda)_{\mathbb{Q}}^{\mathrm{low}}$ and this is called the canonical basis of $V(\lambda)$. 

For $\lambda\in P_+$, we set
\[
\begin{array}{l}
V(\lambda)_{\mathbb{Q}}^{\mathrm{up}}:=\Set{u\in V(\lambda)| (u, V(\lambda)_{\mathbb{Q}}^{\mathrm{low}} )_{\lambda}\subset \mathbb{Q}[q^{\pm 1}]},\\
L(\lambda)^{\mathrm{up}}:=\Set{u\in V(\lambda)| (u, L(\lambda)^{\mathrm{low}} )_{\lambda}\subset \mathcal{A}_0}.\\
\end{array}
\]
Then, $\overline{L(\lambda)^{\mathrm{up}}}=\Set{u\in V(\lambda)| (u, \overline{L(\lambda)^{\mathrm{low}}} )_{\lambda}\subset \mathcal{A}_{\infty}}$ where $\mathcal{A}_{\infty}$ is the subring of $\mathbb{Q}(q)$ consisting of rational functions without poles at $q=\infty$. 

Let $B^{\mathrm{up}}(\lambda)$ be the basis of $L^{\mathrm{up}}(\lambda)/qL^{\mathrm{up}}(\lambda)$ dual to $B^{\mathrm{low}}(\lambda)$ with respect to the induced pairing $(\ ,\ )_{\lambda}:L^{\mathrm{up}}(\lambda)/qL^{\mathrm{up}}(\lambda) \times L^{\mathrm{low}}(\lambda)/qL^{\mathrm{low}}(\lambda)\to \mathbb{Q}$. Then, the pair $(L^{\mathrm{up}}(\lambda), B^{\mathrm{up}}(\lambda))$ satisfies the following properties(\cite[Proposition 3.2.2]{Kasglo}):
\begin{itemize}
\item[(i)] The $\mathcal{A}_0$-submodule $L^{\mathrm{up}}(\lambda)$ of $V(\lambda)$ is free and $\mathbb{Q}(q)\otimes_{\mathcal{A}_0}L^{\mathrm{up}}(\lambda)\simeq V(\lambda)$,
\item[(ii)] The set $B^{\mathrm{up}}(\lambda)$ is a basis of the $\mathbb{Q}$-vector space $L^{\mathrm{up}}(\lambda)/qL^{\mathrm{up}}(\lambda)$,
\item[(iii)] $L^{\mathrm{up}}(\lambda)=\bigoplus_{\mu\in P}(L^{\mathrm{up}}(\lambda)\cap V(\lambda)_{\mu})(=: L^{\mathrm{up}}(\lambda)_{\mu})$ and\\
$B^{\mathrm{up}}(\lambda)=\coprod_{\mu\in P}(B^{\mathrm{up}}(\lambda)\cap L^{\mathrm{up}}(\lambda)_{\mu}/qL^{\mathrm{up}}(\lambda)_{\mu})(=: B^{\mathrm{up}}(\lambda)_{\mu})$,
\item[(iv)] $\tilde{e}_i^{\mathrm{up}}L^{\mathrm{up}}(\lambda)\subset L^{\mathrm{up}}(\lambda)$ and $\tilde{f}_i^{\mathrm{up}}L^{\mathrm{up}}(\lambda)\subset L^{\mathrm{up}}(\lambda)$ for all $i\in I$,
\item[(v)] $\tilde{e}_i^{\mathrm{up}}: B^{\mathrm{up}}(\lambda)\to B^{\mathrm{up}}(\lambda)\coprod \{ 0\}$ and $\tilde{f}_i^{\mathrm{up}}: B^{\mathrm{up}}(\lambda)\to B^{\mathrm{up}}(\lambda)\coprod \{ 0\}$ for all $i\in I$,
\item[(vi)] For $b, b'\in B^{\mathrm{up}}(\lambda)$, we have $b'=\tilde{f}_i^{\mathrm{up}}b$ if and only if $b=\tilde{e}_i^{\mathrm{up}}b'$.
\end{itemize}
This pair $(L^{\mathrm{up}}(\lambda), B^{\mathrm{up}}(\lambda))$ is called the upper crystal basis of $V(\lambda)$. We define the maps $\varepsilon_i^{\lambda,\mathrm{up}}, \varphi_i^{\lambda, \mathrm{up}}:B^{\mathrm{up}}(\lambda)\to \mathbb{Z}\coprod \{-\infty \}$ ($i\in I$) by
\begin{align*}
\varepsilon_i^{\lambda,\mathrm{up}}(b)&=\max\Set{ k\in \mathbb{Z}_{\geqq 0}| (\tilde{e}_i^{\mathrm{up}})^k b\in B^{\mathrm{up}}(\lambda) },\\
\varphi_i^{\lambda,\mathrm{up}}(b)&=\max\Set{ k\in \mathbb{Z}_{\geqq 0}| (\tilde{f}_i^{\mathrm{up}})^k b\in B^{\mathrm{up}}(\lambda) },
\end{align*}
for $b\in B^{\mathrm{up}}(\lambda)$. (The images of $\varepsilon_i^{\lambda,\mathrm{up}}, \varphi_i^{\lambda,\mathrm{up}}$'s are in $\mathbb{Z}$.) Then, the sextuple $(B^{\mathrm{up}}(\lambda); \weight, \{\tilde{e}_i^{\mathrm{up}}\}_{i\in I}, \{\tilde{f}_i^{\mathrm{up}}\}_{i\in I}, \{ \varepsilon_i^{\lambda,\mathrm{up}}\}_{i\in I}, \{\varphi_i^{\lambda,\mathrm{up}}\}_{i\in I})$ is a crystal.

The canonical map
\begin{center}
$E^{\mathrm{up}}_{\lambda}:=L^{\mathrm{up}}(\lambda)\cap \overline{L^{\mathrm{up}}(\lambda)}\cap V(\lambda)_{\mathbb{Q}}^{\mathrm{up}}\to L^{\mathrm{up}}(\lambda)/qL^{\mathrm{up}}(\lambda)$
\end{center}
is an isomorphism of $\mathbb{Q}$-vector spaces (\cite[Lemma 2.2.3]{Kasglo}). 

Let $G_{\lambda}^{\mathrm{up}}:L^{\mathrm{up}}(\lambda)/qL^{\mathrm{up}}(\lambda)\to E^{\mathrm{up}}_{\lambda}$ be the inverse of this map.

Now, the set $\Set{ G_{\lambda}^{\mathrm{up}}(b) | b\in B^{\mathrm{up}}(\lambda)}$ is an $\mathbb{Q}[q^{\pm 1}]$-basis of $V(\lambda)_{\mathbb{Q}}^{\mathrm{up}}$ and this is called the dual canonical basis of $V(\lambda)$. The dual canonical basis of $V(\lambda)$ is the dual basis of the canonical basis of $V(\lambda)$ with respect to $(\ ,\ )_{\lambda}$.

In fact, $B^{\mathrm{low}}(\lambda)$ and $B^{\mathrm{up}}(\lambda)$ are isomorphic as (abstract) crystals. Hence, we will abbreviate both crystals $B^{\mathrm{low}}(\lambda)$, $B^{\mathrm{up}}(\lambda)$ to $B(\lambda)$ and accordingly $\tilde{e}_i^{\mathrm{low}}$, $\tilde{e}_i^{\mathrm{up}}$ to $\tilde{e}_i$ (when we consider these operations as ones on $B(\lambda)\coprod \{ 0 \}$), etc.
\end{definition}
\subsection{Some properties of the canonical and dual canonical bases}
In this subsection, we collect some important properties of the canonical and dual canonical bases. 
\begin{proposition}[{\cite[Theorem 5, Lemma 7.3.2]{Kas_originalcrys}}]\label{subcrystal}
For $\lambda\in P_+$, we define a surjective $U_q(\mathfrak{n}^-)$-module homomorphism $\pi_{\lambda}:U_q(\mathfrak{n}^-)\to V(\lambda)$ by $x\mapsto x.v_{\lambda}$. Then, we have 
\begin{itemize}
\item[{\nor (i)}] $\pi_{\lambda}(L(\infty))=L^{\mathrm{low}}(\lambda)$. Hence, $\pi_{\lambda}$ induces the surjective $\mathbb{Q}$-linear map $L(\infty)/qL(\infty)\to L^{\mathrm{low}}(\lambda)/qL^{\mathrm{low}}(\lambda)\ (\text{denoted\ again\ by}\ \pi_{\lambda})$.
\item[{\nor (ii)}] The map $\pi_{\lambda}$ induces the bijection
$$\Set{b\in B(\infty)| \pi_{\lambda}(b)\neq 0 (\text{\ in\ }L^{\mathrm{low}}(\lambda)/qL^{\mathrm{low}}(\lambda)) }\tilde{\to}B(\lambda).$$
\item[{\nor (iii)}] $\tilde{f}_i\pi_{\lambda}(b)=\pi_{\lambda}(\tilde{f}_ib)$ for all $b\in B(\infty)$ and $i\in I$.
\item[{\nor (iv)}] If $b\in B(\infty)$ satisfies $\pi_{\lambda}(b)\neq 0$, then $\tilde{e}_i\pi_{\lambda}(b)=\pi_{\lambda}(\tilde{e}_ib)$ for all $i\in I$.
\item[{\nor (v)}] $\pi_{\lambda}(G^{\mathrm{low}}(b))=G_{\lambda}^{\mathrm{low}}(\pi_{\lambda}(b))$ for any $b\in B(\infty)$.
\end{itemize}
\end{proposition}
\begin{remark}\label{crystalmorph}
Let $\lambda\in P_+$. It follows easily from Proposition \ref{subcrystal} (iii), (iv) that, for any $b\in B(\infty)$ such that $\pi_{\lambda}(b)\in B(\lambda)$, 
\begin{itemize}
\item[(i)] $\weight \pi_{\lambda}(b)=\weight b+\lambda$,
\item[(ii)] $\varepsilon_i^{\lambda}(\pi_{\lambda}(b))=\varepsilon_i(b)$,
\item[(iii)] $\varphi_i^{\lambda}(\pi_{\lambda}(b))=\varphi_i(b)+\langle \lambda, \alpha_i^{\vee} \rangle$.
\end{itemize}
\end{remark}
\begin{proposition}[{\cite[Lemma 7.3.4]{Kas_originalcrys}}]\label{barinv}
For $b\in B(\infty)$, we have $$\overline{G^{\mathrm{low}}(b)}=G^{\mathrm{low}}(b).$$
\end{proposition}
\begin{proposition}[{\cite[Theorem 7]{Kas_originalcrys}}]\label{c1=epsiloni1}
The following hold:
\begin{itemize}
\item[{\nor (i)}] The set $\{ G^{\mathrm{low}}(b)\}_{b; \varepsilon_i(b)\geqq p}$ is a $\mathbb{Q}(q)$-basis of $F_i^{(p)}U_q(\mathfrak{n}^-)$ for any $p\in \mathbb{Z}_{\geqq 0}$.
\item[{\nor (ii)}] The set $\{ G^{\mathrm{low}}(b)\}_{b; \varepsilon_i^{\ast}(b)\geqq p}$ is a $\mathbb{Q}(q)$-basis of $U_q(\mathfrak{n}^-)F_i^{(p)}$ for any $p\in \mathbb{Z}_{\geqq 0}$.
\end{itemize}
\end{proposition}
\begin{proposition}[{\cite[Lemma 3.2.1, Proposition 4.1]{KasDem}}]\label{Demazure}
For $\lambda\in P_+$ and $w\in W$, we define the subspace $V_w(\lambda)$ of $V(\lambda)$ by
\begin{center}
$V_w(\lambda):=U_q(\mathfrak{n}^-).v_{w\lambda}$.
\end{center}
Then, the following hold:
\begin{itemize}
\item[{\nor (i)}] If $l(s_iw)>l(w)$, then $V_w(\lambda)=U_{q_i}(\mathfrak{sl}_{2, i}).V_{s_iw}(\lambda)$.
\item[{\nor (ii)}] For any $b\in B(\infty)$, $G^{\mathrm{low}}(b).v_{w\lambda}\in G_{\lambda}^{\mathrm{low}}(B(\lambda))\coprod \{ 0\}$.
\item[{\nor (iii)}] If $b, \hat{b}\in B(\infty)$ satisfy $G^{\mathrm{low}}(b).v_{w\lambda}=G^{\mathrm{low}}(\hat{b}).v_{w\lambda}\neq 0$, then $b=\hat{b}$.
\end{itemize}
\end{proposition}
\begin{proposition}[The action of $E_i, F_i$. {\cite[Proposition 5.3.1, Remark]{Kasglo}$+$\cite[Theorem 7]{Kas_originalcrys}}(integrality)]\label{EFaction}
For $b'\in B(\lambda) (\lambda \in P_+), i\in I, p\in \mathbb{Z}_{\geqq 0}$, we have 
\begin{align*}
\dis E_i^{(p)}G_{\lambda}^{\mathrm{up}}(b')&=\left[ \begin{array}{c} \varepsilon_i^{\lambda}(b')\\p \end{array} \right]_iG_{\lambda}^{\mathrm{up}}(\tilde{e}_i^p b')+\sum_{\substack{b''\in B(\lambda),\\ \varepsilon_i^{\lambda}(b')-p>\varepsilon_i^{\lambda}(b'')}}E_{b', b''}^{(p), i}G_{\lambda}^{\mathrm{up}}(b''),\\
&\text{with\ } E_{b', b''}^{(p), i}\in qq_i^{-p(\varepsilon_i^{\lambda}(b')-p)}\mathbb{Z}[q],\\
\dis F_i^{(p)}G_{\lambda}^{\mathrm{up}}(b')&=\left[ \begin{array}{c} \varphi_i^{\lambda}(b')\\p \end{array} \right]_iG_{\lambda}^{\mathrm{up}}(\tilde{f}_i^p b')+\sum_{\substack{b''\in B(\lambda),\\ \varphi_i^{\lambda}(b')-p>\varphi_i^{\lambda}(b'')}}F_{b', b''}^{(p), i}G_{\lambda}^{\mathrm{up}}(b''),\\
&\text{with\ }F_{b', b''}^{(p), i}\in qq_i^{-p(\varphi_i^{\lambda}(b')-p)}\mathbb{Z}[q],
\end{align*}
and
\begin{align*}
\dis E_i^{(p)}G_{\lambda}^{\mathrm{low}}(b')&=\left[ \begin{array}{c} \varphi_i^{\lambda}(b')+p\\p \end{array} \right]_iG_{\lambda}^{\mathrm{low}}(\tilde{e}_i^p b')+\sum_{\substack{b''\in B(\lambda),\\ \varphi_i^{\lambda}(b')+p<\varphi_i^{\lambda}(b'')}}\widehat{E}_{b', b''}^{(p), i}G_{\lambda}^{\mathrm{low}}(b''),\\
&\text{with\ }\widehat{E}_{b', b''}^{(p), i}\in qq_i^{-p(\varphi_i^{\lambda}(b'')-p)}\mathbb{Z}[q],\\
\dis F_i^{(p)}G_{\lambda}^{\mathrm{low}}(b')&=\left[ \begin{array}{c} \varepsilon_i^{\lambda}(b')+p\\p \end{array} \right]_iG_{\lambda}^{\mathrm{low}}(\tilde{f}_i^p b')+\sum_{\substack{b''\in B(\lambda),\\ \varepsilon_i^{\lambda}(b')+p<\varepsilon_i^{\lambda}(b'')}}\widehat{F}_{b', b''}^{(p), i}G_{\lambda}^{\mathrm{low}}(b''),\\
&\text{with\ }\widehat{F}_{b', b''}^{(p), i}\in qq_i^{-p(\varepsilon_i^{\lambda}(b'')-p)}\mathbb{Z}[q].
\end{align*}
In particular,
\[
\begin{array}{ll}
E_i^{(\varepsilon_i^{\lambda}(b'))}G_{\lambda}^{\mathrm{up}}(b')=G_{\lambda}^{\mathrm{up}}(\tilde{e}_i^{\varepsilon_i^{\lambda}(b')} b'), &\text{and\ } E_i^{(\varepsilon_i^{\lambda}(b')+1)}G_{\lambda}^{\mathrm{up}}(b')=0,\\
F_i^{(\varphi_i^{\lambda}(b'))}G_{\lambda}^{\mathrm{up}}(b')=G_{\lambda}^{\mathrm{up}}(\tilde{f}_i^{\varphi_i^{\lambda}(b')} b'), &\text{and\ } F_i^{(\varphi_i^{\lambda}(b')+1)}G_{\lambda}^{\mathrm{up}}(b')=0.
\end{array}
\]
\end{proposition}
\begin{notation}
We will use the notations $E_{b', b''}^{(p), i}, F_{b', b''}^{(p), i}, \widehat{E}_{b', b''}^{(p), i}$ and $ \widehat{F}_{b', b''}^{(p), i}$ for any $b', b''\in B(\lambda), i\in I$ and $p\in \mathbb{Z}_{\geqq 0}$ in a natural way. (For example, $E_i^{(p)}G_{\lambda}^{\mathrm{low}}(b')=\sum_{b''\in B(\lambda)}\widehat{E}_{b', b''}^{(p), i}G_{\lambda}^{\mathrm{low}}(b'')$.)
\end{notation}
\begin{notation}\label{coeff_notation}
For any $b\in B(\infty), i\in I$ and $p\in \mathbb{Z}_{\geqq 0}$, we write
\begin{align*}
&F_i^{(p)}G^{\mathrm{low}}(b)=\sum_{\tilde{b}\in B(\infty)}c_{-pi, b}^{\tilde{b}}G^{\mathrm{low}}(\tilde{b}),\\
&(_ie')^p(G^{\mathrm{low}}(b))=\sum_{\tilde{b}\in B(\infty)}d_{b, \tilde{b}}^{i, p}G^{\mathrm{low}}(\tilde{b}),\\
&(e'_i)^p(G^{\mathrm{low}}(b))=\sum_{\tilde{b}\in B(\infty)}\widehat{d}_{b, \tilde{b}}^{i, p}G^{\mathrm{low}}(\tilde{b}).
\end{align*}
\end{notation}
\begin{notation}
We will denote by $\rho$ an element of $P_+$ satisfying $\langle \rho, \alpha_i^{\vee} \rangle=1$ for all $i\in I$.
\end{notation}
The following proposition is essentially written in the reference \cite[Proposition 2.2]{Kas_parameter}. We attach its proof for the convenience of the reader.
\begin{proposition}\label{EFaction'}
For $b\in B(\infty)$, $i\in I$ and $p\in \mathbb{Z}_{\geqq 0}$, we have
\begin{itemize}
\item[{\nor (i)}] $\dis F_i^{(p)}G^{\mathrm{low}}(b)=\left[ \begin{array}{c} \varepsilon_i(b)+p\\p \end{array} \right]_iG^{\mathrm{low}}(\tilde{f}_i^pb)+\sum_{\substack{\tilde{b}\in B(\infty)\\ \varepsilon_i(b)+p<\varepsilon_i(\tilde{b})}}c_{-pi, b}^{\tilde{b}}G^{\mathrm{low}}(\tilde{b})$\\
\hspace{70pt}$\text{with}\ c_{-pi, b}^{\tilde{b}}\in qq_i^{-p(\varepsilon_i(\tilde{b})-p)}\mathbb{Z}[q].$\\
\item[{\nor (ii)}] $\dis (e'_i)^pG^{\mathrm{low}}(b)=q_i^{-p\varepsilon_i(b)+\frac{1}{2}p(p+1)}G^{\mathrm{low}}(\tilde{e}_i^pb)+\sum_{\substack{\tilde{b}\in B(\infty)\\ \varepsilon_i(b)-p<\varepsilon_i(\tilde{b})}}\widehat{d}_{b, \tilde{b}}^{i, p}G^{\mathrm{low}}(\tilde{b})$\\
\hspace{70pt} $\text{with}\ \widehat{d}_{b, \tilde{b}}^{i, p}\in qq_i^{-p\varepsilon_i(\tilde{b})-\frac{1}{2}p(p-1)}\mathbb{Z}[q].$
\end{itemize}
\end{proposition}
\begin{proof}
The statement (i) easily follows from Remark \ref{crystalmorph}, Proposition \ref{subcrystal} (v) and \ref{EFaction}. We prove (ii). For $r\in \mathbb{Z}_{>0}$ with $0 \neq G^{\mathrm{low}}(b).v_{r\rho}(\in V(r\rho))$, we have
\begin{align*}
&E_iG^{\mathrm{low}}(b).v_{r\rho}\\
&=\dis \frac{( _ie')(G^{\mathrm{low}}(b))K_i-K_{-i}(e'_i)(G^{\mathrm{low}}(b))}{q_i-q_i^{-1}}.v_{r\rho}\ (\text{cf.\ Definition\ }\ref{qderiv}.)\\
&=\dis \frac{q_i^{\langle\weight b+\alpha_i, \alpha_i^{\vee}\rangle}\overline{(e'_i)(G^{\mathrm{low}}(b))}K_i-K_{-i}(e'_i)(G^{\mathrm{low}}(b))}{q_i-q_i^{-1}}.v_{r\rho}\\
&=\dis \frac{q_i^{r+\langle\weight b+\alpha_i, \alpha_i^{\vee}\rangle}\overline{(e'_i)(G^{\mathrm{low}}(b))}-q_i^{-r-\langle\weight b+\alpha_i, \alpha_i^{\vee}\rangle}(e'_i)(G^{\mathrm{low}}(b))}{q_i-q_i^{-1}}.v_{r\rho}\\
&=\dis \sum_{\tilde{b}\in B(\infty)}\frac{q_i^{r+\langle\weight b+\alpha_i, \alpha_i^{\vee}\rangle}\overline{\widehat{d}_{b, \tilde{b}}^{i, 1}}-q_i^{-r-\langle\weight b+\alpha_i, \alpha_i^{\vee}\rangle}\widehat{d}_{b, \hat{b}}^{i, 1}}{q_i-q_i^{-1}}G_{r\rho}^{\mathrm{low}}(\pi_{r\rho}(\tilde{b})),
\end{align*}
by the equality (\ref{qderiv2}) and Proposition \ref{barinv}. On the other hand, by Proposition \ref{EFaction}, 
\begin{align*}
E_iG^{\mathrm{low}}(b).v_{r\rho}=&[ \varphi_i^{r\rho}(b)+1]_iG_{r\rho}^{\mathrm{low}}(\tilde{e}_i\pi_{r\rho}(b))\\
&+\sum_{\substack{\tilde{b}\in B(\infty)\\ \varphi_i^{r\rho}(b)+1<\varphi_i^{r\rho}(\tilde{b})}}\widehat{E}_{\pi_{r\rho}(b), \pi_{r\rho}(\tilde{b})}^{(1), i}G_{r\rho}^{\mathrm{low}}(\pi_{r\rho}(\tilde{b})),
\end{align*}
with $\widehat{E}_{\pi_{r\rho}(b), \pi_{r\rho}(\tilde{b})}^{(1), i}\in qq_i^{-\varphi_i^{r\rho}(\pi_{r\rho}(\tilde{b}))+1}\mathbb{Z}[q]$. By Proposition \ref{subcrystal} and Remark \ref{crystalmorph}, this equality can be rewritten as follows:
\begin{align*}
E_iG^{\mathrm{low}}(b).v_{r\rho}=&[ \varphi_i(b)+r+1]_iG_{r\rho}^{\mathrm{low}}(\pi_{r\rho}(\tilde{e}_ib))\\
&+\sum_{\substack{\tilde{b}\in B(\infty)\\ \varepsilon_i(b)-1<\varepsilon_i(\tilde{b})}}\widehat{E}_{\pi_{r\rho}(b), \pi_{r\rho}(\hat{b})}^{(1), i}G_{r\rho}^{\mathrm{low}}(\pi_{r\rho}(\tilde{b})),
\end{align*}
with $\widehat{E}_{\pi_{r\rho}(b), \pi_{r\rho}(\tilde{b})}^{(1), i}\in qq_i^{-\varphi_i(\tilde{b})-r+1}\mathbb{Z}[q]$.

Comparing the above equalities, we deduce that there is a sufficiently large positive integer $r$ such that
\begin{itemize}
\item $\pi_{r\rho}(\tilde{b})\neq 0$ for any $\tilde{b}\in B(\infty)_{\weight b+\alpha_i}$, and
\item the degree $<0$ part of the Laurent polynomial $(q_i-q_i^{-1})\widehat{E}_{\pi_{r\rho}(b), \pi_{r\rho}(\tilde{b})}^{(1), i}$ is equal to $-q_i^{-r-\langle\weight b+\alpha_i, \alpha_i^{\vee}\rangle}\widehat{d}_{b, \tilde{b}}^{i, 1}$ for any $\tilde{b}\in B(\infty)$.
\end{itemize}
Therefore, we obtain
$$e'_iG^{\mathrm{low}}(b)=q_i^{-\varepsilon_i(b)+1}G^{\mathrm{low}}(\tilde{e}_ib)+\sum_{\substack{\tilde{b}\in B(\infty)\\ \varepsilon_i(b)-1<\varepsilon_i(\tilde{b})}}\widehat{d}_{b, \tilde{b}}^{i, 1}G^{\mathrm{low}}(\tilde{b}) \text{with}\ \widehat{d}_{b, \tilde{b}}^{i, 1}\in qq_i^{-\varepsilon_i(\tilde{b})}\mathbb{Z}[q].$$
Moreover, we have
\begin{align*}
(e'_i)^p(G^{\mathrm{low}}(b))=\sum_{\tilde{b}\in B(\infty)}\left(\sum_{\substack{ b^{1},\dots, b^{p-1}\in B(\infty)}}\mspace{-30mu}\widehat{d}_{b, b^1}^{i, 1}\widehat{d}_{b^1, b^2}^{i, 1}\cdots \widehat{d}_{b^{p-1}, \tilde{b}}^{i, 1}\right)G^{\mathrm{low}}(\tilde{b}).
\end{align*}
It follows from the above argument that if $\widehat{d}_{b^{s}, b^{s+1}}^{i, 1}\neq 0$ then $$\varepsilon_i(b^{s})-1\leqq \varepsilon_i(b^{s+1}).$$
Therefore, if $\widehat{d}_{b, b^1}^{i, 1}\widehat{d}_{b^1, b^2}^{i, 1}\cdots \widehat{d}_{b^{p-1}, \tilde{b}}^{i, 1}\neq 0$, then
\begin{align*}
\varepsilon_i(\tilde{b})&\geqq \varepsilon_i(b^{p-1})-1\\
&\geqq \varepsilon_i(b^{p-2})-2\\
&\geqq\cdots\geqq \varepsilon_i(b^1)-p+1\geqq \varepsilon_i(b)-p.
\end{align*}
Hence, $\widehat{d}_{b, b^1}^{i, 1}\widehat{d}_{b^1, b^2}^{i, 1}\cdots \widehat{d}_{b^{p-1}, \tilde{b}}^{i, 1}\in q_i^{-\varepsilon_i(\tilde{b})-\varepsilon_i(b^{p-1})-\cdots -\varepsilon_i(b^1)}\mathbb{Z}[q]\subset q_i^{-p\varepsilon_i(\tilde{b})-\frac{1}{2}p(p-1)}\mathbb{Z}[q].$
Combining the above arguments, we obtain
$$\widehat{d}_{b, \tilde{b}}^{i, p}=
\begin{cases}q_i^{-p\varepsilon_i(b)+\frac{1}{2}p(p+1)}& \text{if\ } \tilde{b}=\tilde{e}_i^pb,\\
\in qq_i^{-p\varepsilon_i(\tilde{b})-\frac{1}{2}p(p-1)}\mathbb{Z}[q]& \text{if\ }\varepsilon_i(b)-p<\varepsilon_i(\tilde{b}),\\
0& \text{otherwise.}\end{cases}$$
This completes the proof of (ii).
\end{proof}
\begin{notation}
For a Laurent polynomial $P\in \mathbb{Z}[q^{\pm 1}]$ and an integer $m\in \mathbb{Z}$, the degree $<m$  part of $P$ will be denoted by $P_{<m}$.
\end{notation}
\begin{notation}
Write $\dis \Delta_i:=\frac{(\alpha_i, \alpha_i)}{2} (i\in I)$.
\end{notation}
\begin{proposition}[Similarity of the structure constants]\label{similarity}
For any $b, \hat{b}\in B(\infty), i\in I$ and $N\in \mathbb{Z}_{\geqq 0}$, we have 
$$\left( c_{-Ni, b}^{\hat{b}} \right)_{<-\Delta_i(d-1)N}=\left(q_i^{\frac{1}{2}d(d-1)}\left[ \begin{array}{c} \varepsilon_i(\hat{b})\\N \end{array} \right]_i\widehat{d}_{b, \tilde{e}_i^{\varepsilon_i(\hat{b})}\hat{b}}^{i, d}\right)_{<-\Delta_i(d-1)N,}$$
where $d:=\varepsilon_i(\hat{b})-N$.
\end{proposition}
\begin{proof}
By Proposition \ref{EFaction'}, we have
\begin{align}
(e'_i)^{\varepsilon_i(\hat{b})}(F_i^{(N)}G^{\mathrm{low}}(b))&=\sum_{\substack{\tilde{b}\in B(\infty)\\ \varepsilon_i(b)+N\leqq\varepsilon_i(\tilde{b})}}c_{-Ni, b}^{\tilde{b}}(e'_i)^{\varepsilon_i(\hat{b})}G^{\mathrm{low}}(\tilde{b}) \notag\\
&=\sum_{\substack{\tilde{b}\in B(\infty)\\ \varepsilon_i(b)+N\leqq\varepsilon_i(\tilde{b})}}\sum_{\substack{\tilde{b}'\in B(\infty)\\ \varepsilon_i(\tilde{b})-\varepsilon_i(\hat{b})\leqq\varepsilon_i(\tilde{b}')}}c_{-Ni, b}^{\tilde{b}}\widehat{d}_{\tilde{b}, \tilde{b}'}^{i, \varepsilon_i(\hat{b})}G^{\mathrm{low}}(\tilde{b}'),\label{derivmulti}
\end{align}
and
\begin{align}
c_{-Ni, b}^{\tilde{b}}\widehat{d}_{\tilde{b}, \tilde{b}'}^{i, \varepsilon_i(\hat{b})}\in q_i^{-N(\varepsilon_i(\tilde{b})-N)-\varepsilon_i(\hat{b})\varepsilon_i(\tilde{b}')-\frac{1}{2}\varepsilon_i(\hat{b})(\varepsilon_i(\hat{b})-1)}\mathbb{Z}[q].\label{derivmulti_estimate}
\end{align}
Let $C_0$ be the coefficient of $G^{\mathrm{low}}(\tilde{e}_i^{\varepsilon_i(\hat{b})}\hat{b})(\neq 0)$ in (\ref{derivmulti}). Then,
$$C_0=\sum_{\substack{\tilde{b}\in B(\infty)\\ \varepsilon_i(b)+N\leqq\varepsilon_i(\tilde{b})\leqq\varepsilon_i(\hat{b})}}c_{-Ni, b}^{\tilde{b}}\widehat{d}_{\tilde{b}, \tilde{e}_i^{\varepsilon_i(\hat{b})}\hat{b}}^{i, \varepsilon_i(\hat{b})}.$$
By (\ref{derivmulti_estimate}), we have
$$\sum_{\substack{\tilde{b}\in B(\infty)\\ \varepsilon_i(b)+N\leqq\varepsilon_i(\tilde{b})<\varepsilon_i(\hat{b})}}c_{-Ni, b}^{\tilde{b}}\widehat{d}_{\tilde{b}, \tilde{e}_i^{\varepsilon_i(\hat{b})}\hat{b}}^{i, \varepsilon_i(\hat{b})}\in q_i^{-N(d-1)-\frac{1}{2}\varepsilon_i(\hat{b})(\varepsilon_i(\hat{b})-1)}\mathbb{Z}[q].$$
Moreover, if $\varepsilon_i(\tilde{b})=\varepsilon_i(\hat{b})$ and $\widehat{d}_{\tilde{b}, \tilde{e}_i^{\varepsilon_i(\hat{b})}\hat{b}}^{i, \varepsilon_i(\hat{b})}\neq 0$, then $\tilde{e}_i^{\varepsilon_i(\hat{b})}\tilde{b}=\tilde{e}_i^{\varepsilon_i(\hat{b})}\hat{b}(\neq 0)$ (equivalently, $\tilde{b}=\hat{b}$) by Proposition \ref{EFaction'}. Therefore, we have
$$(C_0)_{<(\bullet)}=\left(q_i^{-\frac{1}{2}\varepsilon_i(\hat{b})(\varepsilon_i(\hat{b})-1)}c_{-Ni, b}^{\hat{b}}\right)_{<(\bullet),}$$
where $\dis(\bullet)=-\Delta_i\left(N(d-1)+\frac{1}{2}\varepsilon_i(\hat{b})(\varepsilon_i(\hat{b})-1)\right).$

On the other hand, 
\begin{align*}
&(e'_i)^{\varepsilon_i(\hat{b})}(F_i^{(N)}G^{\mathrm{low}}(b))\\
&=\sum_{s=0}^Nq_i^{-2\varepsilon_i(\hat{b})N+(\varepsilon_i(\hat{b})+N)s-\frac{1}{2}s(s-1)}\left[ \begin{array}{c} \varepsilon_i(\hat{b})\\s \end{array} \right]_iF_i^{(N-s)}(e'_i)^{\varepsilon_i(\hat{b})-s}G^{\mathrm{low}}(b).
\end{align*}
This follows from a direct computation and this calculation result is written in the reference \cite[(3.1.2)]{Kas_originalcrys}. 
By Proposition \ref{EFaction'}, we have
\begin{align*}
&F_i^{(N-s)}(e'_i)^{\varepsilon_i(\hat{b})-s}G^{\mathrm{low}}(b)\\
&=\sum_{\substack{\tilde{b}\in B(\infty)\\ \varepsilon_i(b)-\varepsilon_i(\hat{b})+s\leqq\varepsilon_i(\tilde{b})}}\widehat{d}_{b, \tilde{b}}^{i, \varepsilon_i(\hat{b})-s}F_i^{(N-s)}G^{\mathrm{low}}(\tilde{b})\\
&=\sum_{\substack{\tilde{b}\in B(\infty)\\ \varepsilon_i(b)-\varepsilon_i(\hat{b})+s\leqq\varepsilon_i(\tilde{b})}}\sum_{\substack{\tilde{b}'\in B(\infty)\\ \varepsilon_i(\tilde{b})+N-s\leqq\varepsilon_i(\tilde{b}')}}\widehat{d}_{b, \tilde{b}}^{i, \varepsilon_i(\hat{b})-s}c_{-(N-s)i, \tilde{b}}^{\tilde{b}'}G^{\mathrm{low}}(\tilde{b}').
\end{align*}
Since we consider the case $\tilde{b}'=\tilde{e}_i^{\varepsilon_i(\hat{b})}\hat{b}$, the inequality $\varepsilon_i(\tilde{b})+N-s\leqq\varepsilon_i(\tilde{b}')(=0)$ does not hold unless $s=N$. Hence, we have
$$C_0=q_i^{-\varepsilon_i(\hat{b})N+\frac{1}{2}N(N+1)}\left[ \begin{array}{c} \varepsilon_i(\hat{b})\\N \end{array} \right]_i\widehat{d}_{b, \tilde{e}_i^{\varepsilon_i(\hat{b})}\hat{b}}^{i, d}.$$
Therefore,
$$\left(q_i^{-\frac{1}{2}\varepsilon_i(\hat{b})(\varepsilon_i(\hat{b})-1)}c_{-Ni, b}^{\hat{b}}\right)_{<(\bullet)}=\left(q_i^{-\varepsilon_i(\hat{b})N+\frac{1}{2}N(N+1)}\left[ \begin{array}{c} \varepsilon_i(\hat{b})\\N \end{array} \right]_i\widehat{d}_{b, \tilde{e}_i^{\varepsilon_i(\hat{b})}\hat{b}}^{i, d}\right)_{<(\bullet).}$$
Hence, we obtain the equality
$$\left( c_{-Ni, b}^{\hat{b}} \right)_{<-\Delta_i(d-1)N}=\left(q_i^{\frac{1}{2}d(d-1)}\left[ \begin{array}{c} \varepsilon_i(\hat{b})\\N \end{array} \right]_i\widehat{d}_{b, \tilde{e}_i^{\varepsilon_i(\hat{b})}\hat{b}}^{i, d}\right)_{<-\Delta_i(d-1)N.}$$
\end{proof}
\begin{remark}\label{d_and_dhat}
The coefficients $d_{b, \tilde{b}}^{i, p}$ and $\widehat{d}_{b, \tilde{b}}^{i, p}$ are related to each other.
 
Since $e'_i=\left.\ast\circ {_ie'}\circ\ast\right|_{U_q(\mathfrak{n}^-)}$, we have 
\begin{center}
$\widehat{d}_{b, \tilde{b}}^{i, p}=d_{\ast b, \ast\tilde{b}}^{i, p}$ for all $b, \tilde{b}\in B(\infty), i\in I$ and $p\in \mathbb{Z}_{\geqq 0}.$
\end{center}
In particular, $d_{b, \tilde{b}}^{i, p}\in q_i^{-p\varepsilon_i^{\ast}(\tilde{b})-\frac{1}{2}p(p-1)}\mathbb{Z}[q]$.

Using the equality (\ref{qderiv2}) in Definition \ref{qderiv} repeatedly, we obtain 
\begin{center}
$\widehat{d}_{b, \tilde{b}}^{i, p}=q_i^{p\langle \weight b,\alpha_i^{\vee}\rangle+p(p+1)}\overline{d_{b, \tilde{b}}^{i, p}}$ for all $b, \tilde{b}\in B(\infty), i\in I$ and $p\in \mathbb{Z}_{\geqq 0}$. 
\end{center}
\end{remark}
If $\mathfrak{g}$ is a symmetric Kac-Moody Lie algebra (=the Lie algebra corresponding to a symmetric generalized Cartan matrix), then the canonical bases have some positivities. 
\begin{proposition}[{\cite[Theorem 11.5]{Lus_quiper}}]\label{positivty_1}
Assume that the Lie algebra $\mathfrak{g}$ is a symmetric Kac-Moody Lie algebra. Then, for any $b, \tilde{b}\in B(\infty), i\in I$ and $p\in \mathbb{Z}_{\geqq 0}$, we have
$$c_{-pi, b}^{\tilde{b}}, d_{b, \tilde{b}}^{i, p}, \widehat{d}_{b, \tilde{b}}^{i, p} \in \mathbb{N}[q^{\pm 1}].$$
\end{proposition}
\section{The transition matrices from Canonical bases to PBW bases}\label{main}
Again, we assume that $\mathfrak{g}$ is a finite dimensional complex simple Lie algebra.
\begin{notation}
For a reduced word ${\bf i}$ of $w_0$, we set
$$F_{{\bf i}}^{{\bf c}}:=F_{i_1}^{(c_1)}T''_{i_1, 1}(F_{i_2}^{(c_2)})\cdots T''_{i_1, 1}T''_{i_2, 1}\cdots T''_{i_{l(w_0)-1}, 1}(F_{i_{l(w_0)}}^{(c_{l(w_0)})})=\omega(E_{{\bf i}}^{{\bf c}}).$$
\end{notation}
\subsection{The main theorem}
The following is one of the main conclusions of this paper.
\begin{theorem}[Positivity]\label{maintheorem}
Assume that the Lie algebra $\mathfrak{g}$ is of type $ADE$. Take an arbitrary reduced word ${\bf i}$ of $w_0$. Then, for any $b\in B(\infty)$, we have
$$G^{\mathrm{low}}(b)=\sum_{{\bf d}\in (\mathbb{Z}_{\geqq 0})^{l(w_0)}}{_{{\bf i}}\zeta}_{{\bf d}}^{G^{\mathrm{low}}(b)}F_{{\bf i}}^{{\bf d}}\ \text{with\ } {_{{\bf i}}\zeta}_{{\bf d}}^{G^{\mathrm{low}}(b)}\in \mathbb{N}[q^{\pm 1}].$$
\end{theorem}
\begin{remark}
This theorem was originally proved by Lusztig in his original paper of the canonical bases (\cite[Corollary 10.7]{Lus_canari1}) in the case when PBW bases are associated with the adapted reduced words of $w_0$ (See \cite[4.7]{Lus_canari1}.), through his geometric realization of the elements of the canonical bases and PBW bases. Recently, this fact for arbitrary PBW bases was proved by Kato (\cite[Theorem 4.17]{KatoKLR}) and subsequently by McNamara (\cite{McNamara_finite}), through the categorification of (dual) PBW bases by using the Khovanov-Lauda-Rouquier algebras. We give another algebraic proof from now on.

For general nonsymmetric finite types, McNamara also established the categorification of dual PBW bases (\cite[Theorem 3.1]{McNamara_finite}). However, the ``canonical basis'' arising from the Khovanov-Lauda-Rouquier categorification does not coincide with the canonical basis in Section \ref{dab2} and, in fact, the positivity fails in general. 
\end{remark}
\begin{remark}
Although we assume that the Lie algebra $\mathfrak{g}$ is of type $ADE$ in the statement of Theorem \ref{maintheorem}, we never put this assumption in the following calculation. This assumption will be used only when we check the positivity of the calculation results. (See Theorem \ref{maincalc}.)
\end{remark}
\subsection{Calculations}\label{prfidea}
Fix an element $G^{\mathrm{low}}(b_0)$ of the canonical basis of $U_q(\mathfrak{n}^-)$ and an expression
\begin{center}
$\dis G^{\mathrm{low}}(b_0)=\sum_{\substack{j_1,\dots, j_l\in I\\ -\sum_{k=1}^{l}\alpha_{j_k}=\weight b_0}}\eta_{j_1,\dots,j_l}^{(0)}F_{j_1}\cdots F_{j_l}$ with $\eta_{j_1,\dots,j_l}^{(0)}\in \mathbb{Q}(q)$. 
\end{center}
(The expression of this form is not unique.) 

Fix an arbitrary reduced word ${\bf i}$ of $w_0$. 

For any $j_1,\dots, j_k \in I$, we write
\begin{center}
$\dis E_{j_1}\cdots E_{j_k}=\sum_{{\bf d}\in (\mathbb{Z}_{\geqq 0})^{l(w_0)}}m_{{\bf d}}^{j_1,\dots ,j_k}E_{{\bf i}}^{{\bf d}}$ with $m_{ {\bf d}}^{j_1,\dots ,j_k}\in \mathbb{Z}[q^{\pm 1}]$.
\end{center}
\begin{definition}[The large number $L_{(b_0, \{ \eta_{j_1,\dots,j_l}^{(0)} \}, {\bf i})}$]\label{largenumber}
Let 
\begin{align*}
&A:= Q_- \bigcap {\big (}(\cup_{w\in W} w\Set{\beta \in Q_- | \weight b_0\leq \beta \leq 0})+Q_+{\big )},\\
&h_A:=\max\Set{-\height\beta|\beta\in A}.
\end{align*}
We define the ``large'' number $L_{(b_0, \{ \eta_{j_1,\dots,j_l}^{(0)} \}, {\bf i})}$ ($=:L$ for short) by $L=l_0+l_1+l_2+l_3+l_4+l_5+1$, where
\begin{itemize}
\item $l_0 :=$ the minimum of the elements $l$ of $\mathbb{Z}_{\geqq 0}$ such that ${_{{\bf i}}\zeta}_{{\bf d}}^{G^{\mathrm{low}}(b_0)}\in q^l\mathbb{Z}[q^{-1}]$ for any ${\bf d} \in (\mathbb{Z}_{\geqq 0})^{l(w_0)}$.
\item $l_1 :=$ the minimum of the elements $l$ of $\mathbb{Z}_{\geqq 0}$ such that $\eta_{j_1,\dots,j_l}^{(0)}\in q^{-l}\mathcal{A}_0$ for all $j_1,\dots, j_l\in I$ with $-\sum_{k=1}^{l}\alpha_{j_k}=\weight b_0$.
\item $l_2 :=$ the minimum of the elements $l$ of $\mathbb{Z}_{\geqq 0}$ such that $m_{ {\bf d}}^{j_1,\dots ,j_k}\in q^{-l}\mathbb{Z}[q]$ for any $j_1,\dots, j_l\in I$ with $-\sum_{k=1}^{l}\alpha_{j_k}=\weight b_0$ and any ${\bf d} \in (\mathbb{Z}_{\geqq 0})^{l(w_0)}$.
\item $l_3 :=$ the minimum of the elements $l$ of $\mathbb{Z}_{\geqq 0}$ such that $d_{b, b'}^{j, p}, \widehat{d}_{b, b'}^{j, p}\in q^{-l}\mathbb{Z}[q]$ for any $j\in I$, $p\in \mathbb{Z}_{\geqq 0}$ and $b, b'\in B(\infty)$ with $\weight b\in A$.
\item $l_4:=\max\Set{ h_A\vert (\alpha, \beta)\vert | \alpha\in \Delta_+, \beta \in A }$.
\item $\dis l_5:= \frac{(\alpha, \alpha)}{2}\left(\frac{3}{2}(\height (\weight b_0))^2+\frac{1}{2}\height (\weight b_0)\right)$ for a long root $\alpha\in \Delta$.
\end{itemize}
Moreover, we put $\gamma:=l(w_0)+1$ and $\lambda_0:= 2\gamma L\rho(\in P_+)$. 
\end{definition}
\begin{remark}
The definition of the large number $L$ seems to be rather complicated, but it is nothing serious. In fact, the following method for proving Theorem \ref{maintheorem} is valid as long as we take sufficiently large positive integer as $L$. (cf. Corollary \ref{converge}.) We merely fix such a number explicitly. It is probably possible to take a smaller positive integer as $L$ than the one in Definition \ref{largenumber}.
\end{remark}
Recall the notation in Section \ref{isom} and the equality (\ref{act}). Then, we have 
\begin{align*}
&\adr_{\lambda_0}(\varphi(G^{\mathrm{low}}(b_0)))K_{\lambda_0}\\
&=\dis \sum_{j_1,\dots, j_l\in I}\eta_{j_1,\dots,j_l}^{(0)}\adr_{\lambda_0}(E_{j_l}\cdots E_{j_1})K_{\lambda_0}\\
&=\sum_{j_1,\dots, j_l\in I}\eta_{j_1,\dots,j_l}^{(0)}G_{j_1}\cdots G_{j_l}K_{\lambda_0}+q^{4\gamma L -l_1-l_4}\sum_{j'_1,\dots, j'_l\in I}\eta_{j'_1,\dots,j'_l}^{(1)}G_{j'_1}\cdots G_{j'_l}K_{\lambda_0},
\end{align*}
for some $\eta_{j'_1,\dots,j'_l}^{(1)} \in \mathcal{A}_0$. Note that $2\gamma L\leqq (\alpha_i, \alpha_i)\gamma L$ for all $i\in I$. Hence, by Theorem \ref{maintool}, we get 
\begin{align*}
&\Upsilon(\adr_{\lambda_0}(\varphi(G^{\mathrm{low}}(b_0)))K_{\lambda_0}).\Phi_{\mathrm{KOY}}^{-1}(|(0) \rangle_{\bf i})\\
&=\sum_{j_1,\dots, j_l\in I}\eta_{j_1,\dots,j_l}^{(0)}E_{j_1}\cdots E_{j_l}+q^{4\gamma L -l_1-l_4}\sum_{j'_1,\dots, j'_l\in I}\eta_{j'_1,\dots,j'_l}^{(1)}E_{j'_1}\cdots E_{j'_l}\\
&=\omega (G^{\mathrm{low}}(b_0))+q^{4\gamma L -l_1-l_2-l_4}\sum_{{\bf d}'\in (\mathbb{Z}_{\geqq 0})^{l(w_0)}}\eta_{{\bf d}'}^{(2)}E_{{\bf i}}^{{\bf d}'}\\
&=\omega\left(\sum_{{\bf d}\in (\mathbb{Z}_{\geqq 0})^{l(w_0)}}{_{{\bf i}}\zeta}_{{\bf d}}^{G^{\mathrm{low}}(b_0)}F_{{\bf i}}^{{\bf d}}\right)+q^{4\gamma L -l_1-l_2-l_4}\sum_{{\bf d}'\in (\mathbb{Z}_{\geqq 0})^{l(w_0)}}\eta_{{\bf d}'}^{(2)}E_{{\bf i}}^{{\bf d}'}\\
&=\sum_{{\bf d}\in (\mathbb{Z}_{\geqq 0})^{l(w_0)}}{_{{\bf i}}\zeta}_{{\bf d}}^{G^{\mathrm{low}}(b_0)}E_{{\bf i}}^{{\bf d}}+q^{4\gamma L -l_1-l_2-l_4}\sum_{{\bf d}'\in (\mathbb{Z}_{\geqq 0})^{l(w_0)}}\eta_{{\bf d}'}^{(2)}E_{{\bf i}}^{{\bf d}'}
\end{align*}
for some $\eta_{{\bf d}'}^{(2)} \in \mathcal{A}_0$. Note that ${_{{\bf i}}\zeta}_{{\bf d}}^{G^{\mathrm{low}}(b_0)}\in q^{l_0}\mathbb{Z}[q^{-1}]$ and $4\gamma L -l_1-l_2-l_4>4\gamma L-L(>L)$. On the other hand, 
\begin{align*}
&\Upsilon(\adr_{\lambda_0}(\varphi(G^{\mathrm{low}}(b_0)))K_{\lambda_0}).\Phi_{\mathrm{KOY}}^{-1}(|(0) \rangle_{\bf i})\\
&=\left(c_{f_{\lambda_0}, v_{w_0\lambda_0}}^{\lambda_0}.\varphi(G^{\mathrm{low}}(b_0))\right).\Phi_{\mathrm{KOY}}^{-1}(|(0) \rangle_{\bf i})\\
&=\Phi_{\mathrm{KOY}}^{-1}\left(\left(c_{f_{\lambda_0}, v_{w_0\lambda_0}}^{\lambda_0}.\varphi(G^{\mathrm{low}}(b_0))\right).|(0) \rangle_{\bf i}\right).
\end{align*}
Set 
\begin{align}\label{zetaprime}
\left(c_{f_{\lambda_0}, v_{w_0\lambda_0}}^{\lambda_0}.\varphi(G^{\mathrm{low}}(b_0))\right).|(0) \rangle_{\bf i}=\sum_{{\bf d}\in (\mathbb{Z}_{\geqq 0})^{l(w_0)}}{_{{\bf i}}\zeta'}_{{\bf d}}^{G^{\mathrm{low}}(b_0)}|({\bf d})\rangle_{{\bf i}},
\end{align}
for some ${_{{\bf i}}\zeta'}_{{\bf d}}^{G^{\mathrm{low}}(b_0)}\in \mathbb{Q}(q)$. Then, we have
\begin{center}
$\dis \Phi_{\mathrm{KOY}}^{-1}\left(\left(c_{f_{\lambda_0}, v_{w_0\lambda_0}}^{\lambda_0}.\varphi(G^{\mathrm{low}}(b_0))\right).|(0) \rangle_{\bf i}\right)
=\sum_{{\bf d}\in (\mathbb{Z}_{\geqq 0})^{l(w_0)}}{_{{\bf i}}\zeta'}_{{\bf d}}^{G^{\mathrm{low}}(b_0)}E_{{\bf i}}^{{\bf d}}$.
\end{center}
These results give the equality ${_{{\bf i}}\zeta'}_{{\bf d}}^{G^{\mathrm{low}}(b_0)}={_{{\bf i}}\zeta}_{{\bf d}}^{G^{\mathrm{low}}(b_0)}+q^{4\gamma L -l_1-l_2-l_4}\eta_{{\bf d}}^{(2)} ({\bf d}\in (\mathbb{Z}_{\geqq 0})^{l(w_0)})$. Hence, the proof of Theorem \ref{maintheorem} is completed by showing that the principal part of the Laurent expansion at 0 of $q^{-L}{_{{\bf i}}\zeta'}_{{\bf d}}^{G^{\mathrm{low}}(b_0)}$ belongs to $\mathbb{N}[q^{-1}]$ for any ${\bf d}\in (\mathbb{Z}_{\geqq 0})^{l(w_0)}$ when the Lie algebra $\mathfrak{g}$ is of type $ADE$. Hence, from now on, we compute $\left(c_{f_{\lambda_0}, v_{w_0\lambda_0}}^{\lambda_0}.\varphi(G^{\mathrm{low}}(b_0))\right).|(0) \rangle_{\bf i}$. 

Since $f_{\lambda_0}=(v_{\lambda_0},\cdot)_{\lambda_0}\in V(\lambda_0)^{\ast}$ (See Definition \ref{crystalsinrepresentation}.), we obtain
\begin{align*}
c_{f_{\lambda_0}, v_{w_0\lambda_0}}^{\lambda_0}.\varphi(G^{\mathrm{low}}(b_0))&=(v_{\lambda_0}, \varphi(G^{\mathrm{low}}(b_0))(\cdot ).v_{w_0\lambda_0})_{\lambda_0}\\
&=(G^{\mathrm{low}}(b_0).v_{\lambda_0}, (\cdot ).v_{w_0\lambda_0})_{\lambda_0}\\
&=(G_{\lambda_0}^{\mathrm{low}}(\pi_{\lambda_0}(b_0)), (\cdot ).v_{w_0\lambda_0})_{\lambda_0} \  \text{by\ Proposition\ }\ref{subcrystal} \text{(v)},\\
&=c_{(G_{\lambda_0}^{\mathrm{low}}(\pi_{\lambda_0}(b_0)), \cdot)_{\lambda_0}, v_{w_0\lambda_0}}^{\lambda_0},
\end{align*}
in $\mathbb{Q}_q[G]$. By the way, for $f\in V(\lambda_0)^{\ast}, v\in V(\lambda_0)$, 
\begin{center}
$\dis \Delta(c_{f, v}^{\lambda_0})=\sum_{b'\in B(\lambda_0)}c_{f, G_{\lambda_0}^{\mathrm{up}}(b')}^{\lambda_0}\otimes c_{(G_{\lambda_0}^{\mathrm{low}}(b'), \cdot)_{\lambda_0}, v}^{\lambda_0}$.
\end{center}

Therefore, we have 
\begin{align}\label{shiki}
c_{(G_{\lambda_0}^{\mathrm{low}}(\pi_{\lambda_0}(b_0)), \cdot)_{\lambda_0}, v_{w_0\lambda_0}}^{\lambda_0}.&|(0) \rangle_{{\bf i}} \notag\\
=\sum_{b'_1,\dots, b'_{l(w_0)-1}\in B(\lambda_0)}&c_{(G_{\lambda_0}^{\mathrm{low}}(\pi_{\lambda_0}(b_0)), \cdot )_{\lambda_0}, G_{\lambda_0}^{\mathrm{up}}(b'_1)}^{\lambda_0}.|0 \rangle_{i_1}\otimes c_{(G_{\lambda_0}^{\mathrm{low}}(b'_1), \cdot )_{\lambda_0}, G_{\lambda_0}^{\mathrm{up}}(b'_2)}^{\lambda_0}.|0 \rangle_{i_2}\notag \\
&\hspace{20pt}\otimes\dots\otimes c_{(G_{\lambda_0}^{\mathrm{low}}(b'_{l(w_0)-1}), \cdot)_{\lambda_0}, v_{w_0\lambda_0}}^{\lambda_0}.|0 \rangle_{i_{l(w_0)}}.
\end{align}
Note that $v_{w_0\lambda_0}=G_{\lambda_0}^{\mathrm{up}}(b'_{l(w_0)})$ with $b'_{l(w_0)}\in B(\lambda)_{w_0\lambda_0}$. (Such a $b'_{l(w_0)}$ is uniquely determined.) We set $b'_0=\pi_{\lambda_0}(b_0)\in B(\lambda_0)$.
\subsubsection{A method of calculation}
We first give a method of calculating the terms of the form
$$c_{(G_{\lambda_0}^{\mathrm{low}}(b'_{k-1}), \cdot)_{\lambda_0}, G_{\lambda_0}^{\mathrm{up}}(b'_k)}^{\lambda_0}.|0 \rangle_{i_k}.$$
(For later use, we attach some additional condition.) The main result of this subsection is Lemma \ref{method}.

\begin{lemma}\label{sl2case}
Let $\lambda\in P_+$. If $c_{f, v}^{\lambda}.|0 \rangle_i\neq 0$ for some weight vectors $f\in V(\lambda)^{\ast}, v\in V(\lambda)$ and $i\in I$, then
\begin{center}
$\weight f-\weight v\in \mathbb{Z}\alpha_i$ and $\langle \weight f+\weight v, \alpha_i^{\vee}\rangle\leqq 0$.
\end{center}
\end{lemma}
\begin{proof}
By definition, $c_{f, v}^{\lambda}.|0 \rangle_i=\phi_i^{\ast}(c_{f, v}^{\lambda}).|0 \rangle_i$. (See Definition \ref{vsigma}.) So, the first part of the statement clearly holds. We can write
$$\dis \phi_i^{\ast}(c_{f, v}^{\lambda})=\sum_{m_1, m_2, m_3, m_4\in \mathbb{Z}_{\geqq 0}}a_{m_1, m_2, m_3, m_4}c_{22}^{m_1}c_{21}^{m_2}c_{12}^{m_3}c_{11}^{m_4}$$ with $a_{m_1, m_2, m_3, m_4}\in \mathbb{Q}(q)$. Then, we have 
\begin{align*}
\langle \weight f, \alpha_i^{\vee}\rangle &=-(m_1+m_2)+m_3+m_4\ \text{and},\\
\langle \weight v, \alpha_i^{\vee}\rangle &=-(m_1+m_3)+m_2+m_4 ,
\end{align*}
for any $a_{m_1, m_2, m_3, m_4}\neq 0$. Therefore, if $\langle \weight f+\weight v, \alpha_i^{\vee}\rangle >0$, then $-2m_1+2m_4>0$, in particular, $m_4>0$ for any $a_{m_1, m_2, m_3, m_4}\neq 0$. So, $c_{f, v}^{\lambda}.|0 \rangle_i=0$ by Definition \ref{vsigma}. This contradicts our assumption.
\end{proof}

In the proof of the following lemma, we use the following proposition. This is a direct consequence of \cite[Proposition 31.2.6]{Lusbook}.
\begin{proposition}\label{divide}
Let $\lambda_1, \lambda_2, \lambda_3\in P_+$. Set
\[
\begin{array}{c}
\tilde{V}:=\Set{v\in V(\lambda_3)_{\lambda_2-\lambda_1}|\begin{array}{l}E_i^{(m)}.v=0 \text{\footnotesize\ for\ all\ }m>\langle\lambda_1, \alpha_i^{\vee}\rangle \text{\footnotesize\ and\ } i\in I, \text{\footnotesize\ and}\\
F_i^{(m)}.v=0 \text{\footnotesize\ for\ all\ }m>\langle\lambda_2, \alpha_i^{\vee}\rangle \text{\footnotesize\ and\ }i\in I\end{array}}.
\end{array}
\]
Then, the map
\begin{center}
$\Hom_{U_q(\mathfrak{g})}(V(-w_0\lambda_1)\otimes V(\lambda_2), V(\lambda_3))\to \tilde{V}, f\mapsto f(v_{-\lambda_1}\otimes v_{\lambda_2})$
\end{center}
is a $\mathbb{Q}(q)$-linear isomorphism.
\end{proposition}

\begin{lemma}[A method of calculation]\label{method}
Let $\lambda\in P_+$, $w\in W$ and $i\in I$ with $l(s_iw)>l(w)$. Take $b'_+, b'_-\in B(\lambda)$ with $\weight b'_+ -\weight b'_-=n\alpha_i$ for some $n\in \mathbb{Z}$.

Assume that $G_{\lambda}^{\mathrm{low}}(b'_+), G_{\lambda}^{\mathrm{low}}(b'_-)\in V_w(\lambda)$. Then, by Proposition \ref{Demazure} (ii), (iii), there uniquely exist $b_+,b_- \in B(\infty)$ such that $G^{\mathrm{low}}(b_{\pm}).v_{w\lambda}=G_{\lambda}^{\mathrm{low}}(b'_{\pm})$. Set $\dis \lambda_1= \varepsilon_i^{\lambda}(b'_-)\varpi_i+\sum_{j\in I\setminus \{i\}}\varepsilon_j\varpi_j$ with $\varepsilon_j^{\lambda}(b'_-)\leqq \varepsilon_j\in \mathbb{Z}_{\geqq 0}$ for all $j\in I\setminus \{i\}$.

Then, we have
\begin{align*}
&c_{(G_{\lambda}^{\mathrm{low}}(b'_{+}), \cdot)_{\lambda}, G_{\lambda}^{\mathrm{up}}(b'_{-})}^{\lambda}.|0\rangle_i\\
&=\dis (-q_i)^{\varphi_i^{\lambda}(b'_-)}\prod_{s=1}^{-\frac{1}{2}\langle \weight b'_+ +\weight b'_-, \alpha_i^{\vee}\rangle}\mspace{-25mu}(1-q_i^{2s})\sum_{t=0}^{\varphi_i^{\lambda}(b'_-)}(-1)^tq_i^{-\varphi_t^2-t+\varphi_t\langle w\lambda, \alpha_i^{\vee}\rangle+N(\varepsilon_i^{\lambda}(b'_-)-N)}\mspace{-8mu}\\
&\hspace{15pt}\times(G_{\lambda_1}^{\mathrm{up}}((\pi_{\lambda_1}(\ast b_-)), E_i^{(t)}\ast(G_{\varphi_t})F_i^{(N)}.v_{\lambda_1})_{\lambda_1}\left|-\frac{1}{2}\langle \weight b'_+ +\weight b'_-, \alpha_i^{\vee}\rangle\right>_i,
\end{align*}
where $\varphi_t:=\varphi_i^{\lambda}(b'_{-})-t$, $N:=n+\varphi_i^{\lambda}(b'_{-})$ and $G_s:=q_i^{\frac{1}{2}s(s-1)}(_ie')^s(G^{\mathrm{low}}(b_{+}))$.
\end{lemma}
\begin{remark}
If $\dis \phi_i^{\ast}{\big (}c_{(G_{\lambda}^{\mathrm{low}}(b'_+), \cdot)_{\lambda}, G_{\lambda}^{\mathrm{up}}(b'_-)}^{\lambda}{\big )}\neq 0$, then
$$\dis (G_{\lambda}^{\mathrm{low}}(b'_+), E_i^{(a)}F_i^{(b)}G_{\lambda}^{\mathrm{up}}(b'_-))_{\lambda}\neq 0,$$
for some $a, b\in \mathbb{Z}_{\geqq 0}$ with $a-b=n$. (in particular $b\geqq -n$). Since $F_i^{(b)}G_{\lambda}^{\mathrm{up}}(b'_-)=0$ when $b>\varphi_i^{\lambda}(b'_-)$, we have $\varphi_i^{\lambda}(b'_-)\geqq -n$, that is, $N\geqq 0$ in this case.
\end{remark}
\begin{proof}
Set $\dis \lambda_2=\varphi_i^{\lambda}(b'_-)\varpi_i+\sum_{j\in I\setminus \{i\}}\varphi_j\varpi_j$, where $\varphi_j:=\varepsilon_j+\langle \weight b'_-, \alpha_j^{\vee}\rangle(\geqq \varphi_j^{\lambda}(b'_-))$. Since $E_j^{(m)}G_{\lambda}^{\mathrm{up}}(b'_-)=0$ for all $m>\varepsilon_j^{\lambda}(b'_-)$ and $F_j^{(m)}G_{\lambda}^{\mathrm{up}}(b'_-)=0$ for all $m>\varphi_j^{\lambda}(b'_-)$ by Proposition \ref{EFaction} for all $j\in I$, there exists a surjective homomorphism $\Psi: V(-w_0\lambda_1)\otimes V(\lambda_2)\to V(\lambda)$ of left $U_q(\mathfrak{g})$-modules uniquely determined by $$v_{-\lambda_1}\otimes v_{\lambda_2}\mapsto G_{\lambda}^{\mathrm{up}}(b'_-),$$ by Proposition \ref{divide}. (Surjectivity follows from the irreducibility of $V(\lambda)$.) Then, the dual injective $\mathbb{Q}(q)$-linear map $\Psi^{\ast}: V(\lambda)^{\ast}\to (V(-w_0\lambda_1)\otimes V(\lambda_2))^{\ast}$ given by $f\mapsto f\circ \Psi$ is a homomorphism of right $U_q(\mathfrak{g})$-modules. 

Moreover, $\xi:V(-w_0\lambda_1)^{\ast}\otimes V(\lambda_2)^{\ast}\to (V(-w_0\lambda_1)\otimes V(\lambda_2))^{\ast}$ given by $f_1\otimes f_2\mapsto (v_1\otimes v_2\mapsto \left<f_1, v_1\right>\left<f_2, v_2\right>)$ is an isomorphism of right $U_q(\mathfrak{g})$-modules. Note that $V(-w_0\lambda_1)$ and $V(\lambda_2)$ are finite dimensional.

Set the composition map
$$\iota: V(\lambda)^{\ast}\xrightarrow[ ]{\Psi^{\ast}} (V(-w_0\lambda_1)\otimes V(\lambda_2))^{\ast}\xrightarrow[ ]{\xi^{-1}}V(-w_0\lambda_1)^{\ast}\otimes V(\lambda_2)^{\ast},$$
and write $\iota ((G_{\lambda}^{\mathrm{low}}(b'_+), \cdot)_{\lambda})=\sum_{l}f_l^{(1)}\otimes f_l^{(2)}$ for some weight vectors $f_l^{(1)}, f_l^{(2)}$. Then, we have
$$\dis c_{(G_{\lambda}^{\mathrm{low}}(b'_+), \cdot)_{\lambda}, G_{\lambda}^{\mathrm{up}}(b'_-)}^{\lambda}=\sum_{l}c_{f_l^{(1)}, v_{-\lambda_1}}^{-w_0\lambda_1}c_{f_l^{(2)}, v_{\lambda_2}}^{\lambda_2}.$$
By Proposition \ref{hw} (ii), we have only to consider $l$ with $\weight f_l^{(2)}=s_i\lambda_2=\lambda_2-\varphi_i^{\lambda}(b'_-)\alpha_i$. Since $\dim V(\lambda_2)_{s_i\lambda_2}=1$, we may assume that such a $l$ (denoted by $l^{(0)}$) is uniquely determined and $f_{l^{(0)}}^{(2)}=f_{\lambda_2}.E_i^{(\varphi_i^{\lambda}(b'_-))}$. Therefore, we investigate the element $f_{l^{(0)}}^{(1)}\in V(-w_0\lambda_1)^{\ast}$.

Now, we have the equality
\begin{align*}
(G_{\lambda}^{\mathrm{low}}(b'_+), \cdot)_{\lambda}&=(G^{\mathrm{low}}(b_+).v_{w\lambda}, \cdot)_{\lambda}\\
&=(v_{w\lambda}, \varphi(G^{\mathrm{low}}(b_+))(\cdot))_{\lambda}=f_{w\lambda}.\varphi(G^{\mathrm{low}}(b_+)).
\end{align*}
Therefore, $\iota ((G_{\lambda}^{\mathrm{low}}(b'_+), \cdot)_{\lambda})=\iota (f_{w\lambda}).\varphi(G^{\mathrm{low}}(b_+))$. Set $$\iota (f_{w\lambda})=\sum_{l'} {f'_{l'}}^{(1)}\otimes {f'_{l'}}^{(2)}$$ for some weight vectors ${f'_{l'}}^{(1)}, {f'_{l'}}^{(2)}$. We may assume that there exists a unique $l'$ (denoted by ${l'}^{(0)}$) such that $\weight {f'_{l'}}^{(2)}=\lambda_2$, and also assume that ${f'_{{l'}^{(0)}}}^{\mspace{-25mu}(2)}=f_{\lambda_2}$.
\begin{claim}\label{firstclaim}
The element ${f'_{{l'}^{(0)}}}^{\mspace{-25mu}(1)}$ is given by $$\varphi(G^{\mathrm{low}}(b)).v_{-\lambda_1}\mapsto \delta_{b, b_-}$$ for $b\in B(\infty)$.
\end{claim}
\begin{proof}[Proof of Claim \ref{firstclaim}]
The element ${f'_{{l'}^{(0)}}}^{\mspace{-25mu}(1)}$ is a weight vector of $V(-w_0\lambda_1)^{\ast}$ of weight $w\lambda-\lambda_2$. Since $-\lambda_1+\lambda_2=\weight b'_-=w\lambda+\weight b_-$, we have $w\lambda-\lambda_2=-\lambda_1-\weight b_-$. Hence, $\langle {f'_{{l'}^{(0)}}}^{\mspace{-25mu}(1)}, v\rangle=0$ for any weight vector $v\in V(-w_0\lambda_1)$ with $\weight v\neq -\lambda_1-\weight b_-$. 

The set $\{ \varphi(G^{\mathrm{low}}(b)).v_{-\lambda_1} \}_{b\in B(\infty)_{\weight b_-}}\setminus \{ 0\}$ forms a basis of $V(-w_0\lambda_1)_{-\lambda_1-\weight b_-}$. So, we consider the images of the elements of this basis under ${f'_{{l'}^{(0)}}}^{\mspace{-25mu}(1)}$.

We claim that $\varphi(G^{\mathrm{low}}(b)).G_{\lambda}^{\mathrm{up}}(b'_-)=\delta_{b, b_-}v_{w\lambda}$ for $b\in B(\infty)_{\weight b_-}$. Indeed, we have $\weight \left(\varphi(G^{\mathrm{low}}(b)).G_{\lambda}^{\mathrm{up}}(b'_-)\right)=w\lambda$ because $\weight b=\weight b_-$. So, we can write $\varphi(G^{\mathrm{low}}(b)).G_{\lambda}^{\mathrm{up}}(b'_-)=c_bv_{w\lambda}$ for some $c_b\in \mathbb{Q}(q)$, and
\begin{center}
$c_b=(v_{w\lambda}, \varphi(G^{\mathrm{low}}(b)).G_{\lambda}^{\mathrm{up}}(b'_-))_{\lambda}=(G^{\mathrm{low}}(b).v_{w\lambda}, G_{\lambda}^{\mathrm{up}}(b'_-))_{\lambda}=\delta_{b, b_-}$,
\end{center}
by Proposition \ref{Demazure} (ii), (iii). 

Hence, for $b\in B(\infty)_{\weight b_-}$, we have
\begin{align*}
\langle{f'_{{l'}^{(0)}}}^{\mspace{-25mu}(1)}, \varphi(G^{\mathrm{low}}(b)).v_{-\lambda_1}\rangle&=\langle{f'_{{l'}^{(0)}}}^{\mspace{-25mu}(1)}, \varphi(G^{\mathrm{low}}(b)).v_{-\lambda_1}\rangle\langle f_{\lambda_2}, v_{\lambda_2}\rangle\\
&=\left<\xi (\sum_{l'} {f'_{l'}}^{(1)}\otimes {f'_{l'}}^{(2)}), (\varphi(G^{\mathrm{low}}(b)).v_{-\lambda_1})\otimes v_{\lambda_2}\right>\\
&=\langle\Psi^{\ast} (f_{w\lambda}), \varphi(G^{\mathrm{low}}(b)).(v_{-\lambda_1}\otimes v_{\lambda_2})\rangle\\
&=\langle f_{w\lambda}, \varphi(G^{\mathrm{low}}(b))G_{\lambda}^{\mathrm{up}}(b'_-)\rangle\\
&=\langle f_{w\lambda}, \delta_{b, b_-}v_{w\lambda}\rangle=\delta_{b, b_-}.
\end{align*}
(In particular, $\varphi(G^{\mathrm{low}}(b_-)).v_{-\lambda_1}\neq 0$.)
\end{proof}
Since $l(s_iw)>l(w)$, we have $\iota (f_{w\lambda}).F_i=\iota (f_{w\lambda}.F_i)=0$. Therefore, we have
\begin{center}
$\dis \iota (f_{w\lambda})=\mspace{-5mu}\sum_{t=0}^{\langle\lambda_2, \alpha_i^{\vee} \rangle(=\varphi_i^{\lambda}(b'_-))}\mspace{-30mu}X_t\otimes f_{\lambda_2}.E_i^{(t)}+\mspace{-20mu}\sum_{\substack{l'\\ \weight {f'_{l'}}^{(2)}\notin \lambda_2+\mathbb{Z}_{\leqq 0}\alpha_i}}\mspace{-20mu}{f'_{l'}}^{(1)}\otimes {f'_{l'}}^{(2)},$
\end{center}
where
\begin{center}
$\dis X_t=(-1)^tq_i^{-t\varphi_i^{\lambda}(b'_-)+t(t-1)}\left[ \begin{array}{c} \varphi_i^{\lambda}(b'_-)\\t \end{array} \right]_i^{-1}{f'_{{l'}^{(0)}}}^{\mspace{-25mu}(1)}.F_i^{(t)}$,
\end{center}
which is easily checked by induction on $t$.
 
Let $$\dis \Delta (G^{\mathrm{low}}(b_+))=\sum_{s=0}^{\infty}G_s\otimes K_{\weight G_s}F_i^{(s)}+\sum_{\substack{G^{(2)}:\text{homogeneous}\\ \weight G^{(2)}\notin \mathbb{Z}_{\leqq 0}\alpha_i}}G^{(1)}\otimes G^{(2)}.$$ By the statement in Definition \ref{qderiv}, this definition of $G_s$ is compatible with the definition of $G_s$ in the statement of Lemma \ref{method}. We prepare one claim here whose proof is straightforward.
\begin{claim}\label{automrel}
For all homogeneous element $x\in U_q(\mathfrak{n}^-)$, we have
$$\varphi(x)=q^{-\frac{1}{2}(\weight x, \weight x)+(\rho, \weight x)}K_{-\weight x}\omega'(x).$$
(See Definition \ref{automorphisms}.)
\end{claim}
By Definition \ref{automorphisms} and Claim \ref{automrel}, 
\begin{align*}
\Delta(\varphi(G^{\mathrm{low}}(b_+)))&=q^{(\ast)}\Delta(K_{-\weight b_+}\omega'(G^{\mathrm{low}}(b_+)))\\
&=q^{(\ast)}(K_{-\weight b_+}\otimes K_{-\weight b_+})(\omega'\otimes\omega')(\Delta(G^{\mathrm{low}}(b_+)))\\
&=q^{(\ast)}\left(\sum_{s=0}^{\infty}K_{-\weight b_+}\omega'(G_s)\otimes K_{-\weight b_+}\omega'(K_{\weight G_s}F_i^{(s)})\right.\\
&\hspace{40pt}\left.+\sum_{\substack{G^{(2)}:\text{homogeneous}\\ \weight G^{(2)}\notin \mathbb{Z}_{\leqq 0}\alpha_i}}\mspace{-15mu}K_{-\weight b_+}\omega'(G^{(1)})\otimes K_{-\weight b_+}\omega'(G^{(2)})\right)\\
&=\sum_{s=0}^{\infty}K_{s\alpha_i}\varphi(G_s)\otimes E_i^{(s)}+\sum_{\substack{{G'}^{(2)}:\text{homogeneous}\\ \weight {G'}^{(2)}\notin \mathbb{Z}_{\geqq 0}\alpha_i}}\mspace{-15mu}{G'}^{(1)}\otimes {G'}^{(2)},
\end{align*}
where $\dis (\ast)=-\frac{1}{2}(\weight b_+, \weight b_+)+(\rho, \weight b_+)$.
Combining the above results, we have
\begin{align*}
\iota ((G_{\lambda}^{\mathrm{low}}(b'_+), \cdot)_{\lambda})&=\iota (f_{w\lambda}).\varphi(G^{\mathrm{low}}(b_+))\\
&=\left(\sum_{t=0}^{\varphi_i^{\lambda}(b'_-)}X_t\otimes f_{\lambda_2}.E_i^{(t)}+\sum_{\substack{l'\\ \weight {f'_{l'}}^{(2)}\notin \lambda_2+\mathbb{Z}_{\leqq 0}\alpha_i}} {f'_{l'}}^{(1)}\otimes {f'_{l'}}^{(2)}\right)\\
&\hspace{35pt}.\left(\sum_{s=0}^{\infty}K_{s\alpha_i}\varphi(G_s)\otimes E_i^{(s)}+\sum_{\substack{{G'}^{(2)}:\text{homogeneous}\\ \weight {G'}^{(2)}\notin \mathbb{Z}_{\geqq 0}\alpha_i}}\mspace{-15mu}{G'}^{(1)}\otimes {G'}^{(2)}\right).
\end{align*}
Therefore, we obtain 
\begin{align*}
f_{l^{(0)}}^{(1)}&=\sum_{t=0}^{\varphi_i^{\lambda}(b'_-)}\left[ \begin{array}{c} \varphi_i^{\lambda}(b'_-)\\t \end{array} \right]_iX_tK_{\varphi_t\alpha_i}\varphi(G_{\varphi_t})\\
&=\sum_{t=0}^{\varphi_i^{\lambda}(b'_-)}(-1)^tq_i^{-t\varphi_i^{\lambda}(b'_-)+t(t-1)+\varphi_t\langle\weight {f'_{{l'}^{(0)}}}^{\mspace{-25mu}(1)}+t\alpha_i, \alpha_i^{\vee}\rangle}{f'_{{l'}^{(0)}}}^{\mspace{-25mu}(1)}.F_i^{(t)}\varphi(G_{\varphi_t})\\
&=\sum_{t=0}^{\varphi_i^{\lambda}(b'_-)}(-1)^tq_i^{-\varphi_t^2-t+\varphi_t\langle w\lambda, \alpha_i^{\vee}\rangle}{f'_{{l'}^{(0)}}}^{\mspace{-25mu}(1)}.F_i^{(t)}\varphi(G_{\varphi_t}),
\end{align*}
where $\varphi_t:=\varphi_i^{\lambda}(b'_-)-t$. Note that $\weight {f'_{{l'}^{(0)}}}^{\mspace{-25mu}(1)}=w\lambda-\lambda_2$. Set $Y_t:={f'_{{l'}^{(0)}}}^{\mspace{-25mu}(1)}.F_i^{(t)}\varphi(G_{\varphi_t})$. 

By the way, $\dis \phi_i^{\ast}\left( c_{f_{l^{(0)}}^{(2)}, v_{\lambda_2}}^{\lambda_2}\right)=c_{21}^{\varphi_i^{\lambda}(b'_-)}$. Hence, 
\begin{align*}
&c_{(G_{\lambda}^{\mathrm{low}}(b'_+), \cdot)_{\lambda}, G_{\lambda}^{\mathrm{up}}(b'_-)}^{\lambda}.|0\rangle_i\\
&=\left(\sum_{t=0}^{\varphi_i^{\lambda}(b'_-)}(-1)^tq_i^{-\varphi_t^2-t+\varphi_t\langle w\lambda, \alpha_i^{\vee}\rangle}c_{Y_t, v_{-\lambda_1}}^{-w_0\lambda_1}\right)c_{f_{l^{(0)}}^{(2)}, v_{\lambda_2}}^{\lambda_2}.|0\rangle_i\\
&=(-q_i)^{\varphi_i^{\lambda}(b'_-)}\sum_{t=0}^{\varphi_i^{\lambda}(b'_-)}(-1)^tq_i^{-\varphi_t^2-t+\varphi_t\langle w\lambda, \alpha_i^{\vee}\rangle}c_{Y_t, v_{-\lambda_1}}^{-w_0\lambda_1}.|0\rangle_i.
\end{align*}
Fix $t\in \{0,\dots, \varphi_i^{\lambda}(b'_-) \}$. Since $Y_t\in V(-w_0\lambda_1)^{\ast}$ and
\begin{align*}
\weight Y_t&=\weight {f'_{{l'}^{(0)}}}^{\mspace{-25mu}(1)}+t\alpha_i+\weight G_{\varphi_t}\\
&=(-\lambda_1-\weight b_-)+t\alpha_i+(\weight b_++\varphi_t\alpha_i)=-\lambda_1+(n+\varphi_i^{\lambda}(b'_-))\alpha_i,
\end{align*}
we have $Y_t=c_tf_{-\lambda_1}.F_i^{(N)}$ for some $c_t\in \mathbb{Q}(q)$. By the way,
\begin{align*}
&\langle {f'_{{l'}^{(0)}}}^{\mspace{-25mu}(1)}, F_i^{(t)}\varphi(G_{\varphi_t})E_i^{(N)}.v_{-\lambda_1}\rangle\\
&=\langle Y_t, E_i^{(N)}.v_{-\lambda_1}\rangle\\
&=c_t\langle f_{-\lambda_1}, F_i^{(N)}E_i^{(N)}.v_{-\lambda_1}\rangle=\dis c_t\left[ \begin{array}{c} \varepsilon_i^{\lambda}(b'_-)\\N \end{array} \right]_i\langle f_{-\lambda_1}, v_{-\lambda_1}\rangle=c_t\left[ \begin{array}{c} \varepsilon_i^{\lambda}(b'_-)\\N \end{array} \right]_i.
\end{align*} 
Note that this equality is valid even when $\varepsilon_i^{\lambda}(b'_-)<N$. Moreover, $f_{-\lambda_1}.F_i^{(N)}=f_{-\lambda_1}.F_i^{(\varepsilon_i^{\lambda}(b'_-))}E_i^{(\varepsilon_i^{\lambda}(b'_-)-N)}$. Hence, $$\phi_i^{\ast}\left( c_{f_{-\lambda_1}.F_i^{(N)}, v_{-\lambda_1}}^{-w_0\lambda_1}\right)=c_{f_{\varepsilon_i^{\lambda}(b'_-)}, v_{-\varepsilon_i^{\lambda}(b'_-)}}^{\varepsilon_i^{\lambda}(b'_-)}.E^{(\varepsilon_i^{\lambda}(b'_-)-N)}.$$ (We used the notation in  Example \ref{sl2ex}.) Therefore, by Example \ref{sl2ex}, we obtain
\begin{align*}
&c_{Y_t, v_{-\lambda_1}}^{-w_0\lambda_1}.|0\rangle_i\\
&=c_t\phi_i^{\ast}\left( c_{f_{-\lambda_1}.F_i^{(N)}, v_{-\lambda_1}}^{-w_0\lambda_1}\right) .|0\rangle_i\\
&=c_tq_i^{N(\varepsilon_i^{\lambda}(b'_-)-N)}\left[ \begin{array}{c} \varepsilon_i^{\lambda}(b'_-)\\N \end{array} \right]_i\mspace{-8mu}\prod_{s=1}^{\varepsilon_i^{\lambda}(b'_-)-N}\mspace{-8mu}(1-q_i^{2s})\ket{\varepsilon_i^{\lambda}(b'_-)-N}_i\\
&=q_i^{N(\varepsilon_i^{\lambda}(b'_-)-N)}\langle {f'_{{l'}^{(0)}}}^{\mspace{-25mu}(1)}, F_i^{(t)}\varphi(G_{\varphi_t})E_i^{(N)}.v_{-\lambda_1}\rangle\mspace{-8mu}\prod_{s=1}^{\varepsilon_i^{\lambda}(b'_-)-N}\mspace{-8mu}(1-q_i^{2s})\ket{\varepsilon_i^{\lambda}(b'_-)-N}_i.
\end{align*}
From now on, we examine $\langle {f'_{{l'}^{(0)}}}^{\mspace{-25mu}(1)}, F_i^{(t)}\varphi(G_{\varphi_t})E_i^{(N)}.v_{-\lambda_1}\rangle$.
\begin{claim}\label{identify}
$c_{{f'_{{l'}^{(0)}}}^{\mspace{-25mu}(1)}, v_{-\lambda_1}}^{-w_0\lambda_1}=(G_{\lambda_1}^{\mathrm{up}}((\pi_{\lambda_1}(\ast b_-)),\omega(\cdot ).v_{\lambda_1})_{\lambda_1}$ in $\mathbb{Q}_q[G]$.
\end{claim}
\begin{proof}[Proof of Claim \ref{identify}]
Let $V(\lambda_1)^{\omega}$ be a left $U_q(\mathfrak{g})$-module such that $V(\lambda_1)^{\omega}=V(\lambda_1)$ as a vactor space and the $U_q(\mathfrak{g})$-module structure on $V(\lambda_1)^{\omega}$ is defined by $x.v:=\omega(x).v$ for any $x\in U_q(\mathfrak{g})$ and $v\in  V(\lambda_1)^{\omega}$ (where the right-hand side is defined by the usual action of $U_q(\mathfrak{g})$ on $V(\lambda_1)$ via the identification of the underlying vector spaces $V(\lambda_1)^{\omega}=V(\lambda_1)$).

Then, there exists a $U_q(\mathfrak{g})$-module isomorphism $\Phi: V(-w_0\lambda_1)\to V(\lambda_1)^{\omega}$ given by $v_{-\lambda_1}\mapsto v_{\lambda_1}$. Note that the vector $v_{\lambda_1}$ is of weight $-\lambda_1$ in $V(\lambda_1)^{\omega}$. Moreover, we have $(u, v)_{-w_0\lambda_1}=(\Phi(u), \Phi(v))_{\lambda_1}$ for any $u, v\in V(-w_0\lambda_1)$, because $(\Phi(v_{-\lambda_1}), \Phi(v_{-\lambda_1}))_{\lambda_1}=(v_{\lambda_1}, v_{\lambda_1})_{\lambda_1}=1(=(v_{-\lambda_1}, v_{-\lambda_1})_{-w_0\lambda_1})$ and 
\begin{align*}
(\Phi(x.u), \Phi(v))_{\lambda_1}&=(\omega(x).\Phi(u), \Phi(v))_{\lambda_1}\\
&=(\Phi(u), \varphi(\omega(x)).\Phi(v))_{\lambda_1}\\
&=(\Phi(u), \omega(\varphi(x)).\Phi(v))_{\lambda_1}=(\Phi(u), \Phi(\varphi (x).v))_{\lambda_1}. 
\end{align*}
Let $\tilde{G}$ be the unique element of $V(-w_0\lambda_1)$ satisfying ${f'_{{l'}^{(0)}}}^{\mspace{-25mu}(1)}=(\tilde{G},\ \cdot\ )_{-w_0\lambda_1}\in V(-w_0\lambda_1)^{\ast}$. Then, by the above argument, ${f'_{{l'}^{(0)}}}^{\mspace{-25mu}(1)}=(\Phi(\tilde{G}), \Phi(\cdot))_{\lambda_1}$. By the way, 
\begin{align*}
\Phi(\varphi(G^{\mathrm{low}}(b)).v_{-\lambda_1})&=\omega(\varphi(G^{\mathrm{low}}(b))).\Phi(v_{-\lambda_1})\\
&=\ast(G^{\mathrm{low}}(b)).v_{\lambda_1}=G^{\mathrm{low}}(\ast b).v_{\lambda_1}=G_{\lambda_1}^{\mathrm{low}}(\pi_{\lambda_1}(\ast b)).
\end{align*}
Combining this with Claim \ref{firstclaim},
$$(\Phi(\tilde{G}), G_{\lambda_1}^{\mathrm{low}}(\pi_{\lambda_1}(\ast b)))_{\lambda_1}=\langle {f'_{{l'}^{(0)}}}^{\mspace{-25mu}(1)}, \varphi(G^{\mathrm{low}}(b)).v_{-\lambda_1}\rangle=\delta_{b, b_-},$$
for $b\in B(\infty)$. So, $\Phi(\tilde{G})=G_{\lambda_1}^{\mathrm{up}}(\pi_{\lambda_1}(\ast b_-))$. Therefore, 
\begin{center}
$c_{{f'_{{l'}^{(0)}}}^{\mspace{-25mu}(1)}, v_{-\lambda_1}}^{-w_0\lambda_1}=(G_{\lambda_1}^{\mathrm{up}}(\pi_{\lambda_1}(\ast b_-)), \Phi((\cdot).v_{-\lambda_1}))_{\lambda_1}=(G_{\lambda_1}^{\mathrm{up}}(\pi_{\lambda_1}(\ast b_-)), \omega(\cdot).v_{\lambda_1})_{\lambda_1}$.
\end{center}
\end{proof}

By Claim \ref{identify},
\begin{align*}
\langle {f'_{{l'}^{(0)}}}^{\mspace{-25mu}(1)}, F_i^{(t)}\varphi(G_{\varphi_t})E_i^{(N)}.v_{-\lambda_1}\rangle&=(G_{\lambda_1}^{\mathrm{up}}((\pi_{\lambda_1}(\ast b_-)),\omega(F_i^{(t)}\varphi(G_{\varphi_t})E_i^{(N)}).v_{\lambda_1})_{\lambda_1}\\
&=(G_{\lambda_1}^{\mathrm{up}}((\pi_{\lambda_1}(\ast b_-)), E_i^{(t)}\ast(G_{\varphi_t})F_i^{(N)}.v_{\lambda_1})_{\lambda_1}.
\end{align*}
Combining all the above results, we obtain
\begin{align*}
&c_{(G_{\lambda}^{\mathrm{low}}(b'_+), \cdot)_{\lambda}, G_{\lambda}^{\mathrm{up}}(b'_-)}^{\lambda}.|0\rangle_i\\
&=(-q_i)^{\varphi_i^{\lambda}(b'_-)}\sum_{t=0}^{\varphi_i^{\lambda}(b'_-)}(-1)^tq_i^{(\ast\ast)}\mspace{-8mu}\prod_{s=1}^{\varepsilon_i^{\lambda}(b'_-)-N}\mspace{-8mu}(1-q_i^{2s})\\
&\hspace{15pt}\times\langle {f'_{{l'}^{(0)}}}^{\mspace{-25mu}(1)}, F_i^{(t)}\varphi(G_{\varphi_t})E_i^{(N)}.v_{-\lambda_1}\rangle\ket{\varepsilon_i^{\lambda}(b'_-)-N}_i.\\
&=(-q_i)^{\varphi_i^{\lambda}(b'_-)}\sum_{t=0}^{\varphi_i^{\lambda}(b'_-)}(-1)^tq_i^{(\ast\ast)}\mspace{-8mu}\prod_{s=1}^{\varepsilon_i^{\lambda}(b'_-)-N}\mspace{-8mu}(1-q_i^{2s})\\
&\hspace{15pt}\times(G_{\lambda_1}^{\mathrm{up}}((\pi_{\lambda_1}(\ast b_-)), E_i^{(t)}\ast(G_{\varphi_t})F_i^{(N)}.v_{\lambda_1})_{\lambda_1}|\varepsilon_i^{\lambda}(b'_-)-N\rangle_i.
\end{align*}
where $(\ast\ast)=-\varphi_t^2-t+\varphi_t\langle w\lambda, \alpha_i^{\vee}\rangle+N(\varepsilon_i^{\lambda}(b'_-)-N)$. Moreover,
\begin{align*}
\varepsilon_i^{\lambda}(b'_-)-N&=\varepsilon_i^{\lambda}(b'_-)-\varphi_i^{\lambda}(b'_-)-n\\
&=\dis -\langle\weight b'_-, \alpha_i^{\vee}\rangle-\frac{1}{2}\langle \weight b'_+-\weight b'_-, \alpha_i^{\vee}\rangle\\
&=-\frac{1}{2}\langle \weight b'_++\weight b'_-, \alpha_i^{\vee}\rangle.
\end{align*}
So, Lemma \ref{method} follows.
\end{proof}
\subsubsection{Rough estimates}
In this subsection, we roughly compute the terms 
$$c_{(G_{\lambda_0}^{\mathrm{low}}(b'_{k-1}), \cdot)_{\lambda_0}, G_{\lambda_0}^{\mathrm{up}}(b'_k)}^{\lambda_0}.|0 \rangle_{i_k}  (k=1,\dots, l(w_0))$$
appearing in the right-hand side of (\ref{shiki}). The main result of this subsection is Lemma \ref{roughestlem}.
\begin{lemma}\label{numest}
In any non-zero summand of the right-hand side of (\ref{shiki}), we have

\begin{itemize}
\item[{\nor (i)}] $\weight b'_k-s_{i_k}\cdots s_{i_1}\lambda_0\in A$ for $k=0,\dots, l(w_0)$,
\item[{\nor (ii)}] $\weight b'_{k-1}-\weight b'_k=n_k\alpha_{i_k}$ with $n_k>\gamma L$ for $k=1,\dots, l(w_0)$,
\item[{\nor (iii)}] $\varphi_{i_k}^{\lambda_0}(b'_k)\leqq -\height (\weight b_0)$ for $k=1,\dots, l(w_0)$.
\end{itemize}
\end{lemma}
\begin{proof}
Let us first prove (i). By Proposition \ref{hw} (II), $\weight b'_k-s_{i_k}\cdots s_{i_1}\lambda_0\in Q_-$ for all $k=0,\dots, l(w_0)$. (Note that $\weight b'_{l(w_0)}-s_{i_{l(w_0)}}\cdots s_{i_1}\lambda_0=0$.) So, it suffices to show that $$\weight b_0\leq s_{i_1}\cdots s_{i_k}(\weight b'_k)-\lambda_0\leq 0$$ for all $k=0,\dots, l(w_0)$. Since $s_{i_1}\cdots s_{i_k}(\weight b'_k)$ is a weight of $V(\lambda_0)$, the right inequality clearly holds. We prove the left one. The proof is by induction on $k$. When $k=0$, it is obvious by definition. Assuming it holds for $k$, we will prove it for $k+1$. By Lemma \ref{sl2case}, we can write $\weight b'_k=\weight b'_{k+1}+n_{k+1}\alpha_{i_{k+1}}$ for some $n_{k+1}\in \mathbb{Z}$ and $\langle \weight b'_k+\weight b'_{k+1}, \alpha_{i_{k+1}}^{\vee}\rangle \leqq 0$. Hence, $\langle \weight b'_k, \alpha_{i_{k+1}}^{\vee}\rangle \leqq n_{k+1}$. Write $m_{k+1}:=n_{k+1}-\langle \weight b'_k, \alpha_{i_{k+1}}^{\vee}\rangle\in \mathbb{Z}_{\geqq 0}$. Then, 
\begin{align}
\weight b'_{k+1}&=\weight b'_k-(m_{k+1}+\langle \weight b'_k, \alpha_{i_{k+1}}^{\vee}\rangle )\alpha_{i_{k+1}}\notag \\
&=s_{i_{k+1}}(\weight b'_k+m_{k+1}\alpha_{i_{k+1}}).\label{inductionb'_k}
\end{align}
Combining this with the induction hypothesis, we obtain 
\begin{align*}
s_{i_1}\cdots s_{i_k}s_{i_{k+1}}(\weight b'_{k+1})-\lambda_0&=s_{i_1}\cdots s_{i_k}(\weight b'_k+m_{k+1}\alpha_{i_{k+1}})-\lambda_0\\
&=s_{i_1}\cdots s_{i_k}(\weight b'_k)-\lambda_0+m_{k+1}s_{i_1}\cdots s_{i_k}(\alpha_{i_{k+1}})\\
&\geq \weight b_0. \ (\text{because\ } s_{i_1}\cdots s_{i_k}(\alpha_{i_{k+1}}) \in \Delta_+.)
\end{align*}
This completes the proof of (i).

By the above argument, for all $k=1,\dots, l(w_0)$, we have
\begin{align*}
n_k&\geqq \langle \weight b'_{k-1}, \alpha_{i_k}^{\vee}\rangle\\
&=\langle \weight b'_{k-1}-s_{i_{k-1}}\cdots s_{i_1}\lambda_0, \alpha_{i_k}^{\vee}\rangle+\langle s_{i_{k-1}}\cdots s_{i_1}\lambda_0, \alpha_{i_k}^{\vee}\rangle\\
&\geqq -l_4+\langle \lambda_0, s_{i_1}\cdots s_{i_{k-1}}\alpha_{i_k}^{\vee}\rangle\ \text{by\ (i),}\\
&\geqq -l_4+2\gamma L\ \text{because\ } s_{i_1}\cdots s_{i_{k-1}}\alpha_{i_k}^{\vee} \text{is\ a\ positive\ coroot},\\
&>\gamma L,
\end{align*}
which proves (ii). 

Let $k\in \{1,\dots, l(w_0)\}$. By Proposition \ref{EFaction}, we have $F_{i_k}^{(\varphi_{i_k}^{\lambda_0}(b'_k))}G_{\lambda_0}^{\mathrm{up}}(b'_k)=G_{\lambda_0}^{\mathrm{up}}((\tilde{f}_{i_k})^{\varphi_{i_k}^{\lambda_0}(b'_k)} b'_k)\neq 0$. Hence, $\weight b'_k-\varphi_{i_k}^{\lambda_0}(b'_k)\alpha_{i_k}$ is a weight of $V(\lambda_0)$. Therefore,
\begin{align*}
&s_{i_1}\cdots s_{i_k}(\weight b'_k-\varphi_{i_k}^{\lambda_0}(b'_k)\alpha_{i_k})\\
&=s_{i_1}\cdots s_{i_{k-1}}(\weight b'_{k-1})+(m_k+\varphi_{i_k}^{\lambda_0}(b'_k))s_{i_1}\cdots s_{i_{k-1}}(\alpha_{i_k}),
\end{align*}
by the equality (\ref{inductionb'_k}), is also a weight of $V(\lambda_0)$. Therefore, $$s_{i_1}\cdots s_{i_{k-1}}(\weight b'_{k-1})+(m_k+\varphi_{i_k}^{\lambda_0}(b'_k))s_{i_1}\cdots s_{i_{k-1}}(\alpha_{i_k})-\lambda_0 \in Q_-.$$ Combining this with $\weight b_0\leq s_{i_1}\cdots s_{i_{k-1}}(\weight b'_{k-1})-\lambda_0$, we have $$\weight b_0+(m_k+\varphi_{i_k}^{\lambda_0}(b'_k))s_{i_1}\cdots s_{i_{k-1}}(\alpha_{i_k}) \in Q_-.$$ Now, we have $m_k\geqq 0$ and $s_{i_1}\cdots s_{i_{k-1}}(\alpha_{i_k})\in \Delta_+$, which gives (iii).
\end{proof}
\begin{lemma}[Rough estimates]\label{roughestlem}
In any non-zero summand of the right-hand side of (\ref{shiki}) and $k=1,\dots, l(w_0)$, we have
$$c_{(G_{\lambda_0}^{\mathrm{low}}(b'_{k-1}), \cdot)_{\lambda_0}, G_{\lambda_0}^{\mathrm{up}}(b'_k)}^{\lambda_0}.|0 \rangle_{i_k}=p_k\left| -\frac{1}{2}\langle \weight b'_{k-1}+\weight b'_k, \alpha_{i_k}^{\vee}\rangle \right>_{i_k}$$
with $p_k\in q_{i_k}^{-\varphi_{i_k}^{\lambda_0}(b'_k)^2}\mathbb{Z}[q]$. 
\end{lemma}
\begin{remark}\label{roughestremark}
By Lemma \ref{numest} (iii),
$$-\frac{(\alpha_{i_k}, \alpha_{i_k})}{2}\varphi_{i_k}^{\lambda_0}(b'_k)^2\geqq -\frac{(\alpha_{i_k}, \alpha_{i_k})}{2}\height (\weight b_0)^2\geqq -l_5>-L.$$
\end{remark}
\begin{proof}
Fix $k\in \{1,\dots, l(w_0)\}$. We abbreviate $i_k$ to $i$, $b'_{k-1}$ to $b'_+$, $b'_k$ to $b'_-$, $n_k$ to $n$. Let $b_{\pm}$ be unique elements of $B(\infty)$ satisfying $\pi_{\lambda_0}(b_{\pm})=b'_{\pm}$. Set $\dis d=-\frac{1}{2}\langle \weight b'_+ +\weight b'_-, \alpha_i^{\vee}\rangle$.

Then, by Lemma \ref{method} associated with $w=e$, we have
\begin{align*}
&c_{(G_{\lambda_0}^{\mathrm{low}}(b'_+), \cdot)_{\lambda_0}, G_{\lambda_0}^{\mathrm{up}}(b'_-)}^{\lambda_0}.|0\rangle_i\\
&=\dis (-q_i)^{\varphi_i^{\lambda_0}(b'_-)}\prod_{s=1}^{d}(1-q_i^{2s})\sum_{t=0}^{\varphi_i^{\lambda_0}(b'_-)}(-1)^tq_i^{-\varphi_t^2-t+2\gamma L\varphi_t+N(\varepsilon_i^{\lambda_0}(b'_-)-N)}\\
&\hspace{40pt}\times(G_{\lambda_1}^{\mathrm{up}}((\pi_{\lambda_1}(\ast b_-)), E_i^{(t)}\ast(G_{\varphi_t})F_i^{(N)}.v_{\lambda_1})_{\lambda_1}|d\rangle_i.
\end{align*}
where $\dis \lambda_1= \varepsilon_i^{\lambda_0}(b'_-)\varpi_i+\sum_{j\in I\setminus \{i\}}\varepsilon_j\varpi_j$ with $\varepsilon_j^{\lambda_0}(b'_-)\leqq \varepsilon_j\in \mathbb{Z}_{\geqq 0}$ for all $j\in I\setminus \{i\}$, $\varphi_t:=\varphi_i^{\lambda_0}(b'_-)-t$, $N:=n+\varphi_i^{\lambda_0}(b'_-)$ and $G_s:=q_i^{\frac{1}{2}s(s-1)}( _ie')^s(G^{\mathrm{low}}(b_+))$. Now, 
\begin{align*}
&(G_{\lambda_1}^{\mathrm{up}}((\pi_{\lambda_1}(\ast b_-)), E_i^{(t)}\ast(G_{\varphi_t})F_i^{(N)}.v_{\lambda_1})_{\lambda_1}\\
&=(\omega(G_{\varphi_t})F_i^{(t)}G_{\lambda_1}^{\mathrm{up}}((\pi_{\lambda_1}(\ast b_-)),F_i^{(N)}.v_{\lambda_1})_{\lambda_1}\\
&=q_i^{\frac{1}{2}\varphi_t(\varphi_t-1)}\sum_{b\in B(\infty)}d_{b_+, b}^{i, \varphi_t}(\omega(G^{\mathrm{low}}(b))F_i^{(t)}G_{\lambda_1}^{\mathrm{up}}((\pi_{\lambda_1}(\ast b_-)),F_i^{(N)}.v_{\lambda_1})_{\lambda_1}.
\end{align*}
By Proposition \ref{c1=epsiloni1}, $\omega(G^{\mathrm{low}}(b))\in U_q(\mathfrak{n}^+)E_i^{(\varepsilon_i^{\ast}(b))}$. Hence, in the above sum, we have only to consider $b\in B(\infty)$ such that
$$\varepsilon_i^{\ast}(b)\leqq \varepsilon_i^{\lambda_1}(\pi_{\lambda_1}(\ast b_-))+t=\varepsilon_i^{\ast}(b_-)+t.$$ 
(See Remark \ref{crystalmorph} and Proposition \ref{EFaction}.) Moreover, since $G^{\mathrm{low}}(b_-).v_{\lambda_0}=G_{\lambda_0}^{\mathrm{low}}(b'_-)\neq 0$, we have $\varepsilon_i^{\ast}(b_-)\leqq 2\gamma L$. 

Therefore,
\begin{align*}
&(G_{\lambda_1}^{\mathrm{up}}((\pi_{\lambda_1}(\ast b_-)), E_i^{(t)}\ast(G_{\varphi_t})F_i^{(N)}.v_{\lambda_1})_{\lambda_1}\\
&=q_i^{\frac{1}{2}\varphi_t(\varphi_t-1)}\sum_{\substack{b\in B(\infty),\\ \varepsilon_i^{\ast}(b)\leqq 2\gamma L+t}}d_{b_+, b}^{i, \varphi_t}(\omega(G^{\mathrm{low}}(b))F_i^{(t)}G_{\lambda_1}^{\mathrm{up}}(\pi_{\lambda_1}(\ast b_-)),F_i^{(N)}.v_{\lambda_1})_{\lambda_1}\\
&=q_i^{\frac{1}{2}\varphi_t(\varphi_t-1)}\sum_{\substack{b\in B(\infty),\\ \varepsilon_i^{\ast}(b)\leqq 2\gamma L+t}}d_{b_+, b}^{i, \varphi_t}(F_i^{(t)}G_{\lambda_1}^{\mathrm{up}}(\pi_{\lambda_1}(\ast b_-)),\ast (F_i^{(N)}G^{\mathrm{low}}(b)).v_{\lambda_1})_{\lambda_1}\\
&=q_i^{\frac{1}{2}\varphi_t(\varphi_t-1)}\mspace{-40mu}\sum_{\substack{b\in B(\infty), b'\in B(\lambda_1),\\ \varepsilon_i^{\ast}(b)\leqq 2\gamma L+t,\\ \varphi_i^{\lambda_1}(\pi_{\lambda_1}(\ast b_-))-t\geqq\varphi_i^{\lambda_1}(b')}}\mspace{-40mu}d_{b_+, b}^{i, \varphi_t}F_{\pi_{\lambda_1}(\ast b_-), b'}^{(t), i}(G_{\lambda_1}^{\mathrm{up}}(b'), \ast (F_i^{(N)}G^{\mathrm{low}}(b)).v_{\lambda_1})_{\lambda_1}.
\end{align*}
Since $\ast(G^{\mathrm{low}}(\hat{b}))\in U_q(\mathfrak{n}^-)F_i^{(\varepsilon_i(\hat{b}))}$, we have
\begin{align*}
\ast (F_i^{(N)}G^{\mathrm{low}}(b)).v_{\lambda_1}&=\sum_{\substack{\hat{b}\in B(\infty)\\ \varepsilon_i(\hat{b})\leqq \langle\lambda_1, \alpha_i^{\vee}\rangle}}c_{-Ni, b}^{\hat{b}}\ast (G^{\mathrm{low}}(\hat{b})).v_{\lambda_1}\\
&=\sum_{\substack{\hat{b}\in B(\infty)\\ \varepsilon_i(\hat{b})\leqq \varepsilon_i^{\lambda_0}(b'_-)}}c_{-Ni, b}^{\hat{b}}G_{\lambda_1}^{\mathrm{low}}(\pi_{\lambda_1}(\ast \hat{b})).
\end{align*}
Combining all the results above, we obtain
\begin{align}
&c_{(G_{\lambda_0}^{\mathrm{low}}(b'_+), \cdot)_{\lambda_0}, G_{\lambda_0}^{\mathrm{up}}(b'_-)}^{\lambda_0}.|0 \rangle_i \notag\\
&=(-q_i)^{\varphi_i^{\lambda_0}(b'_-)}\prod_{s=1}^{d}(1-q_i^{2s}) \notag\\
&\hspace{40pt}\times\sum_{t=0}^{\varphi_i^{\lambda_0}(b'_-)}(-1)^tq_i^{(\#)}(G_{\lambda_1}^{\mathrm{up}}((\pi_{\lambda_1}(\ast b_-)), E_i^{(t)}\ast(G_{\varphi_t})F_i^{(N)}.v_{\lambda_1})_{\lambda_1}|d\rangle_i \notag\\
&=(-q_i)^{\varphi_i^{\lambda_0}(b'_-)}\prod_{s=1}^{d}(1-q_i^{2s}) \notag\\
&\hspace{40pt}\times\sum_{t=0}^{\varphi_i^{\lambda_0}(b'_-)}(-1)^tq_i^{(\#)+\frac{1}{2}\varphi_t(\varphi_t-1)}\sum_{(\clubsuit)}d_{b_+, b}^{i, \varphi_t}F_{\pi_{\lambda_1}(\ast b_-), b'}^{(t), i}c_{-Ni, b}^{\hat{b}}\delta_{b', \pi_{\lambda_1}(\ast \hat{b})}|d\rangle_i, \label{1}
\end{align}
where $$(\#)=-\varphi_t^2-t+2\gamma L\varphi_t+N(\varepsilon_i^{\lambda_0}(b'_-)-N),$$ and
$$(\clubsuit)=\begin{cases}b, \hat{b}\in B(\infty), b'\in B(\lambda_1),\\ \varepsilon_i^{\ast}(b)\leqq 2\gamma L+t,\\ \varphi_i^{\lambda_1}(\pi_{\lambda_1}(\ast b_-))-t\geqq\varphi_i^{\lambda_1}(b'),\\ \varepsilon_i(\hat{b})\leqq \varepsilon_i^{\lambda_0}(b'_-).\end{cases}$$
Denote by $p_{t, b, b', \hat{b}}$ the term $$(-1)^tq_i^{(\#)+\frac{1}{2}\varphi_t(\varphi_t-1)}(-q_i)^{\varphi_i^{\lambda_0}(b'_-)}\prod_{s=1}^{d}(1-q_i^{2s})d_{b_+, b}^{i, \varphi_t}F_{\pi_{\lambda_1}(\ast b_-), b'}^{(t), i}c_{-Ni, b}^{\hat{b}}\delta_{b', \pi_{\lambda_1}(\ast \hat{b})}$$ with $t, b, b', \hat{b}$ appearing in the sum in (\ref{1}). Then, by Proposition \ref{EFaction}, \ref{EFaction'} and Remark \ref{d_and_dhat}, we have $p_{t, b, b', \hat{b}}\in q_i^{m_{t, b, b', \hat{b}}}\mathbb{Z}[q]$, where
\begin{align*}
m_{t, b, b', \hat{b}}&=-\varphi_t^2-t+2\gamma L\varphi_t+N(\varepsilon_i^{\lambda_0}(b'_-)-N)+\frac{1}{2}\varphi_t(\varphi_t-1)+\varphi_i^{\lambda_0}(b'_-)\\
&\hspace{25pt}-\varphi_t\varepsilon_i^{\ast}(b)-\frac{1}{2}\varphi_t(\varphi_t-1)-t(\varphi_i^{\lambda_1}(\pi_{\lambda_1}(\ast b_-))-t)-N(\varepsilon_i(\hat{b})-N)\\
&\geqq \varphi_t-\varphi_t^2+2\gamma L\varphi_t+N(\varepsilon_i^{\lambda_0}(b'_-)-\varepsilon_i(\hat{b}))-\varphi_t(2\gamma L+t)\\
&\hspace{25pt}-t(\varphi_i^{\lambda_1}(\pi_{\lambda_1}(\ast b_-))-t)\\
&\geqq\varphi_t-\varphi_t^2-\varphi_t t-t(\varphi_i^{\lambda_1}(\pi_{\lambda_1}(\ast b_-))-t).
\end{align*}
Now,
\begin{align*}
\varphi_i^{\lambda_1}(\pi_{\lambda_1}(\ast b_-))&=\varphi_i(\ast b_-)+\langle\lambda_1, \alpha_i^{\vee}\rangle\\
&=\varepsilon_i(\ast b_-)+\langle\weight b_-, \alpha_i^{\vee}\rangle+\varepsilon_i^{\lambda_0}(b'_-)\\
&=\varepsilon_i(\ast b_-)+\langle\weight b_-, \alpha_i^{\vee}\rangle+\varphi_i^{\lambda_0}(b'_-)-\langle\lambda_0+\weight b_-, \alpha_i^{\vee}\rangle\\
&=\varepsilon_i(\ast b_-)-\langle\lambda_0, \alpha_i^{\vee}\rangle+\varphi_i^{\lambda_0}(b'_-)\leqq \varphi_i^{\lambda_0}(b'_-).
\end{align*}
Hence, $$m_{t, b, b', \hat{b}}\geqq \varphi_t-\varphi_t^2-\varphi_t t-t\varphi_t\geqq -(\varphi_t+t)^2+t^2+\varphi_t\geqq -\varphi_i^{\lambda_0}(b'_-)^2.$$
So, Lemma \ref{roughestlem} follows.
\end{proof}
\subsubsection{Main calculation}\label{usesymm}
It follows from Lemma \ref{roughestlem} (and Remark \ref{roughestremark}) that 
\begin{itemize}
\item ${ _{{\bf i}}\zeta'}_{{\bf d}}^{G^{\mathrm{low}}(b_0)}\in \mathbb{Z}[q^{\pm 1}]$ (See the equality (\ref{zetaprime}) for the definition of ${ _{{\bf i}}\zeta'}_{{\bf d}}^{G^{\mathrm{low}}(b_0)}$), and
\item we may ignore the degree $\geqq l(w_0) L$ part of the Laurent polynomial $p_k$ for any $k$ when calculating the degree $< L$ parts of the Laurent polynomials ${ _{{\bf i}}\zeta'}_{{\bf d}}^{G^{\mathrm{low}}(b_0)}$.
\end{itemize}
Hence, the following theorem together with Proposition \ref{positivty_1} implies theorem \ref{maintheorem}.
\begin{theorem}\label{maincalc}
In any summand of the right-hand side of (\ref{shiki}) which contributes to the degree $< L$ part of the Laurent polynomial ${ _{{\bf i}}\zeta'}_{{\bf d}}^{G^{\mathrm{low}}(b_0)}$, we have
\begin{itemize}
\item[{\nor (i)}] $G_{\lambda_0}^{\mathrm{low}}(b'_k)\in V_{s_{i_k}\cdots s_{i_1}}(\lambda_0)$ for $k=0,\dots, l(w_0)$, and
\item[{\nor (ii)}] the degree $< l(w_0)L$ part of the Laurent polynomial $p_k$ is of the form $q_{i_k}^{\frac{1}{2}d_k(d_k-1)}\widehat{d}_{b_{k-1}, b}^{i_k, d_k}$ for some $b\in B(\infty)$ where $b_{k-1}\in B(\infty)$ is the element satisfying $G^{\mathrm{low}}(b_{k-1}).v_{s_{i_{k-1}}\cdots s_{i_1}\lambda_0}=G_{\lambda_0}^{\mathrm{low}}(b'_{k-1})$, for $k=1,\dots, l(w_0)$.
\end{itemize}
\end{theorem}
\begin{proof}
Fix $k\in \{1,\dots, l(w_0)\}$ such that $G_{\lambda_0}^{\mathrm{low}}(b'_{k-1})\in V_{s_{i_{k-1}}\cdots s_{i_1}}(\lambda_0)$. (Note that $k=1$ clearly satisfies this condition.) Set $w:=s_{i_{k-1}}\cdots s_{i_1}$ Then, the proof of theorem is completed by showing that $G_{\lambda_0}^{\mathrm{low}}(b'_k)\in V_{s_{i_k}w}(\lambda_0)$ and the degree $< l(w_0) L$ part of the Laurent polynomial $p_k$ is of the form $q_{i_k}^{\frac{1}{2}d_k(d_k-1)}\widehat{d}_{b_{k-1}, b}^{i_k, d_k}$. We abbreviate $i_k$ to $i$, $b'_{k-1}$ to $b'_+$, $b'_k$ to $b'_-$, $n_k$ to $n$, $d_k$ to $d$. Note that by Lemma \ref{method} we have $d=-\frac{1}{2}\langle \weight b'_{k-1} +\weight b'_k, \alpha_{i_k}^{\vee}\rangle$.

By our assumption, there exists the unique element $b_+\in B(\infty)$ such that $G^{\mathrm{low}}(b_+).v_{w\lambda_0}=G_{\lambda_0}^{\mathrm{low}}(b'_+)$. By Proposition \ref{Demazure} (i) and $\phi_i^{\ast}(c_{(G_{\lambda_0}^{\mathrm{low}}(b'_+), \cdot)_{\lambda_0}, G_{\lambda_0}^{\mathrm{up}}(b'_-)}^{\lambda_0})\neq 0$, we have also $G_{\lambda_0}^{\mathrm{low}}(b'_-)\in V_w(\lambda_0)$ and, hence, there exists the unique element $b_-\in B(\infty)$ such that $G^{\mathrm{low}}(b_-).v_{w\lambda_0}=G_{\lambda_0}^{\mathrm{low}}(b'_-)$. 

Then, by Lemma \ref{method} and the similar argument in the proof of Lemma \ref{roughestlem}, we have
\begin{align}\label{2}
&c_{(G_{\lambda_0}^{\mathrm{low}}(b'_+), \cdot)_{\lambda_0}, G_{\lambda_0}^{\mathrm{up}}(b'_-)}^{\lambda_0}.|0\rangle_i\notag \\
&=(-q_i)^{\varphi_i^{\lambda_0}(b'_-)}\prod_{s=1}^{d}(1-q_i^{2s}) \notag\\
&\hspace{40pt}\sum_{t=0}^{\varphi_i^{\lambda_0}(b'_-)}(-1)^tq_i^{(\bigstar)}\sum_{(\clubsuit)'}d_{b_+, b}^{i, \varphi_t}F_{\pi_{\lambda_1}(\ast b_-), b'}^{(t), i}c_{-Ni, b}^{\hat{b}}\delta_{b', \pi_{\lambda_1}(\ast \hat{b})}|d\rangle_i,
\end{align}
where we set $\dis \lambda_1= \varepsilon_i^{\lambda_0}(b'_-)\varpi_i+\sum_{j\in I\setminus \{i\}}\varepsilon_j\varpi_j$ with $\varepsilon_j^{\lambda_0}(b'_-)\leqq \varepsilon_j\in \mathbb{Z}_{\geqq 0}$ for all $j\in I\setminus \{i\}$, $\varphi_t=\varphi_i^{\lambda_0}(b'_-)-t$, $N=n+\varphi_i^{\lambda_0}(b'_-)$ and
$$(\bigstar)=-\varphi_t^2-t+\varphi_t\langle w\lambda_0, \alpha_i^{\vee}\rangle+N(\varepsilon_i^{\lambda_0}(b'_-)-N)+\frac{1}{2}\varphi_t(\varphi_t-1),$$
$$(\clubsuit)'=\begin{cases}b, \hat{b}\in B(\infty), b'\in B(\lambda_1),\\ \varphi_i^{\lambda_1}(\pi_{\lambda_1}(\ast b_-))-t\geqq\varphi_i^{\lambda_1}(b'),\\  \varepsilon_i(\hat{b})\leqq \varepsilon_i^{\lambda_0}(b'_-).\end{cases}$$
Denote by $\tilde{p}_{t, b, b', \hat{b}}$ the term 
$$(-1)^tq_i^{(\bigstar)}(-q_i)^{\varphi_i^{\lambda_0}(b'_-)}\prod_{s=1}^{d}(1-q_i^{2s})d_{b_+, b}^{i, \varphi_t}F_{\pi_{\lambda_1}(\ast b_-), b'}^{(t), i}c_{-Ni, b}^{\hat{b}}\delta_{b', \pi_{\lambda_1}(\ast \hat{b})}$$ 
with $t, b, b', \hat{b}$ appearing in the sum in (\ref{2}). Then, by Proposition \ref{EFaction} and \ref{EFaction'}, we have $\tilde{p}_{t, b, b', \hat{b}}\in q_i^{\tilde{m}_{t, b, b', \hat{b}}}\mathbb{Z}[q]$, where
\begin{align*}
\tilde{m}_{t, b, b', \hat{b}}&\geqq -\varphi_t^2-t+\varphi_t\langle \lambda_0, w^{-1}\alpha_i^{\vee}\rangle+N(\varepsilon_i^{\lambda_0}(b'_-)-N)+\frac{1}{2}\varphi_t(\varphi_t-1)+\varphi_i^{\lambda_0}(b'_-)\\
&\hspace{25pt}-\frac{2}{(\alpha_i, \alpha_i)}l_3-t(\varphi_i^{\lambda_1}(\pi_{\lambda_1}(\ast b_-))-t)-N(\varepsilon_i(\hat{b})-N)\\
&\geqq -\frac{1}{2}\varphi_t(\varphi_t-1)+2\gamma L\varphi_t+N(\varepsilon_i^{\lambda_0}(b'_-)-\varepsilon_i(\hat{b}))\\
&\hspace{25pt}-l_3-t(\varphi_i^{\lambda_1}(\pi_{\lambda_1}(\ast b_-))-t).
\end{align*}
Note that $w^{-1}\alpha_i^{\vee}$ is a positive coroot and $\weight (\ast b_+)=\weight b_+=\weight b'_+-w\lambda_0\in A$ by Lemma \ref{numest} (i). 
Now, 
\begin{align}
&\varphi_i^{\lambda_1}(\pi_{\lambda_1}(\ast b_-)) \notag\\
&=\varphi_i(\ast b_-)+\langle\lambda_1, \alpha_i^{\vee}\rangle \notag\\ 
&=\varepsilon_i(\ast b_-)+\langle\weight b_-, \alpha_i^{\vee}\rangle+\varepsilon_i^{\lambda_0}(b'_-) \notag\\
&=\varepsilon_i(\ast b_-)+\langle\weight b_-, \alpha_i^{\vee}\rangle+\varphi_i^{\lambda_0}(b'_-)-\langle w\lambda_0+\weight b_-, \alpha_i^{\vee}\rangle \notag \\
&=\varepsilon_i(\ast b_-)-\langle w\lambda_0, \alpha_i^{\vee}\rangle+\varphi_i^{\lambda_0}(b'_-) \notag \\ \label{shiki2}
&\leqq \varphi_i^{\lambda_0}(b'_-).
\end{align}
Hence, 
\begin{align}\label{mainineq}
\tilde{m}_{t, b, b', \hat{b}}&\geqq -\frac{1}{2}\varphi_t(\varphi_t-1)+2\gamma L\varphi_t+N(\varepsilon_i^{\lambda_0}(b'_-)-\varepsilon_i(\hat{b}))-l_3-t\varphi_t\notag \\
&\geqq -l_5+2\gamma L\varphi_t+N(\varepsilon_i^{\lambda_0}(b'_-)-\varepsilon_i(\hat{b}))-l_3\notag \\
&>-L+2\gamma L\varphi_t+N(\varepsilon_i^{\lambda_0}(b'_-)-\varepsilon_i(\hat{b})),
\end{align}
because $0\leqq t, \varphi_t\leqq \varphi_i^{\lambda_0}(b'_-)\leqq -\height (\weight b_0)$ by Lemma \ref{numest} (iii). 
Therefore, 
$\tilde{p}_{t, b, b', \hat{b}}$ does not contribute to the degree $< l(w_0) L$ part of the Laurent polynomial $p_{k+1}$ unless $\varphi_t=0$. 

From now on, we consider $\tilde{p}_{t, b, b', \hat{b}}$ in the case $\varphi_t=0$ (i.e. $t=\varphi_i^{\lambda_0}(b'_-)$) and $b'=\pi_{\lambda_1}(\ast \hat{b})$. We have $b=b_+$ because $G_{\varphi_t}=G^{\mathrm{low}}(b_+)$.

Since $\varphi_i^{\lambda_1}(\pi_{\lambda_1}(\ast b_-))\leqq \varphi_i^{\lambda_0}(b'_-)$ and $\varphi_i^{\lambda_1}(\pi_{\lambda_1}(\ast b_-))-\varphi_i^{\lambda_0}(b'_-)\geqq\varphi_i^{\lambda_1}(\pi_{\lambda_1}(\ast \hat{b}))$, we have $\varphi_i^{\lambda_1}(\pi_{\lambda_1}(\ast b_-))=\varphi_i^{\lambda_0}(b'_-)$, equivalently, $\varepsilon_i^{\ast}(b_-)=\langle w\lambda_0, \alpha_i^{\vee}\rangle$. (See the inequality (\ref{shiki2}).) So, by Proposition \ref{c1=epsiloni1},
$$G_{\lambda_0}^{\mathrm{low}}(b'_-)=G^{\mathrm{low}}(b_-).v_{w\lambda_0}\in U_q(\mathfrak{n}^-)F_i^{(\langle w\lambda_0, \alpha_i^{\vee}\rangle)}.v_{w\lambda_0}=V_{s_i w}(\lambda_0).$$
Hence, (i) follows. 

When $t=\varphi_i^{\lambda_0}(b'_-)=\varphi_i^{\lambda_1}(\pi_{\lambda_1}(\ast b_-))$, we also have $\pi_{\lambda_1}(\ast \hat{b})(=b')=\tilde{f}_i^t(\pi_{\lambda_1}(\ast b_-))=\pi_{\lambda_1}(\ast (\tilde{f}_i^{\ast})^t(b_-))$ $(=:{b'}^{(0)}\neq 0)$. (Equivalently, $\hat{b}=(\tilde{f}_i^{\ast})^t(b_-)$.) Moreover, by the inequality (\ref{mainineq}) and $N>\gamma L$ (by Lemma \ref{numest} (ii)) we may assume that $\varepsilon_i(\hat{b})=\varepsilon_i^{\lambda_0}(b'_-)$. (In fact, this equality always holds. See the proof of Theorem \ref{comparison} in Appendix.)



Taking these assumptions into account, we have
\begin{align*}
\tilde{p}_{t, b, b', \hat{b}}&=\tilde{p}_{\varphi_i^{\lambda_0}(b'_-), b_+, {b'}^{(0)}, \hat{b}}\\
&=(-1)^{\varphi_i^{\lambda_0}(b'_-)}q_i^{-\varphi_i^{\lambda_0}(b'_-)+N(\varepsilon_i^{\lambda_0}(b'_-)-N)}(-q_i)^{\varphi_i^{\lambda_0}(b'_-)}\prod_{s=1}^{d}(1-q_i^{2s})c_{-Ni, b_+}^{\hat{b}}\\
&=q_i^{N(\varepsilon_i^{\lambda_0}(b'_-)-N)}\prod_{s=1}^d(1-q_i^{2s})c_{-Ni, b_+}^{\hat{b}}.
\end{align*}
Moreover, by Proposition \ref{similarity}, $\varepsilon_i(\hat{b})-N=\varepsilon_i^{\lambda_0}(b'_-)-N=d$ and $\Delta_iN>l(w_0)L$, we have
\begin{align*}
&(\tilde{p}_{\varphi_i^{\lambda_0}(b'_-), b_+, {b'}^{(0)}, \hat{b}})_{<l(w_0) L}\\
&=\left(q_i^{N(\varepsilon_i^{\lambda_0}(b'_-)-N)}\prod_{s=1}^{d}(1-q_i^{2s})c_{-Ni, b_+}^{\hat{b}}\right)_{<l(w_0) L}\\
&=\left(\prod_{s=1}^{d}(1-q_i^{2s})q_i^{dN}c_{-Ni, b_+}^{\hat{b}}\right)_{<l(w_0) L}\\
&=q_i^{dN}\left(\prod_{s=1}^{d}(1-q_i^{2s})\left(c_{-Ni, b_+}^{\hat{b}}\right)_{<(l(w_0) L-\Delta_idN)}\right)_{<(l(w_0) L-\Delta_idN)}\\
&=q_i^{dN}\left(\prod_{s=1}^{d}(1-q_i^{2s})\left(q_i^{\frac{1}{2}d(d-1)}\left[ \begin{array}{c} \varepsilon_i(\hat{b})\\N \end{array} \right]_i\widehat{d}_{b_+, \tilde{e}_i^{\varepsilon_i(\hat{b})}\hat{b}}^{i, d}\right)_{<(l(w_0) L-\Delta_idN)}\right)_{<(l(w_0) L-\Delta_idN)}\\
&=q_i^{dN}\left(q_i^{\frac{1}{2}d(d-1)}\prod_{s=1}^{d}(1-q_i^{2s})\left[ \begin{array}{c} N+d\\N \end{array} \right]_i\widehat{d}_{b_+, \tilde{e}_i^{\varepsilon_i(\hat{b})}\hat{b}}^{i, d}\right)_{<(l(w_0) L-\Delta_idN)}\\
&=q_i^{dN}\left(q_i^{\frac{1}{2}d(d-1)-dN}\prod_{s=1}^{d}(1-q_i^{2(d+N)-2s+2})\widehat{d}_{b_+, \tilde{e}_i^{\varepsilon_i(\hat{b})}\hat{b}}^{i, d}\right)_{<(l(w_0) L-\Delta_idN)}\\
&=\left(q_i^{\frac{1}{2}d(d-1)}\prod_{s=1}^{d}(1-q_i^{2(d+N)-2s+2})\widehat{d}_{b_+, \tilde{e}_i^{\varepsilon_i(\hat{b})}\hat{b}}^{i, d}\right)_{<l(w_0) L}\\
&=q_i^{\frac{1}{2}d(d-1)}\widehat{d}_{b_+, \tilde{e}_i^{\varepsilon_i(\hat{b})}\hat{b}}^{i, d}.
\end{align*}
The last equality holds as follows:

By Remark \ref{d_and_dhat},
$$\widehat{d}_{b_+, \tilde{e}_i^{\varepsilon_i(\hat{b})}\hat{b}}^{i, d}=q_i^{d\langle \weight b_+, \alpha_i^{\vee}\rangle+d(d+1)}\overline{d_{b_+, \tilde{e}_i^{\varepsilon_i(\hat{b})}\hat{b}}^{i, d}}.$$
So, we have $q_i^{\frac{1}{2}d(d-1)}\widehat{d}_{b_+, \tilde{e}_i^{\varepsilon_i(\hat{b})}\hat{b}}^{i, d}\in q^{l_3+l_4+l_5}\mathbb{Z}[q^{-1}]$. Note that $\weight b_+\in A$.

Moreover, 
$q_i^{\frac{1}{2}d(d-1)}\widehat{d}_{b_+, \tilde{e}_i^{\varepsilon_i(\hat{b})}\hat{b}}^{i, d}\in q^{-l_3}\mathbb{Z}[q].$

Hence, we obtain (ii) in Theorem \ref{maincalc}.
\end{proof}
\def\thesection{\Alph{section}}
 \setcounter{section}{0}
\section{Appendix}\label{app}
After this paper was posted on the preprint server, Yoshiyuki Kimura pointed out to the author the existence of a much simpler proof of the positivity of the transition matrices from canonical bases to PBW bases (Theorem \ref{maintheorem}). Moreover, it has been found that this simple method provides the same constants as in Section \ref{main} even when $\mathfrak{g}$ is of nonsymmetric finite type. In this appendix, we explain this point and the corollaries obtained from the following discussion. The author wishes to express his thanks to Yoshiyuki Kimura. 

\subsection{A simpler proof of the positivity}\label{app1}
\begin{definition}
We can define the $\mathbb{Q}(q)$-bilinear form $(\ ,\ ):U_q(\mathfrak{n}^-)\times U_q(\mathfrak{n}^-)\to \mathbb{Q}(q)$ uniquely by 
\[
\begin{array}{c}
(1,1)=1,\\
\dis (F_ix, y)=\frac{1}{(1-q_i^2)}(x, e'_i(y)),\ (xF_i, y)=\frac{1}{(1-q_i^2)}(x, {_ie'}(y)),
\end{array}
\]
for any $i\in I$ and $x, y\in U_q(\mathfrak{n}^-)$. This bilinear form is symmetric. See \cite[Chapter 1]{Lusbook} for details.
\end{definition}
We prepare some propositions before giving the simpler proof.
\begin{proposition}[{\cite[38.1.6]{Lusbook}}]\label{kernel}
For all $i\in I$, 
\begin{itemize}
\item[{\nor (i)}] $\Ker e'_i=\Set{x\in U_q(\mathfrak{n}^-)| {T''_{i, 1}}^{-1}(x)\in U_q(\mathfrak{n}^-) }$,
\item[{\nor (ii)}] $\Ker {_ie'}=\Set{x\in U_q(\mathfrak{n}^-)| T''_{i, 1}(x)\in U_q(\mathfrak{n}^-) }$.
\end{itemize}
\end{proposition}
\begin{proposition}[{\cite[38.2.1]{Lusbook}}]\label{T-inv}
For $x, y\in \Ker e'_i$ $(i\in I)$, we have 
$$(x, y)=((T''_{i, 1})^{-1}(x), (T''_{i, 1})^{-1}(y) ).$$
{\nor (}See also Proposition \ref{kernel} (i).{\nor )}
\end{proposition}
\begin{proposition}[The dual bases of PBW bases with respect to $(\ ,\ )$]\label{dualbasis}
Let ${\bf i}$ be a reduced word of $w_0$ and set 
$$\tilde{F}_{{\bf i}}^{{\bf d}}:=\left(\prod_{k=1}^{l(w_0)}q_{i_k}^{\frac{1}{2}d_k(d_k-1)}(1-q_{i_k}^2)^{d_k}\right)F_{i_1}^{d_1}T''_{i_1, 1}(F_{i_2}^{d_2})\cdots T''_{i_1, 1}T''_{i_2, 1}\cdots T''_{i_{l(w_0)-1}, 1}(F_{i_{l(w_0)}}^{d_{l(w_0)}}),$$
for ${\bf d}\in (\mathbb{Z}_{\geqq 0})^{l(w_0)}.$
Then, we have 
$$(\tilde{F}_{{\bf i}}^{{\bf d}}, F_{{\bf i}}^{{\bf d}'})=\delta_{{\bf d}, {\bf d}'}\ \text{for}\ {\bf d}, {\bf d}'\in (\mathbb{Z}_{\geqq 0})^{l(w_0)}.$$
\end{proposition}
\begin{proof}
Using Proposition \ref{kernel} (i) and $e'_i(F_i^{(d)})=q_i^{-d+1}F_i^{(d-1)}$, we have, for any ${\bf d}\in (\mathbb{Z}_{\geqq 0})^{l(w_0)}$, 
$$e'_i(F_{{\bf i}'}^{{\bf d}})=q_i^{-d_1+1}F_{{\bf i}'}^{{\bf d}-(1,0,\dots,0)},$$
and 
$$(e'_i)^e(F_{{\bf i}'}^{{\bf d}})=q_i^{-\frac{1}{2}e(2d_1-e-1)}F_{{\bf i}'}^{{\bf d}-(e,0,\dots,0)}\ \text{for}\ e\in \mathbb{Z}_{\geqq 0},$$
where $F_{{\bf i}'}^{{\bf d}-(e,0,\dots,0)}:=0$ if $e>d_1$. Combining this equality with the definition of the bilinear form $(\ ,\ )$ and Proposition \ref{T-inv}, we can easily obtain this proposition.
\end{proof}
\begin{proposition}[{\cite[Proposition 3.4.7, Corollary 3.4.8]{Saito_PBW}, \cite[Theorem 1.2]{Lus_braid}}]\label{Saito_refl}
Let $i\in I$ and
\begin{align*}
{}^i\pi:U_q(\mathfrak{n}^-)=\bigoplus_{n\in \mathbb{Z}_{\geqq 0}}F_i^{(n)}\Ker e'_i\to \Ker e'_i,\\ \pi^i:U_q(\mathfrak{n}^-)=\bigoplus_{n\in \mathbb{Z}_{\geqq 0}}\Ker {_ie'}F_i^{(n)}\to \Ker {_ie'},
\end{align*}
be natural projections.

Then, for $b\in B(\infty)$ with $\varepsilon_i(b)=0$, we have
$${}^i\pi(G^{\mathrm{low}}(b))=T''_{i, 1}(\pi^i G^{\mathrm{low}}(\Lambda_i^{-1}(b))),$$
where $\Lambda_i^{-1}:\Set{b\in B(\infty)|\varepsilon_i(b)=0}\to\Set{b\in B(\infty)|\varepsilon_i^{\ast}(b)=0}$ is the bijection defined by $b\mapsto \tilde{f}_i^{\varphi_i^{\ast}(b)}(\tilde{e}_i^{\ast})^{\varepsilon_i^{\ast}(b)}b.$
\end{proposition}
The following theorem together with Proposition \ref{positivty_1} implies the positivity of the transition matrices from canonical bases to PBW bases if $\mathfrak{g}$ is of type $ADE$. Moreover, the following proof is simpler than the proof in Section \ref{main}.
\begin{theorem}{\nor (}We do not assume that $\mathfrak{g}$ is of type $ADE$.{\nor )}\label{easyproof}

Fix a reduced word ${\bf i}$ of $w_0$. Then, for $b\in B(\infty)$ and ${\bf d}\in (\mathbb{Z}_{\geqq 0})^{l(w_0)}$, we have
$${_{{\bf i}}\zeta}_{{\bf d}}^{G^{\mathrm{low}}(b)}=\left(\prod_{k=1}^{l(w_0)}q_{i_k}^{\frac{1}{2}d_k(d_k-1)}\right)\sum_{\substack{b_1,\dots,b_{l(w_0)-1}\in B(\infty)\\ \text{with}\ \varepsilon_{i_l}^{\ast}(b_l)=0\ \text{for\ all}\ l}}\mspace{-20mu}\widehat{d}_{b, \Lambda_{i_1}(b_1)}^{i_1, d_1}\widehat{d}_{b_1, \Lambda_{i_2}(b_2)}^{i_2, d_2}\cdots \widehat{d}_{b_{l(w_0)-1}, 1}^{i_{l(w_0)}, d_{l(w_0)}}.$$
\end{theorem}
\begin{proof}
By Proposition \ref{dualbasis}, we have
$${_{{\bf i}}\zeta}_{{\bf d}}^{G^{\mathrm{low}}(b)}=(\tilde{F}_{{\bf i}}^{{\bf d}}, G^{\mathrm{low}}(b)).$$
Set $C:=\prod_{k=1}^{l(w_0)}q_{i_k}^{\frac{1}{2}d_k(d_k-1)}(1-q_{i_k}^2)^{d_k}$.
By the definition of the bilinear form $(\ ,\ )$ and Proposition \ref{EFaction'}, we have
\begin{align*}
&(\tilde{F}_{{\bf i}}^{{\bf d}}, G^{\mathrm{low}}(b))\\
&=C(F_{i_1}^{d_1}T''_{i_1, 1}(F_{i_2}^{d_2})\cdots T''_{i_1, 1}T''_{i_2, 1}\cdots T''_{i_{l(w_0)-1}, 1}(F_{i_{l(w_0)}}^{d_{l(w_0)}}), G^{\mathrm{low}}(b))\\
&=C(1-q_{i_1}^2)^{-d_1}(T''_{i_1, 1}(F_{i_2}^{d_2})\cdots T''_{i_1, 1}T''_{i_2, 1}\cdots T''_{i_{l(w_0)-1}, 1}(F_{i_{l(w_0)}}^{d_{l(w_0)}}), (e'_{i_1})^{d_1}G^{\mathrm{low}}(b))\\
&=q_{i_1}^{\frac{1}{2}d_1(d_1-1)}\sum_{\substack{\tilde{b}\in B(\infty)\\ \tilde{b}=\tilde{e}_{i_1}^{d_1}b\ \text{or}\ \varepsilon_i(b)-d_1<\varepsilon_i(\tilde{b})}}\widehat{d}_{b, \tilde{b}}^{i_1, d_1}(\tilde{F}_{{\bf i}}^{(0, d_2,\dots,d_{l(w_0)})}, G^{\mathrm{low}}(\tilde{b})).
\end{align*}
By Proposition \ref{kernel} (i) and Proposition \ref{c1=epsiloni1}, we have
$$(\tilde{F}_{{\bf i}}^{(0, d_2,\dots,d_{l(w_0)})}, G^{\mathrm{low}}(\tilde{b}))=0$$ 
unless $\varepsilon_i(\tilde{b})=0$.
Moreover, when $\varepsilon_i(\tilde{b})=0$, by Proposition \ref{T-inv} and \ref{Saito_refl}, we have
\begin{align*}
(\tilde{F}_{{\bf i}}^{(0, d_2,\dots,d_{l(w_0)})}, G^{\mathrm{low}}(\tilde{b}))&=(\tilde{F}_{{\bf i}}^{(0, d_2,\dots,d_{l(w_0)})}, {}^{i_1}\pi (G^{\mathrm{low}}(\tilde{b})))\\
&=(\tilde{F}_{{\bf i}}^{(0, d_2,\dots,d_{l(w_0)})}, T''_{i_1, 1}(\pi^{i_1}(G^{\mathrm{low}}(\Lambda_{i_1}^{-1}(\tilde{b})))))\\
&=(\tilde{F}_{{\bf i}'}^{(d_2,\dots,d_{l(w_0)}, 0)}, \pi^{i_1}(G^{\mathrm{low}}(\Lambda_{i_1}^{-1}(\tilde{b})))),
\end{align*}
where $\alpha_{i'_1}:=-w_0\alpha_{i_1}$ and ${\bf i}'=(i_2,\dots,i_{l(w_0)}, i'_1)$. Note that ${\bf i}'$ is also a reduced word of $w_0$. Set $b_1:=\Lambda_{i_1}^{-1}(\tilde{b})$. ($\varepsilon_{i_1}^{\ast}(b_1)=0$.)

Now, by Proposition \ref{kernel} (ii), we have 
$$\tilde{F}_{{\bf i}'}^{(d_2,\dots,d_{l(w_0)}, 0)}\in \Ker {_{i_1}e'}.$$
Moreover, 
$$G^{\mathrm{low}}(b_1)-\pi^{i_1}(G^{\mathrm{low}}(b_1))\in \bigoplus_{n\in \mathbb{Z}_{>0}}\Ker {_{i_1}e'}F_{i_1}^{(n)}.$$
Therefore, 
$$(\tilde{F}_{{\bf i}'}^{(d_2,\dots,d_{l(w_0)}, 0)}, \pi^{i_1}(G^{\mathrm{low}}(b_1)))=(\tilde{F}_{{\bf i}'}^{(d_2,\dots,d_{l(w_0)}, 0)}, G^{\mathrm{low}}(b_1)).$$
We can repeat the argument above for $(\tilde{F}_{{\bf i}'}^{(d_2,\dots,d_{l(w_0)}, 0)}, G^{\mathrm{low}}(b_1))$. Hence, the proof is completed.
\end{proof}
\begin{remark}
For $\tilde{b}\in B(\infty)$ with $\epsilon_i(\tilde{b})=0$, we have 
$$\weight \Lambda_i^{-1}(\tilde{b})=\weight\tilde{b}+(\varepsilon_i^{\ast}(\tilde{b})-\varphi_i^{\ast}(\tilde{b}))\alpha_i=\weight\tilde{b}-\langle \weight \tilde{b}, \alpha_i^{\vee}\rangle\alpha_i=s_i(\weight\tilde{b}).$$
Therefore, in a nonzero summand of the right-hand side of the equality in Theorem \ref{easyproof}, we have $\weight b_k\in A\ \text{for all\ }k$. Here $A$ is as in Section \ref{main}. (We regard $b$ as $b_0$ in Section \ref{main}.)
\end{remark}
\begin{remark}
It also follows from the proof of Theorem \ref{easyproof} that for $b\in B(\infty)$ there exists ${\bf c}\in (\mathbb{Z}_{\geqq 0})^{l(w_0)}$ such that 
$${_{{\bf i}}\zeta}_{{\bf d}}^{G^{\mathrm{low}}(b)}=
\begin{cases}
1&\text{if\ }{\bf d}={\bf c},\\
\in q\mathbb{Z}[q]& \text{if\ }{\bf d}>{\bf c},\\
0&\text{otherwise},
\end{cases}$$
where
\begin{center}
${\bf d}=(d_1, d_2,\dots,d_{l(w_0)})>{\bf c}=(c_1, c_2,\dots,c_{l(w_0)})$\\
$\Leftrightarrow$ There exists $k\in \{1, \dots, l(w_0) \}$ such that $d_1=c_1,\dots ,d_{k-1}=c_{k-1}, d_k>c_k$.
\end{center}
This is known as ``the unitriangularity property''.
\end{remark}
\subsection{Comparison with Section \ref{main}}
We prove that the calculation procedure of ${_{{\bf i}}\zeta}_{{\bf d}}^{G^{\mathrm{low}}(b_0)}$ in Section \ref{main} is the same as the one in Theorem \ref{easyproof}. That is, we show the following theorem:
\begin{theorem}\label{comparison}
Let $\lambda_0$, $A$ be as in Section \ref{main}, $w\in W$ and $i\in I$ with $l(s_iw)>l(w)$. Take $b^{(0)}, b^{(1)}\in B(\infty)$ with $\weight b^{(0)}, \weight b^{(1)}\in A$, $G_{\lambda_0}^{\mathrm{low}}({b'}^{(0)}):=G^{\mathrm{low}}(b^{(0)}).v_{w\lambda_0}\neq 0$ and $G_{\lambda_0}^{\mathrm{low}}({b'}^{(1)}):=G^{\mathrm{low}}(b^{(1)}).v_{s_iw\lambda_0}\neq 0$. (In particular, $\varepsilon_i^{\ast}(b^{(1)})=0$.)
$$c_{(G_{\lambda_0}^{\mathrm{low}}({b'}^{(0)}), \cdot)_{\lambda_0}, G_{\lambda_0}^{\mathrm{up}}({b'}^{(1)})}^{\lambda_0}.|0 \rangle_i:=p\left| -\frac{1}{2}\langle \weight {b'}^{(0)}+\weight {b'}^{(1)}, \alpha_i^{\vee}\rangle \right>_i$$
with $p\in \mathbb{Z}[q^{\pm 1}]$. Then, the degree $< l(w_0)L$ part of $p$ is equal to
$$q_i^{\frac{1}{2}d(d-1)}\widehat{d}_{b^{(0)}, \Lambda_i(b^{(1)})}^{i, d},$$
where $d:=-\frac{1}{2}\langle \weight {b'}^{(0)}+\weight {b'}^{(1)}, \alpha_i^{\vee}\rangle$.
\end{theorem}
Before proving this theorem, we prepare one lemma.
\begin{lemma}\label{coincide}
Let $i\in I$ and $b\in B(\infty)$. Suppose that $\Lambda$ and $\Lambda'$ are any composition operators 
of $\tilde{e}_i, \tilde{f}_i, \tilde{e}_i^{\ast}$ and $\tilde{f}_i^{\ast}$ such that $\Lambda(b)\neq 0$ and $\Lambda'(b)\neq 0$. {\nor (}For example, $\Lambda=\tilde{e}_i\tilde{f}_i^2\tilde{e}_i^{\ast}\tilde{f}_i^{\ast}(\tilde{e}_i^{\ast})^2$ and $\Lambda'=\tilde{e}_i$.{\nor )}

Then, we have $\Lambda(b)=\Lambda'(b)$ if and only if 
\begin{center}
$\langle\weight\Lambda(b), \alpha_i^{\vee}\rangle=\langle\weight\Lambda'(b), \alpha_i^{\vee}\rangle$ and $\varepsilon_i^{\ast}(\Lambda(b))=\varepsilon_i^{\ast}(\Lambda'(b))$.
\end{center}
\end{lemma}
\begin{proof}
Define the crystal structure on the set $B(\infty)\times \mathbb{Z}$ by
\begin{itemize}
\item $\weight (b', m)=\weight b'+m\alpha_i$,
\item $\varepsilon_j((b', m))=
\begin{cases}
\max\{\varepsilon_i(b'), -m-\langle \weight b', \alpha_i^{\vee} \rangle \}& \text{if\ }j=i,\\ \varepsilon_j(b') & \text{otherwise},
\end{cases}$
\item $\varphi_j((b', m))=
\begin{cases}
\max\{\varphi_i(b')+2m, m \}& \text{if\ }j=i,\\
\varphi_j(b')+m\langle\alpha_i, \alpha_j^{\vee}\rangle& \text{otherwise},
\end{cases}$
\item $\tilde{e}_j(b', m)=
\begin{cases}
(\tilde{e}_jb', m) & \text{if\ } j\neq i\ \text{or\ } \varphi_i(b')\geqq -m,\\ 
(b', m+1) & \text{otherwise},
\end{cases}$
\item $\tilde{f}_j(b', m)=
\begin{cases}
(\tilde{f}_jb', m) & \text{if\ } j\neq i\ \text{or\ } \varphi_i(b')> -m,\\ 
(b', m-1) & \text{otherwise},
\end{cases}$
\end{itemize}
Note that $(0, m)=0$ in the sense of crystals for any $m\in \mathbb{Z}$. (See Definition \ref{abstractcrystal}.)
\begin{remark}
This crystal is often denoted by $B(\infty)\otimes B_i$.
\end{remark}
We use the following proposition due to Kashiwara. 
\begin{proposition}[{\cite[Theorem 2.2.1]{KasDem}}]\label{Kasemb}
There exists a unique embedding of crystals $\Psi_i:B(\infty)\to B(\infty)\times \mathbb{Z}$ which sends $b^{(0,\dots, 0)}$ to $(b^{(0,\dots, 0)}, 0)$. Moreover, $\Psi_i$ have the following properties:
\begin{itemize}
\item[{\nor (i)}] $\tilde{e}_j\circ\Psi_i=\Psi_i\circ\tilde{e}_j$ and $\tilde{f}_j\circ\Psi_i=\Psi_i\circ\tilde{f}_j$ for any $j\in I$.
\item[{\nor (ii)}] If $\Psi_i(b')=(b'_0, m)$, then $\Psi_i(\tilde{f}_i^{\ast}b')=(b'_0, m-1)$ and $\varepsilon_i^{\ast}(b')=m$.
\item[{\nor (iii)}] $\Img \Psi_i=\Set{(b', m)|\varepsilon_i^{\ast}(b')=0, m\leqq 0}$.
\end{itemize}
\end{proposition}
Note that (ii) implies that if $\Psi_i(b')=(b'_0, m)$ and $\tilde{e}_i^{\ast}b'\neq 0$ then  $\Psi_i(\tilde{e}_i^{\ast}b')=(b'_0, m+1)$.

Set $\Psi_i(b)=(b_0, m_0), \Psi_i(\Lambda(b))=(\tilde{b}, \tilde{m}), \Psi_i(\Lambda'(b))=(\tilde{b}', \tilde{m}')$ and $$B_0^{(i)}:=\left(\Set{\tilde{e}_i^sb_0|s\in \mathbb{Z}_{\geqq 0}}\cup \Set{\tilde{f}_i^sb_0|s\in \mathbb{Z}_{\geqq 0}}\right)\setminus \{0\}.$$
Then, by the definition of the crystal structure on $B(\infty)\times \mathbb{Z}$ and the proposition above, $\Psi_i(\Lambda(b)), \Psi_i(\Lambda'(b))\in B_0^{(i)}\times\mathbb{Z}_{\leqq 0}$.

Now, for $b_1, b_2\in B_0^{(i)}$, we have $b_1=b_2$ if and only if $\langle \weight b_1, \alpha_i^{\vee}\rangle=\langle \weight b_2, \alpha_i^{\vee}\rangle$. Therefore, it follows that
\begin{align*}
\Lambda(b)=\Lambda'(b) &\Leftrightarrow \Psi_i(\Lambda(b))=\Psi_i(\Lambda'(b))\\
&\Leftrightarrow \langle \weight \tilde{b}, \alpha_i^{\vee}\rangle=\langle \weight \tilde{b}', \alpha_i^{\vee}\rangle\ \text{and}\ \varepsilon_i^{\ast}(\Lambda(b))=\varepsilon_i^{\ast}(\Lambda'(b))\\
&\Leftrightarrow \langle \weight \Lambda(b), \alpha_i^{\vee}\rangle=\langle \weight \Lambda'(b), \alpha_i^{\vee}\rangle\ \text{and}\ \varepsilon_i^{\ast}(\Lambda(b))=\varepsilon_i^{\ast}(\Lambda'(b)).
\end{align*}
\end{proof}
\begin{proof}[Proof of Theorem \ref{comparison}]
Let $G^{\mathrm{low}}(\tilde{b}^{(1)}).v_{w\lambda_0}=G^{\mathrm{low}}(b^{(1)}).v_{s_iw\lambda_0}$. Then, we have 
\begin{align}\label{epsilon_ast}
(\tilde{f}_i^{\ast})^{\langle w\lambda_0, \alpha_i^{\vee}\rangle}b^{(1)}=\tilde{b}^{(1)}.
\end{align}
Indeed, by $\varepsilon_i^{\ast}(b^{(1)})=0$ and Proposition \ref{EFaction'}, we have
$$G^{\mathrm{low}}(b^{(1)})F_i^{(\langle w\lambda_0, \alpha_i^{\vee}\rangle)}=G^{\mathrm{low}}((\tilde{f}_i^{\ast})^{\langle w\lambda_0, \alpha_i^{\vee}\rangle}b^{(1)})+\sum_{\substack{\tilde{b}\in B(\infty)\\ \langle w\lambda_0, \alpha_i^{\vee}\rangle<\varepsilon_i^{\ast}(\tilde{b})}}c_{\tilde{b}}G^{\mathrm{low}}(\tilde{b})$$
for some $c_{\tilde{b}}\in \mathbb{Z}[q^{\pm 1}]$.

We can calculate the degree $< l(w_0)L$ part of $p$ by the same method as in the proof of Lemma \ref{roughestlem} and Theorem \ref{maincalc}. (The element $b^{(0)}$ (resp.~$\tilde{b}^{(1)}$, resp.~${b'}^{(0)}$, resp.~${b'}^{(1)}$) corresponds to $b_+$ (resp.~$b_-$, resp.~$b'_+$, resp.~$b'_-$) in the proof of Theorem \ref{maincalc}.) Note that, by $\varepsilon_i^{\ast}(b^{(1)})=0$ and the equality (\ref{epsilon_ast}), we have $\varepsilon_i^{\ast}(\tilde{b}^{(1)})=\langle w\lambda_0, \alpha_i^{\vee}\rangle$, equivalently, $\varphi_i^{\lambda_1}(\pi_{\lambda_1}(\ast \tilde{b}^{(1)}))=\varphi_i^{\lambda_0}({b'}^{(1)})$, here $\lambda_1$ is as in the proof of Theorem \ref{maincalc}. (See the inequality (\ref{shiki2}).) 

Set $\hat{b}:=(\tilde{f}_i^{\ast})^{\varphi_i^{\lambda_1}(\pi_{\lambda_1}(\ast \tilde{b}^{(1)}))}(\tilde{b}^{(1)})=(\tilde{f}_i^{\ast})^{\varphi_i^{\lambda_0}({b'}^{(1)})}(\tilde{b}^{(1)})$. Then, we have $\varepsilon_i(\hat{b})=\varepsilon_i^{\lambda_0}({b'}^{(1)})$. Indeed, the equalities $G^{\mathrm{low}}(\ast \hat{b}).v_{\lambda_1}\neq 0$, $G^{\mathrm{low}}(\ast \tilde{f}_i^{\ast}\hat{b}).v_{\lambda_1}=0$ imply that 
\begin{itemize}
\item $G^{\mathrm{low}}(\ast \hat{b})\notin U_q(\mathfrak{n}^-)F_i^{(\varepsilon_i^{\lambda_0}({b'}^{(1)})+1)}+\sum_{j\in I\setminus \{i\}}U_q(\mathfrak{n}^-)F_j^{(\varepsilon_j+1)}$, and
\item $G^{\mathrm{low}}(\ast \tilde{f}_i^{\ast}\hat{b})\in U_q(\mathfrak{n}^-)F_i^{(\varepsilon_i^{\lambda_0}({b'}^{(1)})+1)}+\sum_{j\in I\setminus \{i\}}U_q(\mathfrak{n}^-)F_j^{(\varepsilon_j+1)}$.
\end{itemize}
We may assume that $\varepsilon_j$ $(j\in I\setminus \{i\})$ are sufficiently large. (Note that the definition of $\hat{b}$ does not depend on the choice of $\varepsilon_j$ $(j\in I\setminus \{i\})$.) Then, by Proposition \ref{c1=epsiloni1}, we have $\varepsilon_i(\hat{b})< \varepsilon_i^{\lambda_0}({b'}^{(1)})+1$ and $\varepsilon_i(\tilde{f}_i^{\ast}\hat{b})\geqq \varepsilon_i^{\lambda_0}({b'}^{(1)})+1$. On the other hand, it follows from the definition of the crystal $B(\infty)\times \mathbb{Z}$ and Proposition \ref{Kasemb} that $\varepsilon_i(\tilde{f}_i^{\ast}\hat{b})-\varepsilon_i(\hat{b})$ is equal to 0 or 1. Therefore, $\varepsilon_i(\tilde{f}_i^{\ast}\hat{b})=\varepsilon_i^{\lambda_0}({b'}^{(1)})+1$ and $\varepsilon_i(\hat{b})=\varepsilon_i^{\lambda_0}({b'}^{(1)})$.

Hence, it suffices to show that 
$$b^{(1)}=\Lambda_i^{-1}(\tilde{e}_i^{\varepsilon_i(\hat{b})}\hat{b}).$$
We abbreviate $\tilde{e}_i^{\varepsilon_i(\hat{b})}\hat{b}$ to $b$. 

Recall that $\varepsilon_i(\hat{b})=\varepsilon_i^{\lambda_0}({b'}^{(1)})$. We have 
$$b^{(1)}=(\tilde{e}_i^{\ast})^{\varphi_i^{\lambda_0}({b'}^{(1)})+\langle w\lambda_0, \alpha_i^{\vee}\rangle}\tilde{f}_i^{\varepsilon_i^{\lambda_0}({b'}^{(1)})}b.$$
Now, we have $\varepsilon_i^{\ast}(b^{(1)})(=0)=\varepsilon_i^{\ast}(\Lambda_i^{-1}(b))$. So, by Lemma \ref{coincide}, it only remains to prove that $\langle\weight b^{(1)}, \alpha_i^{\vee}\rangle=\langle\weight\Lambda_i^{-1}(b), \alpha_i^{\vee}\rangle$, that is,
\begin{align}\label{crystaleq}
-\varphi_i^{\ast}(b)+\varepsilon_i^{\ast}(b)=\varphi_i^{\lambda_0}({b'}^{(1)})+\langle w\lambda_0, \alpha_i^{\vee}\rangle-\varepsilon_i^{\lambda_0}({b'}^{(1)}),
\end{align}
by the definition of $\Lambda_i^{-1}$. (See Proposition \ref{Saito_refl}.)

By the property of crystals, the left-hand side of (\ref{crystaleq}) is equal to $-\langle \weight b, \alpha_i^{\vee}\rangle$. Moreover, we have $\weight b=\weight \hat{b}+\varepsilon_i^{\lambda_0}({b'}^{(1)})\alpha_i=\weight \tilde{b}^{(1)}-(\varphi_i^{\lambda_0}({b'}^{(1)})-\varepsilon_i^{\lambda_0}({b'}^{(1)}))\alpha_i$.
Hence, 
\begin{align*}
-\langle \weight b, \alpha_i^{\vee}\rangle&=-\langle \weight \tilde{b}^{(1)}, \alpha_i^{\vee}\rangle+2(\varphi_i^{\lambda_0}({b'}^{(1)})-\varepsilon_i^{\lambda_0}({b'}^{(1)}))\\
&=-\langle \weight \tilde{b}^{(1)}, \alpha_i^{\vee}\rangle+\langle \weight {b'}^{(1)}, \alpha_i^{\vee}\rangle+\varphi_i^{\lambda_0}({b'}^{(1)})-\varepsilon_i^{\lambda_0}({b'}^{(1)})\\
&=\langle w\lambda_0, \alpha_i^{\vee}\rangle+\varphi_i^{\lambda_0}({b'}^{(1)})-\varepsilon_i^{\lambda_0}({b'}^{(1)})\\
&=(\text{the\ right-hand\ side\ of\ }(\text{\ref{crystaleq}})).
\end{align*}
\end{proof}
As a corollary of this coincidence, we obtain the following results.
\begin{corollary}\label{converge_app}
For $x\in U_q(\mathfrak{n}^+), \lambda\in P_+, w\in W$, a reduced word ${\bf i}_w$ of $w$ and a reduced word ${\bf i}'_w$ of $w^{-1}w_0$, we write 
$$x=\sum_{\tilde{{\bf c}}\in (\mathbb{Z}_{\geqq 0})^{l(w_0)}}{_{{\bf i}_w {\bf i}'_w}\zeta}_{\tilde{{\bf c}}}^{x}E_{{\bf i}_w {\bf i}'_w}^{\tilde{{\bf c}}}\ \text{with\ } {_{{\bf i}_w{\bf i}'_w}\zeta}_{\tilde{{\bf c}}}^{x}\in \mathbb{Q}(q),\ \text{and}$$
$$(c_{f_{ww_0\lambda}, v_{w_0\lambda}}^{\lambda}.\ast(x)).|(0) \rangle_{{\bf i}_w}=\sum_{{\bf c}\in (\mathbb{Z}_{\geqq 0})^{l(w)}}{_{{\bf i}_w}\zeta}_{{\bf c}}^{ww_0\lambda, x}|({\bf c}) \rangle_{{\bf i}_w}\ \text{with\ } {_{{\bf i}_w}\zeta}_{{\bf c}}^{ww_0\lambda, x}\in \mathbb{Q}(q).$$
When $\lambda \in P_+$ tends to $\infty$ in the sense that $\langle \lambda, \alpha_i^{\vee}\rangle$ tends to $\infty$ for all $i\in I$, ${_{{\bf i}_w}\zeta}_{{\bf c}}^{ww_0\lambda, x}$ converges to ${_{{\bf i}_w {\bf i}'_w}\zeta}_{({\bf c}, 0,\dots, 0)}^{x}$ in the complete discrete valuation field $\mathbb{Q}((q))$.
\end{corollary}
\begin{proof}
Write ${\bf i}_w=(i_1,\dots, i_l)$ and ${\bf i}'_w=(i_{l+1},\dots, i_{l(w_0)})$. We claim that
\begin{center}
$c_{f_{ww_0\lambda}, v_{w_0\lambda}}^{\lambda}.\ast(T'_{i_1, 1}\cdots T'_{i_{k-1}, 1}(E_{i_k}))=0$ if $k>l$ for all $\lambda \in P_+$.
\end{center}
Suppose that the left-hand side is not equal to zero for some $k>l$ and $\lambda \in P_+$. Then, $f_{ww_0\lambda}.\ast(T'_{i_1, 1}\cdots T'_{i_{k-1}, 1}(E_{i_k}))$ is a weight vector whose weight is $ww_0\lambda-s_{i_1}\cdots s_{i_{k-1}}\alpha_{i_k}$. Since the weight set of $V(\lambda)^{\ast}$ｴis $W$-stable, $\lambda-w_0w^{-1}s_{i_1}\cdots s_{i_{k-1}}\alpha_{i_k}=\lambda-w_0s_{i_{l+1}}\cdots s_{i_{k-1}}\alpha_{i_k}>\lambda$ (in $P$) is also a weight of $V(\lambda)^{\ast}$. This is a contradiction.

Hence, we have
\begin{align*}
c_{f_{ww_0\lambda}, v_{w_0\lambda}}^{\lambda}.\ast(x)&=c_{f_{ww_0\lambda}, v_{w_0\lambda}}^{\lambda}.\ast(\sum_{{\bf c}\in (\mathbb{Z}_{\geqq 0})^{l(w)}}{_{{\bf i}_w {\bf i}'_w}\zeta}_{({\bf c}, 0,\dots, 0)}^{x}E_{{\bf i}_w {\bf i}'_w}^{({\bf c}, 0,\dots, 0)})\\
&=c_{((\star),\cdot)_{\lambda}, v_{w_0\lambda}}^{\lambda},
\end{align*}
where $(\star)=\sum_{{\bf c}\in (\mathbb{Z}_{\geqq 0})^{l(w)}}{_{{\bf i}_w {\bf i}'_w}\zeta}_{({\bf c}, 0,\dots, 0)}^{x}F_{{\bf i}_w {\bf i}'_w}^{({\bf c}, 0,\dots, 0)}.v_{ww_0\lambda}$.

Set $\mathcal{P}:=\Set{{\bf c} \in (\mathbb{Z}_{\geqq 0})^{l(w)}| {_{{\bf i}_w {\bf i}'_w}\zeta}_{({\bf c}, 0,\dots, 0)}^{x}\neq 0}$. 

It is well known that, when $\lambda\in P_+$ is sufficiently large in the sense that $\langle \lambda, \alpha_i^{\vee}\rangle \gg 0$ for all $i\in I$, $\{ F_{{\bf i}_w {\bf i}'_w}^{({\bf c}, 0,\dots, 0)}.v_{ww_0\lambda} \}_{{\bf c}\in \mathcal{P}}$ is a linearly independent set. (Use the symmetries $T'_{i, -1}$ $(i\in I)$ on the integrable $U_q(\mathfrak{g})$-module. See \cite[5.2.1, 37.1.2]{Lusbook}.)

By the way, there exist uniquely $b^1,\dots, b^t \in B(\infty)$ and $\eta_1,\dots, \eta_t \in \mathbb{Q}(q)\setminus \{ 0\}$ such that 
$$G^{\mathrm{low}}(b^s).v_{ww_0\lambda}\neq 0\ \text{for\ all\ }s,\ \text{and}$$
$$\sum_{{\bf c}\in \mathcal{P}}{_{{\bf i}_w {\bf i}'_w}\zeta}_{({\bf c}, 0,\dots, 0)}^{x}F_{{\bf i}_w {\bf i}'_w}^{({\bf c}, 0,\dots, 0)}.v_{ww_0\lambda}=\sum_{s=1}^t \eta_s G^{\mathrm{low}}(b^s).v_{ww_0\lambda}.$$
Again, when $\lambda\in P_+$ is sufficiently large, we have
$${_{{\bf i}_w {\bf i}'_w}\zeta}_{({\bf c}, 0,\dots, 0)}^{x}=\sum_{s=1}^t \eta_s {_{{\bf i}_w {\bf i}'_w}\zeta}_{({\bf c}, 0,\dots, 0)}^{G^{\mathrm{low}}(b^s)}\ \text{for\ all\ }{\bf c}\in \mathcal{P}.$$
Therefore, it remains to prove that ${_{{\bf i}_w}\zeta}_{{\bf c}}^{ww_0\lambda, \omega(G^{\mathrm{low}}(b^s))}$ converges to ${_{{\bf i}_w {\bf i}'_w}\zeta}_{({\bf c}, 0,\dots, 0)}^{G^{\mathrm{low}}(b^s)}$ in $\mathbb{Q}((q))$ when $\lambda\in P_+$ tends to $\infty$ for all ${\bf c}\in \mathcal{P}$ and $s$. 

\noindent(Note that $c_{(G^{\mathrm{low}}(b^s).v_{ww_0\lambda}, \cdot)_{\lambda}, v_{w_0\lambda}}^{\lambda}.|(0) \rangle_{{\bf i}_w}=\sum_{{\bf c}'}{_{{\bf i}_w}\zeta}_{{\bf c}'}^{ww_0\lambda, \omega(G^{\mathrm{low}}(b^s))}|({\bf c}') \rangle_{{\bf i}_w}$.)

This follows from the computation in Section \ref{main} and Theorem \ref{comparison}.
\end{proof}
The following corollary follows by the same method as in the proof of Corollary \ref{converge'}.
\begin{corollary}\label{converge'_app}
For $x \in U_q(\mathfrak{n}^+), \lambda\in P_+, w, \tilde{w}\in W$ with $\tilde{w}\leq_r w$, where $\leq_r$ is the weak right Bruhat order, a reduced word ${\bf i}_{\tilde{w}}$ (resp.~${\bf i}_{\tilde{w}^{-1}w}$) of $\tilde{w}$ (resp.~$\tilde{w}^{-1}w$) and a reduced word ${\bf i}'_w$ of $w^{-1}w_0$, we write
$$(c_{f_{ww_0\lambda}, v_{\tilde{w}^{-1}ww_0\lambda}}^{\lambda}.\ast(x)).|(0) \rangle_{{\bf i}_{\tilde{w}}}=\mspace{-5mu}\sum_{{\bf c}\in (\mathbb{Z}_{\geqq 0})^{l(\tilde{w})}}\mspace{-15mu}{_{{\bf i}_{\tilde{w}}}\zeta}_{{\bf c}}^{ww_0\lambda, x}|({\bf c}) \rangle_{{\bf i}_{\tilde{w}}}\ \text{with\ } {_{{\bf i}_{\tilde{w}}}\zeta}_{{\bf c}}^{ww_0\lambda, x}\in \mathbb{Q}(q).$$
When $\lambda\in P_+$ tends to $\infty$, ${_{{\bf i}_{\tilde{w}}}\zeta}_{{\bf c}}^{ww_0\lambda, x}$ converges to ${_{{\bf i}_{\tilde{w}}{\bf i}_{\tilde{w}^{-1}w}{\bf i}'_w}\zeta}_{({\bf c}, 0,\dots, 0, 0,\dots, 0)}^{x}$ in $\mathbb{Q}((q))$.
\end{corollary}
\def\cprime{$'$} \def\cprime{$'$} \def\cprime{$'$} \def\cprime{$'$}

\end{document}